\documentclass[11pt]{amsart}

\usepackage{amscd,amssymb,amsopn,amsmath,amsthm,mathrsfs,graphics,amsfonts,enumerate,verbatim,calc
}

\usepackage[OT2,OT1]{fontenc}
\newcommand\cyr{%
\renewcommand\rmdefault{wncyr}%
\renewcommand\sfdefault{wncyss}%
\renewcommand\encodingdefault{OT2}%
\normalfont
\selectfont}
\DeclareTextFontCommand{\textcyr}{\cyr} 

\usepackage{amssymb,amsmath}

\DeclareFontFamily{OT1}{rsfs}{}
\DeclareFontShape{OT1}{rsfs}{n}{it}{<-> rsfs10}{}
\DeclareMathAlphabet{\mathscr}{OT1}{rsfs}{n}{it}

\numberwithin{equation}{section}
\newtheorem{theorem}{Theorem}[section]
\newtheorem{lemma}[theorem]{Lemma}
\newtheorem{proposition}[theorem]{Proposition}
\newtheorem{corollary}[theorem]{Corollary}

\newtheorem{conjecture}[theorem]{Conjecture}
\newtheorem{maintheorema}{Main Theorem}

\theoremstyle{definition}
\newtheorem{definition}[theorem]{Definition}

\newtheorem{remark}[theorem]{Remark}
\theoremstyle{remark}

\newtheorem{example}[theorem]{Example}

\newcommand{\Ass}{\operatorname{Ass}}
\newcommand{\im}{\operatorname{Im}}
\renewcommand{\ker}{\operatorname{Ker}}

\newcommand{\Aut}{\operatorname{Aut}}
\newcommand{\Spec}{\operatorname{Spec}}

\newcommand{\Ht}{\operatorname{ht}}

\newcommand{\Fitt}{\operatorname{Fitt}}
\newcommand{\Supp}{\operatorname{Supp}}

\newcommand{\Hom}{\operatorname{Hom}}

\newcommand{\Ann}{\operatorname{Ann}}

\newcommand{\Char}{\operatorname{char}}

\newcommand{\depth}{\operatorname{depth}}

\newcommand{\coker}{\operatorname{Coker}}

\newcommand{\ord}{\operatorname{ord}}
\newcommand{\nord}{\operatorname{n.ord}}
\newcommand{\Tr}{\operatorname{Tr}}
\newcommand{\Frac}{\operatorname{Frac}}
\newcommand{\PD}{\operatorname{PD}}
\newcommand{\Cor}{\operatorname{Cor}}

\newcommand{\loc}{\operatorname{loc}}

\newcommand{\Sp}{\operatorname{Sp}}
\newcommand{\rc}{\operatorname{rc}}

\newcommand{\BK}{\operatorname{BK}}
\newcommand{\cyc}{\operatorname{cyc}}
\newcommand{\Eul}{\operatorname{Eul}}
\newcommand{\et}{\operatorname{et}}
\newcommand{\Betti}{\operatorname{Betti}}
\newcommand{\Bettip}{\operatorname{Betti,par}} 
\newcommand{\ES}{\operatorname{ES}} 
\newcommand{\End}{\operatorname{End}} 
\newcommand{\dR}{\operatorname{dR}} 
\newcommand{\Fil}{\operatorname{Fil}} 
\newcommand{\sgn}{\operatorname{sgn}}

\newcommand{\Min}{\operatorname{Min}}
\newcommand{\rank}{\operatorname{rank}}

\newcommand{\fm}{\frak{m}}
\newcommand{\fp}{\frak{p}}
\newcommand{\fq}{\frak{q}}
\newcommand{\fa}{\frak{a}}

\newcommand{\fn}{\frak{n}}

\newcommand{\Sel}{\operatorname{Sel}}
\newcommand{\Gal}{\operatorname{Gal}}
\newcommand{\Frob}{\operatorname{Frob}}

\begin{document}
\title[Specialization method in Krull dimension two]
{Specialization method in Krull dimension two and Euler system theory over normal deformation rings}

\author[T.Ochiai]{Tadashi Ochiai}
\address{Graduate school of Science
Osaka University 1-1 Machikaneyama Toyonaka 
Osaka 560-0043 Japan}
\email{ochiai@math.sci.osaka-u.ac.jp}

\author[K. Shimomoto]{Kazuma Shimomoto}
\address{Department of Mathematics College of Humanities and Sciences Nihon University Setagaya-ku Tokyo 156-8550 Japan}
\email{shimomotokazuma@gmail.com}



\begin{abstract} 
The aim of this article is to establish the specialization method on characteristic ideals for finitely generated torsion modules over a complete local normal domain $R$ that is module-finite over $\mathcal{O}[[x_1,\ldots,x_d]]$, where $\mathcal{O}$ is the ring of integers of a finite extension of the field of $p$-adic integers $\mathbb{Q}_p$. The specialization method is a technique that recovers the information on the characteristic ideal $\Char_R (M)$ from $\Char_{R/I}(M/IM)$, where $I$ varies in a certain family of nonzero principal ideals of $R$. As applications, we prove Euler system bound over Cohen-Macaulay normal domains by combining the main results in \cite{OcSh} and then we prove one of divisibilities of the Iwasawa main conjecture for two-variable Hida deformations generalizing the main theorem obtained in \cite{Oc2}.  
\end{abstract}

\maketitle
\tableofcontents

\section{Introduction}
In the study of arithmetic geometry over a Noetherian local ring $R$, we are often faced with the problem to recover some arithmetic invariant attached to a given module $M$ over $R$ from those attached to modules $M/IM$ over $R/I$, where $I$ runs through a certain set of principal ideals of $R$. We call it a \textit{specialization method} here. A typical work of this type appears in \textit{Iwaswa main conjecture for two-variable Hida deformations} of the first named author's articles \cite{Oc1} and \cite{Oc2}. In these papers, we considered the characteristic ideal of the Selmer group attached to the multi-variable Galois deformation with a continuous action of the Galois group $G_{\mathbb{Q}}=\Gal(\overline{\mathbb{Q}}/\mathbb{Q})$ over a certain Galois deformation ring $R$, under the assumption that $R \simeq \mathcal{O}[[x_1,\ldots,x_d]]$, where $\mathcal{O}$ is a complete discrete valuation ring (note that $R$ is of Krull dimension $d+1$). It was quite essential to assume the regularity of $R$ in \cite{Oc1}

In our previous work \cite{OcSh}, we proved the \textit{local Bertini theorem} for the normality on a Noetherian complete local ring $(R,\fm)$ with an extra assumption that $\depth R \ge 3$, which forces $\dim R \ge 3$.  In fact, it was already visible in the work \cite{Oc1} that for $\mathcal{O}[[x_1,\ldots,x_d]]$, the case $d \ge 2$ is quite different from the case $d=1$. The present article as well as \cite{OcSh} attempt to extend the results of \cite{Oc1} to the case that $R$ is a normal domain which is torsion free and finite over $\mathbb{Z}_p[[x_1,\ldots,x_d]]$. In the present article, we complete the missing case $\dim R=2$. With these results at hand, we establish Euler system bound over Cohen-Macaulay normal domains and one of the divisibilities in the Iwasawa main conjecture for two-variable Hida deformation without assuming that the branch of the Hecke algebra is isomorphic to the power series ring $\mathcal{O}[[X,Y]]$. The method of Euler system is quite powerful in giving an upper bound of a size of a Selmer group. In this article, it will be important to consider the $\Lambda$-adic (Hida-theoretic) version of Euler system whose specialization at an arithmetic point coincides with Beilinson-Kato Euler system constructed from the elements of the $K_2$-group of modular curves.

The structure of this paper is twofold. The first part is concerned with commutative algebra and the second part is concerned with applications of results from commutative algebra to Iwasawa theory.

\subsection{Main results in commutative algebra}
To establish the specialization method in Krull dimension two, we need the local Bertini theorem in this setting. However, if $R$ is a two-dimensional local normal domain and $I$ is a nonzero principal ideal, then $R/I$ is normal if and only if $R/I$ is regular. For this reason, it is too much to expect that $R/I$ is normal. Thus, we attempt to find a large set of specializations ideals $\{I_{\lambda}\}$ for which $R/I_{\lambda}$ are reduced rings of mixed characteristic. The local Bertini theorem is stated as follows (see Theorem \ref{MainBertini}, Corollary \ref{LocalBertini} and Corollary \ref{LocalBertini2}), which is of independent interest in commutative algebra (see Remark \ref{Bertinifail} for some relevant remarks)

\begin{maintheorema}
\label{mainresult1}
Assume that $B \hookrightarrow C$ is a module-finite extension such that $B$ is a Noetherian unique factorization domain, $C$ is a normal domain and $\Frac(B) \to \Frac(C)$ is a separable field extension. Let us define a subset $\mathcal{S}_{C/B}$ of $\Spec B$ as follows:
$$
\mathcal{S}_{C/B}=\{\fp \in \Spec B~|~\fp~\mbox{is a height-one prime in}~\Ass_B(\Omega_{C/B})\}.
$$
Then $\mathcal{S}_{C/B}$ is a finite set and $C/\fp C$ is a reduced ring for any height-one prime $\fp \in \Spec B$ such that $\fp \notin \mathcal{S}_{C/B}$.\\
\end{maintheorema}

We need the theorem it when $B$ is a complete regular local domain in \S~\ref{controltheorem}. Assume that $(R,\fm,\mathbb{F})$ is a two-dimensional complete local normal domain with mixed characteristic and finite residue field. We define certain large sets $\mathscr{L}_{R,W(\mathbb{F})}(z,\{a_n\}_{n \in \mathbb{N}})$ and $\mathscr{E}_{R,W(\mathbb{F})}(z,r,\{a_n\}_{n \in \mathbb{N}})$ (see (\ref{Fittideal}) and (\ref{Fittideal1}) in \S~\ref{controltheorem}, respectively), which are comprised of principal ideals generated by \textit{specialization elements} in $R$ attached to finitely generated torsion $R$-modules and these elements play a role in the control of the behavior of characteristic ideals under specialization. Let $\Char_R(M)$ denote the \textit{characteristic ideal} of a finitely generated torsion module $M$ over a Noetherian normal domain $R$ (see Definition \ref{charideal}). Characteristic ideals appear in the formulations of the Iwasawa main conjectures in various forms. By combining Main Theorem \ref{mainresult1}, we establish the following theorem (see Theorem \ref{prop1} and Theorem \ref{prop2} for precise statements).

\begin{maintheorema}
\label{mainresult2}
Let $(R,\fm,\mathbb{F})$ be a two-dimensional Noetherian complete local normal domain with mixed characteristic $p>0$ and finite residue field $\mathbb{F}$. Suppose that $M$ and $N$ are finitely generated torsion $R$-modules. 
\begin{enumerate}
\item
The following statements are equivalent:
\begin{enumerate}
\item
Let $\fq$ be any height-one prime of $R$ which does not lie over $p$. Then we have
$$
\Char_R(N)_{\fq} \subseteq \Char_R(M)_{\fq}.
$$ 
\item
There exists a constant $c \in \mathbb{N}$, depending on $M$ and $N$, such that
$$ 
c \cdot \frac{|N/\mathbf{x}N |}{|M/\mathbf{x}M |} \in \mathbb{N}
$$
for all but finitely many principal ideals
$$
(\mathbf{x}) \in \bigcup_{1\le i \le k} \mathscr{L}_{R,W(\mathbb{F})}(z_i,\{a_n\}_{n \in \mathbb{N}}).
$$
\end{enumerate}

\item
The following statements are equivalent:
\begin{enumerate}
\item
Let $\fp$ be any height-one prime of $R$ which lies over $p$. Then we have
$$
\Char_R(N)_{\fp} \subseteq \Char_R(M)_{\fp}.
$$
\item
There exists a constant $c \in \mathbb{N}$, depending on $M$ and $N$, such that 
$$
c \cdot \frac{|N/\mathbf{x}N |}{|M/\mathbf{x}M |} \in \mathbb{N}
$$
for all but finitely many principal ideals
$$
(\mathbf{x}) \in \bigcup_{1\le i \le h} \mathscr{E}_{R,W(\mathbb{F})}(z,r_i,\{a_n\}_{n \in \mathbb{N}}).
$$\\
\end{enumerate}
\end{enumerate}
\end{maintheorema}

For arithmetic applications, we will need to combine Main Theorem \ref{mainresult2} with a modified version of \cite[Theorem 8.8]{OcSh} in a practical form (see Theorem \ref{theorem:previous}).

\subsection{Arithmetic applications}
We study \textit{Euler system bound} associated with a $p$-adic Galois representation over a complete Noetherian local ring (see Definition \ref{definitionEuler} and the book \cite{Rub} for the classical Euler system theory over a discrete valuation ring). As remarked in the beginning of  the introduction, the method of Euler system is useful for finding an upper bound of the size of a Selmer group, hence for proving one of the predicted divisibilities of the Iwasawa main conjecture. First, we state Euler system bound for a $p$-adic Galois representation with coefficients in a complete Noetherian reduced local ring that is torsion free and finite over the ring of $p$-adic integers (see Theorem \ref{Euler}).

\begin{maintheorema}
\label{mainresult3}
Let $(R,\fm,\mathbb{F})$ be a complete Noetherian reduced local ring which is finite and flat over $\mathbb{Z}_p$ for a prime number $p>2$. Suppose that
$$
\big\{\mathbf{z}_n\in H^1(G_{\Sigma,n},T^*(1))\big\}_{n \in \mathfrak{N}}
$$ 
is an Euler system for $(T,R,\Sigma)$, $T$ is a free $R$-module of rank two and suppose that the following conditions hold:

\begin{enumerate}
\item[\rm{(i)}]
$T \otimes_R R/\fm$ is absolutely irreducible as a representation of $G_{\mathbb{Q}}$.

\item[\rm{(ii)}]
The quotient $H^1(G_{\Sigma},T^*(1))/R \mathbf{z}_1$ is a finite group.

\item[\rm{(iii)}]
$H^2(G_{\mathbb{Q}_{\ell}},T^*(1))=0$ for every $\ell \in \Sigma\setminus\{\infty\}$ and $H^2(G_{\Sigma},D)$ is a finite group.

\item[\rm{(iv)}]
The determinant character $\wedge^2 \rho:G_{\mathbb{Q}} \to R^{\times}$ $($resp. $\wedge^2 \rho^*(1):G_{\mathbb{Q}} \to R^{\times}$$)$ associated with the $G_{\mathbb{Q}}$-representation $T$ (resp. $T^*(1)$) has an element of infinite order.

\item[\rm{(v)}]
The $R$-module $T$ splits into eigenspaces: $T=T^+ \oplus T^-$ with respect to the complex conjugation in $G_{\mathbb{Q}}$, and $T^+_{\fp}$ (resp. $T^-_{\fp}$) is of $R_{\fp}$-rank one for each minimal prime $\fp$ of $R$.

\item[\rm{(vi)}]
There exist $\sigma_1 \in G_{\mathbb{Q}(\mu_{p^{\infty}})}$ and $\sigma_2 \in G_{\mathbb{Q}}$ such that $\rho(\sigma_1) \simeq \begin{pmatrix} 1 & \epsilon \\ 0 & 1 \end{pmatrix} \in GL_2(R)$ for a nonzero divisor $\epsilon \in R$ and $\sigma_2$ acts on $T$ as multiplication by $-1$. 

\end{enumerate}
Then Euler system bound holds for $(T,R,\Sigma)$:
\begin{enumerate}
\item
$\textcyr{Sh}^2_{\Sigma}(T^*(1))$ is a finite group.
\item
We have
$$
c \cdot |H^1(G_{\Sigma},T^*(1))/R \mathbf{z}_1|~\mbox{is divisible by}~| \textcyr{Sh}^2_{\Sigma}(T^*(1))|,
$$
where $c:=|R/(\epsilon^k)|$ is an error term for Euler system bound and $k$ is the number of minimal generators of the $R$-module $\textcyr{Sh}^2_{\Sigma}(T^*(1))$.\\
\end{enumerate}
\end{maintheorema}

Main Theorem \ref{mainresult3} has been known to hold when $R$ is a complete discrete valuation ring. We reduce the proof of the theorem to the case where $R$ is a product of complete discrete valuation rings after normalizing the reduced ring $R$ in its total ring of fractions. Euler system bound is less powerful, if the image of the Galois representation attached to an ordinary modular form is small. This phenomena occurs when the modular form under consideration has \textit{complex multiplication}. The paper \cite{OcPr} discusses the case of elliptic cusp forms with complex multiplication and \cite{HaOc15} generalized it to the case of Hilbert modular cusp forms with complex multiplication. 

Our primary concern, as in the hypothesis $\rm(iv)$ of Main Theorem \ref{mainresult3}, is when the Galois representation has big image and $R$ is a certain (Hecke) deformation ring. As we will see soon, the Main Theorem \ref{mainresult3} is not a mere generalization of the classical Euler system bound over a discrete valuation ring. The theorem is used to deduce Euler system bound for a $p$-adic Galois representation with coefficients in a complete Noetherian local domain of Krull dimension at least two. Indeed, we obtain, from Theorem \ref{theorem:previous} (to descend from Krull dimension at least three to two) and Main Theorem \ref{mainresult2}, the following Euler system bound over a Cohen-Macaulay normal domain (see Theorem \ref{veryfinal}):

\begin{maintheorema}
\label{mainresult4}
Let $(R,\fm,\mathbb{F})$ be a Noetherian complete local Cohen-Macaulay normal domain 
of Krull dimension $d\geq 2$ with mixed characteristic $p>2$ and finite residue field $\mathbb{F}$. Suppose that
$$
\big\{\mathbf{z}_n\in H^1(G_{\Sigma,n},T^*(1))\big\}_{n \in \mathfrak{N}}
$$ 
is an Euler system for $(T,R,\Sigma)$, $T$ is a free $R$-module of rank two with 
continuous $G_{\mathbb{Q}}$-action and suppose that the following conditions hold:

\begin{enumerate}
\item[\rm{(i)}]
$T \otimes_R R/\fm$ is absolutely irreducible as a representation of $G_{\mathbb{Q}}$.

\item[\rm{(ii)}]
The quotient $H^1(G_{\Sigma},T^*(1))/R \mathbf{z}_1$ is an $R$-torsion module.

\item[\rm{(iii)}]
$H^2(G_{\mathbb{Q}_{\ell}},T^*(1))=0$ for every $\ell \in \Sigma\setminus\{\infty\}$ and $H^2(G_{\Sigma},D)$ is a finite group.

\item[\rm{(iv)}]
The determinant character $\wedge^2 \rho:G_{\mathbb{Q}} \to R^{\times}$ (resp. $\wedge^2 \rho^*(1):G_{\mathbb{Q}} \to R^{\times}$) associated with the $G_{\mathbb{Q}}$-representation $T$ (resp. $T^*(1)$) has an element of infinite order.

\item[\rm{(v)}]
The $R$-module $T$ splits into eigenspaces: $T=T^+ \oplus T^-$ with respect to the complex conjugation in $G_{\mathbb{Q}}$, and $T^+$ (resp. $T^-$) is of $R$-rank one.

\item[\rm{(vi)}] 
There exist $\sigma_1 \in G_{\mathbb{Q}(\mu_{p^{\infty}})}$ and $\sigma_2 \in G_{\mathbb{Q}}$ such that $\rho(\sigma_1) \simeq \begin{pmatrix} 1 & P \\ 0 & 1 \end{pmatrix} \in GL_2(R)$ for a nonzero element $P \in R$ and $\sigma_2$ acts on $T$ as multiplication by $-1$. 
\end{enumerate}
Then Euler system bound holds for $(T,R,\Sigma)$:
\begin{enumerate}
\item
$\textcyr{Sh}^2_{\Sigma}(T^*(1))$ is a finitely generated torsion $R$-module.

\item
We have an inclusion of reflexive ideals:
$$ 
(P^k ) \Char_R\big(H^1 (G_{\Sigma},T^*(1))/R \mathbf{z}_1\big) \subseteq \Char_R\big(\textcyr{Sh}^2_{\Sigma} (T^*(1))\big),
$$ 
where $k$ is the number of minimal generators of $\textcyr{Sh}^2_{\Sigma}(T^*(1))$ as an $R$-module.\\
\end{enumerate}
\end{maintheorema}

The proof of Main Theorem \ref{mainresult4} is reduced to Main Theorem \ref{mainresult3} by an inductive argument with respect to the Krull dimension of the Cohen-Macaulay normal domain $R$. The most essential part of hypotheses in Main Theorem \ref{mainresult4} is:
\begin{enumerate}
\item[-]
$H^1(G_{\Sigma},T^*(1))/R \mathbf{z}_1$ is a torsion $R$-module.
\end{enumerate}
This is deeply related to the non-triviality of Euler system, which is conjectured to yield the existence of a $p$-adic $L$-function. We will apply the theorem in the case that $T=\mathcal{T}^{\nord}$ is a \textit{nearly ordinary Hida deformation space} attached to a Hida family $\mathbf{f}$. The hypotheses in Main Theorem \ref{mainresult4} are not so restrictive, except for the vanishing of $H^2(G_{\mathbb{Q}_{\ell}},T^*(1))$. At the moment, this seems necessary for the proof to work and it is related to the local automorphic representation of $GL_2(\mathbb{Q}_{\ell})$ spanned by an arithmetic specialization $\mathbf{f}_{\kappa}$ of $\mathbf{f}$. However, we believe that it suffices to assume that $H^2(G_{\mathbb{Q}_{\ell}},T^*(1))$ is an $R$-torsion module, although it is not clear at the moment. See \cite{Oc2} for the detailed study of the local monodromy of $\mathcal{T}^{\nord}$. Now we state the Iwasawa main conjecture for two-variable Hida deformations (see \S~\ref{Hidafamily} for notation and the conditions $(\bf{NOR})$, $(\bf{IRR})$ and $(\bf{FIL})$).

\begin{conjecture}
\label{MainConj}
Let $\mathbf{f}$ be a Hida family of $p$-ordinary $p$-stabilized cusp newforms of tame level $N$, and assume that $\mathcal{T}$ is the nearly ordinary Hida deformation associated with $\mathbf{f}$. Fix an $\mathbf{I}^{\nord}_{\mathbf{f}}$-basis $\mathbf{B}^{\pm}_{\mathbf{f}}$ of the modules of $\mathbf{I}^{\nord}_{\mathbf{f}}$-adic modular symbols $\mathbf{MS}^{\pm}_{\mathbf{f}}$. If the conditions $(\bf{NOR})$, $(\bf{IRR})$ and $(\bf{FIL})$ hold, then the Pontryagin dual of the Selmer group $\Sel_{\mathbb{Q}}(\mathcal{T})$ is a finitely generated, torsion $\mathbf{I}^{\nord}_\mathbf{f}$-module and we have
$$
\Char_{\mathbf{I}^{\nord}_{\mathbf{f}}}\big(\Sel_{\mathbb{Q}}(\mathcal{T})^{\PD}\big)=\big(L_p(\{\mathbf{B}^{\pm}_{\mathbf{f}}\})\big),
$$
where $L_p(\{\mathbf{B}^{\pm}_{\mathbf{f}}\})$ is a two-variable $p$-adic $L$-function attached to $\mathbf{B}^{\pm}_{\mathbf{f}}$.\\
\end{conjecture}

The two-variable $p$-adic $L$-function $L_p(\{\mathbf{B}^{\pm}_{\mathbf{f}}\})$ was constructed by Mazur and Kitagawa and we review its construction in this paper. In \cite{Hid0} and \cite{Hid1}, Hida showed that $\mathbf{I}^{\nord}_{\mathbf{f}}$ is a three-dimensional Noetherian complete local domain. This ring is called a \textit{branch} of the nearly ordinary Hecke algebra attached to $\mathbf{f}$ and we review its properties in \S~\ref{Hidafamily}. A partial answer to Conjecture \ref{MainConj} is given in \cite{Oc1} under the assumption that $\mathbf{I}^{\nord}_{\mathbf{f}} \simeq \mathcal{O}[[X,Y]]$. The Euler system bound as stated in Main Theorem \ref{mainresult4} is used to prove one of the divisibilities of Conjecture $\ref{MainConj}$ under the assumption that $R=\mathbf{I}^{\nord}_{\mathbf{f}}$ is a normal domain. If $\mathbf{I}^{\nord}_{\mathbf{f}}$ is not a regular ring, the quotient $\mathbf{I}^{\nord}_{\mathbf{f}}/\fa$ can only be a reduced ring for an ideal $\fa=(x,y)$ of height-two, as explained in the beginning of the introduction. This is the main reason why we need some technical results from commutative algebra, including authors' previous work \cite{OcSh}.  As already said, the method of Euler system provides the following inclusion relation: 
$$
(p\mbox{-adic}~L\mbox{-function}) \subseteq (\mbox{characteristic ideal of the Pontryagin dual of Selmer group})
$$
To achieve this, we will make an essential use of Beilinson-Kato Euler system constructed from elements of the $K_2$-group of modular curves. The following result is obtained as a partial answer to Conjecture \ref{MainConj} (see Theorem \ref{MainTh}).

\begin{corollary}
Assume that $(\bf{NOR})$, $(\bf{IRR})$ and $(\bf{FIL})$ hold for the nearly ordinary Hida deformation $\mathcal{T}$ attached to a Hida family $\mathbf{f}$. Fix an $\mathbf{I}^{\nord}_{\mathbf{f}}$-basis $\mathbf{B}^{\pm}_{\mathbf{f}}$ of the modules of $\mathbf{I}^{\nord}_{\mathbf{f}}$-adic modular symbols $\mathbf{B}^{\pm}_{\mathbf{f}}$. Assume further that

\begin{enumerate}
\item[\rm{(i)}]
There exists an element $\sigma_1 \in G_{\mathbb{Q}(\mu_{p^{\infty}})}$ such that $\rho_{\mathbf{f}}^{\nord}(\sigma_1) \simeq \begin{pmatrix} 1 & \epsilon \\ 0 & 1 \end{pmatrix} \in GL_2(\mathbf{I}^{\nord}_{\mathbf{f}})$ for a nonzero element $\epsilon \in \mathbf{I}^{\nord}_{\mathbf{f}}$.

\item[\rm{(ii)}]
There exists an element $\sigma_2 \in G_{\mathbb{Q}}$ such that $\sigma_2$ acts on $\mathcal{T}$ as multiplication by $-1$.

\item[\rm{(iii)}]
If $\ell$ is any prime dividing $Np$, the maximal Galois invariant quotient vanishes; $(\mathcal{T}^*)_{G_{\mathbb{Q}_{\ell}}}=0$.
\end{enumerate}
Let $k$ be the number of minimal $\mathbf{I}^{\nord}_{\mathbf{f}}$-generators of $\textcyr{Sh}^2_{\Sigma}(\mathcal{T}^*(1))$. Then we have
$$
\big(\epsilon^k L_p(\{\mathbf{B}^{\pm}_{\mathbf{f}}\})\big) \subseteq \Char_{\mathbf{I}^{\nord}_{\mathbf{f}}}\big(\Sel_{\mathbb{Q}}(\mathcal{T})^{\PD}\big).
$$\\
\end{corollary}

\subsection{Outline of the paper}

In \S~\ref{LocalBertiniTheorem}, we prove some preliminary lemmas to prove the local Bertini theorem for reduced quotients. Then we prove its corollary which is fundamental in the specialization methods.

In \S~\ref{controltheorem}, we establish specialization methods on characteristic ideals over complete local normal domains with mixed characteristic in dimension two. Together with the main results in \cite{OcSh}, we obtain all algebraic tools for the specialization methods.

In \S~\ref{GaloisCoho}, after reviewing global (Poitou-Tate) and local (Tate) dualities of Galois cohomology groups over number fields, we prove some preliminary results of Galois cohomology with a coefficient ring that is a module-finite extension over either $\mathbb{F}_p[[x_1,\ldots,x_d]]$ or $\mathbb{Z}_p[[x_1,\ldots,x_d]]$ based on the techniques developed in Greenberg's paper \cite{Gr2}.

In \S~\ref{EulerSystem}, we prove Euler system bound with coefficient in a reduced local ring that is finite torsion free over $\mathbb{Z}_p$ as Theorem \ref{Euler}. Then we establish Euler system bound with coefficient in a Cohen-Macaulay normal domain as Theorem \ref{veryfinal}, whose proof is based on Theorem \ref{Euler} and the specialization methods.

In \S~\ref{Hidafamily}, we review the theory of $p$-adic modular forms and Hida theory to an extent we need, including Galois representations attached to a $\Lambda$-adic family of ordinary cusp forms (called a \textit{Hida family}) with basic properties. Then we introduce a two-variable $p$-adic $L$-function defined over a branch of the nearly ordinary Hecke algebra. To this aim, we define $p$-optimal complex period and $p$-adic period after introducing the module of modular symbols. We use these periods to formulate the interpolation formula of the two-variable $p$-adic $L$-function.

In \S~\ref{GaloisDeformation}, we begin by recalling Greenberg's Selmer group associated to the nearly ordinary Hida deformation. After giving a summary on Beilinson-Kato Euler system, we explain its $\Lambda$-adic version and a construction of $p$-adic $L$-function via Coleman's map. The rest of this section is devoted to proving finiteness results on local Galois cohomology groups.

In \S~\ref{proof2}, we apply our main result to Beilinson-Kato Euler systems over the nearly ordinary Hida deformation and we give a partial result on Conjecture \ref{MainConj} (see Theorem \ref{MainTh} and Corollary \ref{MainCo}).\\

\section{Local Bertini theorem for reduced quotients}
\label{LocalBertiniTheorem}

\subsection{Preliminary lemmas}
In this section, we prove that a Noetherian complete local normal domain $R$ of Krull dimension $\ge 2$ admits infinitely many nonzero principal ideals $\{(\mathbf{x}_n)\}_{n \in \mathbb{N}}$ such that $R/(\mathbf{x}_n)$ are reduced rings. This is a fundamental tool for the specialization method. In \cite{OcSh}, the authors proved the local Bertini theorem with its application to characteristic ideals for Noetherian complete local normal domains in dimension $\ge 3$. Note that the quotient of a two-dimensional local normal domain by a nonzero element is not normal, if the ring itself is not regular. 

The nature of this section is pure commutative algebra, and we refer the reader to \cite{BrHer} and \cite{Mat} for relevant commutative algebra and to \cite{Kunz} for modules of K\"ahler differentials. We begin with some preliminary lemmas.

\begin{lemma}
\label{Koszul}
Let $f,g$ be a regular sequence in an integral domain $B$. Then, we have 
$$
B=B[\tfrac{1}{f}] \cap B[\tfrac{1}{g}].
$$
\end{lemma}

\begin{proof}
The inclusion $B \subseteq B[\frac{1}{f}] \cap B[\frac{1}{g}]$ is clear. We will prove the other inclusion 
$B[\frac{1}{f}] \cap B[\frac{1}{g}] \subseteq B$. 
Let us take $x \in B[\frac{1}{f}] \cap B[\frac{1}{g}]$. Let $m$ be the smallest non-negative integer 
such that $f^m x \in B$ and we put $a=f^m x$. We thus have $x =\frac{a}{f^m}$. We take another presentation 
$x =\frac{b}{g^n}$ with $b\in B$ and $n$ a positive integer (note that we do not necessarily assume the minimality for $n$). 
It suffices to prove that $m=0$. Assume $m>0$ to prove the lemma by contradiction. We have $bf^m=ag^n$.  
Since $f,g$ is a regular sequence and $m,n$ are positive integers, $f^m ,g^n$ is also a regular sequence.   
Hence, there exists an element $a' \in B$ such that $a=a'f$. We have
$x=\frac{a'}{f^{m-1}}$, which contradicts to the minimality of $m$. Hence $m=0$, showing that $x \in B$. 
This completes the proof of the inclusion $B[\frac{1}{f}] \cap B[\frac{1}{g}] \subseteq B$.
\end{proof}

\begin{lemma}
\label{Koszul2}
Let $B \hookrightarrow C$ be a ring extension of normal domains. Let $(y) \subseteq B$ be a principal prime ideal, let $B_{(y)}$ denote the localization of $B$ at the prime $yB$ and let $C_{(y)}:= C \otimes_B B_{(y)}$. Assume the following conditions: 
\begin{enumerate}
\item[\rm{(i)}] 
$C$ is integral over $B$. 

\item[\rm{(ii)}]
$C_{(y)}/(y)$ is a reduced ring.  
\end{enumerate}
Then $C/(y)$ is a reduced ring. 
\end{lemma}

\begin{proof}
Assume that $C/(y)$ is not reduced and deduce a contradiction. Then there exists $t \in C$ such that $t^N \in (y) \subseteq C$ and $t \notin (y) \subseteq C$ for some $N>0$. Since $C_{(y)}/(y)$ is reduced by assumption, we have $t \in (y) \subseteq C_{(y)}$ and we can write
$$
t=\tfrac{by}{x}~\mbox{for some}~x \in B \setminus (y)~\mbox{and}~b \in C.
$$
Since $t \in C$, we have $\frac{b}{x}=\frac{t}{y} \in C[\frac{1}{y}]$ and $\frac{b}{x} \in C[\frac{1}{x}]$. Since $B/(y)$ is an integral domain, $x$ is a regular element on $B/(y)$. Hence, $y,x$ is a regular sequence on $B$ and we have $B=B[\frac{1}{y}] \cap B[\frac{1}{x}]$ by Lemma \ref{Koszul}. On the other hand, we claim that
$$
B=B[\tfrac{1}{y}] \cap B[\tfrac{1}{x}] \to C[\tfrac{1}{y}] \cap C[\tfrac{1}{x}]
$$
is an integral extension. For this, let $\Frac(B)[X]$ denote the polynomial algebra over $\Frac(B)$ and let $F(X) \in \Frac(B)[X]$ be the monic minimal polynomial of an element $\alpha \in C[\frac{1}{y}] \cap C[\frac{1}{x}]$. Then $\alpha$ is integral over $B[\frac{1}{y}]$ and $B[\frac{1}{x}]$ which are both normal domains by normality of $B$. It follows from \cite[Theorem 2.1.17]{SwHu} that
$$
F(X) \in B[\tfrac{1}{y}][X]~\mbox{and}~F(X) \in B[\tfrac{1}{x}][X].
$$
Thus, it follows that $F(X) \in B[X]$ and $B \to C[\frac{1}{y}] \cap C[\frac{1}{x}]$ is an integral extension of integral domains. Since $C[\frac{1}{y}]$ and $C[\frac{1}{x}]$ are normal domains, $C[\frac{1}{y}] \cap C[\frac{1}{x}]$ is also normal. This implies that $C=C[\frac{1}{y}] \cap C[\frac{1}{x}]$ by normality of $C$. Finally,
$$
\tfrac{b}{x}\in C[\tfrac{1}{y}] \cap C[\tfrac{1}{x}]=C,
$$
which contradicts to our hypothesis. Hence we must have $t \in (y) \subseteq C$.\\
\end{proof}

\subsection{Local Bertini theorem}
Let $B \hookrightarrow C$ be a module-finite extension of Noetherian domains. Then it is known that the support of the $C$-module $\Omega_{C/B}$ is the largest closed subset of $\Spec C$ on which $\Spec C \to \Spec B$ is ramified \cite[Corollary 6.10]{Kunz}. We prove the following theorem as a main corollary of the preliminary lemmas. A ring map $B \to C$ is \textit{etale}, if it is flat, unramified and of finite presentation. Let $\Omega_{C/B}$ be the module of K\"ahler differentials and let $\Ass_B(\Omega_{C/B})$ be the set of associated primes of $\Omega_{C/B}$ as a $B$-module.

\begin{theorem}[Local Bertini Theorem]
\label{MainBertini}
Assume that $B \hookrightarrow C$ is a module-finite extension such that $B$ is a Noetherian unique factorization domain, $C$ is a normal domain and $\Frac(B) \to \Frac(C)$ is a separable field extension. Let us define a subset $\mathcal{S}_{C/B}$ of $\Spec B$ as follows: 
\begin{equation}
\label{equation:scb}
\mathcal{S}_{C/B}=\{\fp \in \Spec B~|~\fp~\mbox{is a height-one prime in}~\Ass_B(\Omega_{C/B})\}.
\end{equation}
Then $\mathcal{S}_{C/B}$ is a finite set and $C/\fp C$ is a reduced ring for any height-one prime $\fp \in \Spec B$ such that $\fp \notin \mathcal{S}_{C/B}$.
\end{theorem}

\begin{proof}
Since $\Frac(B) \to \Frac(C)$ is separable by assumption, we have
$$
\Omega_{C/B} \otimes_C \Frac(C) \simeq \Omega_{\Frac(C)/\Frac(B)}=0,
$$
which implies that $\Omega_{C/B}$ is a finitely generated torsion $C$-module. Since $C$ is a finitely generated $B$-module, $\Omega_{C/B}$ is a finitely generated torsion $B$-module. Let $\{ \fp_1,\ldots,\fp_k \}$ be the finite set of all associated primes of the finitely generated torsion $B$-module $\Omega_{C/B}$ which are of height-one. Then, for a height-one prime $\fp \in \Spec B$, we have $(\Omega_{C/B})_{\fp}=\Omega_{C_{\fp}/B_{\fp}} \ne 0$ 
if and only if $\fp \in \{\fp_1,\ldots,\fp_k\}$, where $C_{\fp}=C \otimes_B B_{\fp}$. In other words, if $\fp \notin \{\fp_1,\ldots,\fp_k\}$, the localization $B_{\fp} \to C_{\fp}$ 
is an etale extension of semi-local Dedekind domains. 
Note that every height-one prime of $B$ is principal because $B$ is a unique factorization domain. Now take an arbitrary height-one prime $\fp \notin \{\fp_1,\ldots,\fp_k\}$ of $B$ and 
we will apply Lemma \ref{Koszul2} with $(y) =\mathfrak{p}$. 
Since $B \hookrightarrow C$ is a module-finite extension, $C$ is integral over $B$, which is nothing but the condition (i) of Lemma \ref{Koszul2}. The condition (ii) of Lemma \ref{Koszul2} was already verified above. We thus conclude that $C/\fp C$ is a reduced ring by Lemma \ref{Koszul2}.\\
\end{proof}

\subsection{Cohen structure theorem}
We discuss applications of Theorem \ref{MainBertini} to Betini-type theorems on complete local normal domains, which we will apply to study the characteristic ideals. Now let $(R,\fm,\mathbf{k})$ be a Noetherian complete local domain of mixed characteristic, where $\fm$ is its unique maximal ideal and $\mathbf{k}$ is its residue field of characteristic $p>0$. Let $(A,pA,\mathbf{k})$ be a complete discrete valuation ring of mixed characteristic $p>0$. We fix a \textit{coefficient ring} of $R$:
$$
(A,pA,\mathbf{k}) \xrightarrow{\pi} R.
$$
That is, $\pi$ is a flat local map from $A$ to $R$ inducing an isomorphism on residue fields. Let $I \subseteq \fm$ be an ideal of a Noetherian local ring $(R,\fm,\mathbf{k})$. We say that $I$ is a \textit{parameter ideal}, if it is generated by a system of parameters of $R$. Let us recall that Cohen structure theorem is essential in the construction of a module-finite extension of a complete local ring over a complete regular local ring via a certain choice of a system of parameters \cite{Cohen}.

\begin{theorem}[Cohen structure theorem]
\label{Cohen}
A Noetherian complete local domain $(R,\fm,\mathbf{k})$ with mixed characteristic $p>0$ and dimension $d$ has a coefficient ring $(A,pA,\mathbf{k})$ and there exists a regular local subring $A[[x_2,\ldots,x_d]]$ of $R$ such that $A[[x_2,\ldots,x_d]] \hookrightarrow R$ is module-finite.
\end{theorem}

The difficult part of Cohen structure theorem is to prove that a coefficient ring exists. Once a coefficient ring is determined, we use the following lemma to find a module-finite extension from a complete regular local ring.

\begin{lemma}
\label{CohenStructure}
Let the notation be as in Theorem \ref{Cohen}, let $p,z_2,\ldots,z_d$ be any system of parameters of $R$ and let $A[[z_2,\ldots,z_d]]$ be the completion of the subring $A[z_2,\ldots,z_d]$ of $R$ with respect to the $(p,z_2,\ldots,z_d)$-adic topology. Then there is a module-finite extension $A[[z_2,\ldots,z_d]] \to R$. Moreover, $A[[z_2,\ldots,z_d]]$ is a $d$-dimensional regular local ring.
\end{lemma}

\begin{proof}
Taking the $(p,z_2,\ldots,z_d)$-adic completion, $A[z_2,\ldots,z_d] \hookrightarrow R$ extends to a ring map:
$$
A[[z_2,\ldots,z_d]] \to R.
$$
Let us prove that $R$ is a finitely generated $A[[z_2,\ldots,z_d]]$-module. Since $A[[z_2,\ldots,z_d]]$ and $R$ have the same residue field, $R/(p,z_2,\ldots,z_d)$ is an $A[[z_2,\ldots,z_d]]$-module of finite length. Applying the topological Nakayama's theorem \cite[Theorem 8.4]{Mat}, it follows that $A[[z_2,\ldots,z_d]] \to R$ is module-finite. Since the Krull dimension of $A[[z_2,\ldots,z_d]]$ is $d$ and its maximal ideal is generated by $d$ elements, it is a regular local ring.\\
\end{proof}

\subsection{An application of the local Bertini theorem}
We will need the following corollary of Theorem \ref{MainBertini} in \S~\ref{controltheorem}. Let us fix a prime number $p>0$.

\begin{corollary}
\label{LocalBertini}
Let $(R,\fm,\mathbf{k})$ be a two-dimensional Noetherian complete local normal domain with mixed characteristic $p>0$, where $A$ is a coefficient ring of $R$. Then the following statements hold.

\begin{enumerate}
\item 
Let us choose an element $z \in \fm$ such that $(p,z)$ is a parameter ideal of $R$. Then, for each $n>0$, $R/(z+a_np^n)$ is a reduced ring with mixed characteristic for all but finitely many $a_n \in A^{\times}$.

\item
Let us choose an element $z, r \in \fm$ such that $(p,z^n+r)$ is a parameter ideal of $R$ for all $n>0$.  
Then, for each $n>0$, $R/(z^n+r+a_n p)$ is a reduced ring with mixed characteristic for all but finitely many $a_n \in A^{\times}$.
\end{enumerate}
\end{corollary}

By a theorem of Auslander and Buchsbaum, it is known that a regular local ring is a unique factorization domain. Now we apply Theorem \ref{MainBertini} by setting $B=A[[z]]$ or $B=A[[z^n+r]]$ to prove the corollary.

\begin{proof}
The proofs of the assertions $\rm (1)$ and $\rm (2)$ proceed almost in the same way. The situation 
of $\rm (2)$ is slightly more complicated since the ring $B=A[[z^n+r]]$ is different for varying $n$. Thus, we will prove only the assertion $\rm (2)$ omitting the proof of the assertion $\rm (1)$.

Let us start the proof of the assertion $\rm (2)$. Since $(p,z^n+r)$ is a parameter ideal, $A[[z^n+r]] \hookrightarrow R$ is module-finite. Note that this map depends on the choice of $n>0$. For each $n>0$, we denote by $\mathcal{S}_n$ the finite subset $\mathcal{S}_{R/ A[[z^n+r]]}$ of $\Spec A[[z^n+r]]$ 
defined in \eqref{equation:scb}.  
Let $\mathcal{T}_n \subseteq \Spec R$ be a finite set defined to be the inverse image of $\mathcal{S}_n$ via the map 
$\Spec R \rightarrow \Spec A[[z^n+r]]$. 
When we have two distinct elements $a_n , a'_n \in A^{\times}$, 
$(z^n+r+a_np)$ and $(z^n+r+a'_np)$ are pairwise distinct primes of $A[[z^n+r]]$. 
Hence we have $
(z^n+r+a_np) \notin \mathcal{T}_n $ for all but finitely many $a_n \in A^{\times}$. 
If we have $(z^n+r+a_np) \notin \mathcal{T}_n $, then $R/(z^n+r+a_np)$ is a reduced ring of mixed characteristic. Thus, we complete the proof of $\rm(2)$.\\
\end{proof}

\begin{remark}
\label{heightoneprime}
An inspection of the proof of Lemma \ref{LocalBertini} reveals that a fixed height-one prime ideal of $R$ cannot contain infinitely many distinct principal ideals from $\{(z+a_np^n)\}_{n \in \mathbb{N}}$ (resp. $\{(z^n+r+a_np)\}_{n \in \mathbb{N}}$), as long as $z \in \fm$ is chosen to be nonzero. This fact will be used in the proof of Theorem \ref{veryfinal}.
\end{remark}

\subsection{An application of Cohen-Gabber theorem}

Let us recall that Cohen-Gabber theorem is a refinement of Cohen structure theorem in the equal characteristic $p>0$ case. We prove the following corollary as a straightforward consequence of Theorem \ref{MainBertini} for independent interest.

\begin{corollary}
\label{LocalBertini2}
Suppose that $(R,\fm,\mathbf{k})$ is a Noetherian complete local normal domain of dimension $d \ge 2$ and of equal characteristic. Let $\fq_1,\ldots,\fq_r$ be pairwise distinct height-one primes. Then there exists an infinite set of nonzero principal ideals $\{(\mathbf{x}_n)\}_{n \in \mathbb{N}}$ such that $R/(\mathbf{x}_n)$ are reduced rings and every $\mathbf{x}_n$ is not contained in $\fq_1 \cup \cdots \cup \fq_r$.
\end{corollary}

\begin{proof}
By assumption, $R$ contains the field $\mathbf{k}$.
When $\mathbf{k}$ is of characteristic $p>0$, we need to use Cohen-Gabber theorem. By \cite[Th\'eor\`em 2.1.1]{Book} (see \cite{KS} for an elementary proof) , there exist a regular local ring $\mathbf{k}[[t_1,\ldots,t_d]]$ and a module-finite extension: $\mathbf{k}[[t_1,\ldots,t_d]] \to R$ such that
$$
\Frac(\mathbf{k}[[t_1,\ldots,t_d]]) \to \Frac(R)
$$
is a separable extension, which is automatic in the characteristic zero case. Let us put
$$
\fp_i:=\mathbf{k}[[t_1,\ldots,t_d]] \cap \fq_i.
$$
Then $\fp_i$ is a height-one prime ideal. We denote by $\mathcal{S}$ the set of height-one primes of $\mathbf{k}[[t_1,\ldots,t_d]]$ attached by Theorem \ref{MainBertini} applied to the map $B:=\mathbf{k}[[t_1,\ldots,t_d]] \to C:=R$. Now we can find pairwise distinct elements $\mathbf{x}_1,\mathbf{x}_2,\ldots \in \mathbf{k}[[t_1,\ldots,t_d]]$ inductively so that $(\mathbf{x}_n) \notin \mathcal{S} \cup \{\fp_1,\ldots,\fp_r\}$ for $n>0$. Hence $R/(\mathbf{x}_n)$ is a reduced ring.
\end{proof}

\begin{remark}
\label{Bertinifail}
\begin{enumerate}
\item
Theorem \ref{MainBertini} fails without the assumption that $\Frac(B) \to \Frac(C)$ is separable. For example, let $C=\mathbf{k}[x,y]$ for a perfect field $\mathbf{k}$ of characteristic $p>0$. Let $B=\mathbf{k}[x^p,y^p]$. Then $B \to C$ is a purely inseparable extension. Set $\mathbf{x}=ax^p+by^p$ with $a,b \in \mathbf{k}^{\times}$. Then $C/(\mathbf{x})$ is not reduced for any choice $a,b \in \mathbf{k}^{\times}$.

\item
Corollary \ref{LocalBertini} can also be proved by using the basic elements and the derivation of completed K\"ahler differential modules, which are the main tools for the proof of the Bertini theorem for normality on local rings in the article \cite{OcSh}.

\item
In \S~\ref{controltheorem}, we will make a special choice of $r \in R$ in Corollary \ref{LocalBertini}, which is necessary to take care of height-one primes lying above $p$ in a normal domain.
\end{enumerate}
\end{remark}

Finally let us prove a lemma on when an ind-etale extension of a coefficient ring $A$ of a complete local ring $R$ preserves many good properties of $R$.

\begin{lemma}
\label{localring}
Let $(R,\fm,\mathbf{k})$ be a Noetherian complete local ring with mixed characteristic and residue field $\mathbf{k}$ and let $(A,pA,\mathbf{k}) \to R$ be a coefficient ring map. Assume that $(A,pA,\mathbf{k}) \to (B,pB,\mathbf{k}')$ is an integral flat extension of discrete valuation rings. We put $R_B:=R \widehat{\otimes}_A B$, where $\widehat{\otimes}$ is the complete tensor product. Then we have the following statements. 
\begin{enumerate}
\item 
$R_B$ is a Noetherian complete local ring and $R \otimes_A B$ is a quasilocal ring.
\item 
If we assume further that $R$ is a normal domain and $A \to B$ is an ind-etale extension, then $R_B$ is a normal domain.
\end{enumerate}
\end{lemma}

\begin{proof}
Let us prove the assertion $\rm (1)$. There is a surjetion $A[[t_1,\ldots,t_n]] \twoheadrightarrow R$ by Cohen structure theorem. Then we have a surjection
$B[[t_1,\ldots,t_n]] \twoheadrightarrow R_B$, which shows that $R_B$ is a Noetherian complete local ring. We prove that $R \otimes_A B$ is local. Note that $p$ is in the Jacobson radical of $R \otimes_A B$ and so it suffices to see that $(R \otimes_A B)/(p)=R/(p) \otimes_{\mathbf{k}} \mathbf{k}'$ is a local ring. Since $R/(p)$ is a local ring with the residue field $\mathbf{k}$ and $\mathbf{k} \to \mathbf{k}'$ is an algebraic extension, it is not hard to see 
that $R/(p) \otimes_{\mathbf{k}} \mathbf{k}'$ is a local ring (see \cite[Theorem]{Swe} for example). This completes the proof of $\rm (1)$. 

For the proof of the assertion $\rm (2)$, notice that $R \to R \otimes_A B$ is an ind-etale local extension and $R \otimes_A B$ is quasilocal by the assumption that $A \to B$ is ind-etale and $\rm(1)$. In particular, we have $R \to R \otimes_A B \to R^{\rm{sh}}$, where $R^{\rm{sh}}$ is the strict henselization of $R$. Since it is known that $R^{\rm{sh}}$ is Noetherian and $R \otimes_A B \to R^{\rm{sh}}$ is faithfully flat, $R \otimes_A B$ is also Noetherian. Since $R \to R \otimes_A B$ is ind-etale and $R$ is a complete local normal domain, we find that $R\otimes_A B$ is normal and excellent by \cite[Theorem 5.3 (iv)]{Gre}. Then the $\fm$-adic completion of $R \otimes_A B$, which is $R_B$, is a normal domain.\\
\end{proof}

\section{Characteristic ideals and specialization methods}
\label{controltheorem}

In this section, we develop the specialization method for characteristic ideals over Noetherian local normal domains in dimension two. In classical Iwasawa theory, one considers the characteristic ideal attached to the torsion module arising as the projective limit of the $p$-part of the ideal class groups of a certain tower of number fields, which gives rise to the completed group algebra $\mathbb{Z}_p[[1+p\mathbb{Z}_p]]$. In Hida theory, one considers the projective limit of the space of $p$-adic ordinary modular forms and the $p$-adic ordinary Hecke algebras. Thus, the resulting projective limit would be more complicated than $\mathbb{Z}_p[[1+p\mathbb{Z}_p]]$. This motivates us to consider the structure of characteristic ideals attached to finitely generated torsion modules over normal domains. We first make the definition of characteristic ideals. Let $R$ be a Noetherian normal domain and let
$$
N^{\rc}:=\Hom_R(\Hom_R(N,R),R)
$$
denote the \textit{reflexive closure} of an $R$-module $N$ and let $\ell_R(M)$ denote the length of an Artinian $R$-module $M$.

\begin{definition}
\label{charideal}
Let $R$ be a Noetherian normal domain.
\begin{enumerate}
\item
Let $M$ be a finitely generated torsion module over $R$. The \textit{characteristic ideal} associated to $M$ is defined as the reflexive ideal of $R$ as follows:
$$
\Char_R(M):=\big(\prod_{\fp} \fp^{\ell_{R_{\fp}}(M_{\fp})}\big)^{\rc},
$$
where $\fp$ ranges over all height-one prime ideals of $R$. If $M$ is not a torsion $R$-module, we set $\Char_R(M)=0$.

\item
A finitely generated $R$-module $M$ is said to be \textit{pseudo-null}, if $M_{\fp}=0$ for all height-one primes $\fp$ of $R$.

\item
Let $M$ and $N$ be finitely generated $R$-modules. Then an $R$-module map $f:M \to N$ is said to be a \textit{pseudo-isomorphism} if both kernel and cokernel of $f$ are pseudo-null $R$-modules.
\end{enumerate}
\end{definition}

More details on characteristic ideals are found in \cite[\S~7]{OcSh}. If $R$ is a unique factorization domain (UFD), the characteristic ideal is principal and it is not necessary to take the reflexive closure in Definition \ref{charideal}. The following proposition is a collection of the basic properties \cite[Proposition 9.6]{OcSh}.

\begin{proposition}
The following statements hold.

\begin{enumerate}
\item 
Let $R$ be a finitely generated torsion $R$-module. Then we have
$$
\Char_R(M)=\big(\prod_{\Ht \fp=1} \fp^{\ell_{R_{\fp}}(M_{\fp})}\big)^{\rc}=\Fitt_R(M)^{\rc},
$$
where $\Fitt_R(M)$ is the Fitting ideal of $M$. In particular, $\Fitt_R(M) \subseteq \Char_R(M)$. If $R$ is a UFD, then
$$
\Fitt_R(M) \subseteq \prod_{\Ht \fp=1} \fp^{\ell_{R_{\fp}}(M_{\fp})}=\Char_R(M).
$$

\item
Let $0 \to L \to M \to N \to 0$ be a short exact sequence of finitely generated $R$-modules. Then
$$
\Char_R(M)=\big(\Char_R(L) \Char_R(N)\big)^{\rc}.
$$

\item
Let $M$ and $N$ be finitely generated torsion $R$-modules. Then $\Char_R(M) \subseteq \Char_R(N)$ if and only if $\Char_R(M)_{\fp} \subseteq \Char_R(N)_{\fp}$ for all height-one primes $\fp \in \Spec R$.

\end{enumerate}
\end{proposition}

In what follows, we set
$$
M[x]:=\{m \in M~|~xm=0\}
$$
for a module $M$ over a ring $R$ and $x \in R$. This is an $R$-submodule of $M$.

\begin{lemma}
\label{lemma2}
Let $0 \to L \to M \to N \to 0$ be a short exact sequence of modules over a ring $R$. Then for $x \in R$, there is a short
exact sequence:
$$
0 \to L[x] \to M[x] \to N[x] \to L/xL \to M/xM \to N/xN \to 0.
$$
\end{lemma}

\begin{proof}
The lemma follows immediately by applying the snake lemma to the following diagram:
$$
\begin{CD}
0 @>>> L @>>> M @>>> N @>>> 0\\
@. @V\times xVV @V\times xVV @V\times xVV \\
0 @>>> L @>>> M @>>> N @>>> 0.\\
\end{CD}
$$
\end{proof}

For a finitely generated module $M$ over a domain $R$, the \textit{rank} of $M$ is defined as the $\Frac(R)$-dimension of the vector space $M \otimes_R \Frac(R)$. We have the following lemma:

\begin{lemma}
\label{lemma3}
Assume that $(R,\fm)$ is a one-dimensional local Noetherian domain and $M$ is a finitely generated $R$-module. Then, for a nonzero element $x \in \fm$, we have
$$
\ell_R(M/xM)-\ell_R(M[x])=\rank_R(M) \cdot \ell_R(R/(x)).
$$
\end{lemma}

\begin{proof}
Since the only prime ideals of $R$ are $0$ and $\fm$, there exists a filtration $0=M_0 \subseteq M_1 \subseteq \cdots \subseteq M_n=M$ such that each quotient $M_i/M_{i+1}$ is isomorphic to $R$ or $R/\fm$ by \cite[Theorem 6.4]{Mat}. By Lemma \ref{lemma2}, the left-hand side of the displayed equation is additive with respect to short exact sequences. 
Since $\ell_R(R/(x))$ is a constant which has nothing to do with $M$, 
the right-hand side of the displayed equation is additive with respect to short exact sequences. Hence, it suffices to consider the case that $M=R$ or $M=R/\fm$. If $M=R$, then both sides are equal to $\ell_R(R/(x))$. If $M=R/\fm$, then both sides are obviously zero, as desired.
\end{proof}

\begin{lemma}
\label{zerodivisor}
If $x,y$ are nonzero divisors of a commutative ring $R$, then there is a short exact sequence:
$$
0 \to R/(x) \to R/(xy) \to R/(y) \to 0.
$$
\end{lemma}

\begin{proof}
This is an easy exercise.
\end{proof}

\begin{definition}
Let $R[[X]]$ be a power series ring over a Noetherian local ring $(R,\fm,\mathbf{k})$. Then we say that $f \in R[[X]]$ is a \textit{distinguished polynomial}, if it is of the form:
$$
f =X^n+a_{n-1}X^{n-1}+\cdots+a_1X+a_0
$$
for some integer $n \ge 0$ and $a_0,\ldots,a_{n-1} \in \fm$.
\end{definition}

The following theorem is a classical result of much importance (see \cite[Chap. 7, \S 3.8]{Bourbaki} or \cite[Theorem 5.3.4]{NSW} and see also \cite{Ger} for a short proof)

\begin{theorem}[Weierstrass Preparation Theorem]
\label{Weierstrass}
Let $(R,\fm,\mathbf{k})$ be a complete Noetherian local ring and let 
$f=\sum_{i=0}^{\infty}a_i X^i \in R[[X]]$ be a nonzero element. Assume that there exists $n \in \mathbb{N}$ such that $a_i \in \fm$ for all $i < n$ and $a_n \notin \fm$. Then there exist a unit $u$ in $R[[X]]$ and a distinguished polynomial $f_0$ of degree $n$ such that $f=uf_0$. Furthermore, $u$ and $f_0$ are uniquely determined by $f$.\\
\end{theorem}

\subsection{Specialization method in dimension two for height-one primes outside $p$}
Let $(R,\fm,\mathbb{F})$ be a two-dimensional Noetherian complete local normal domain with mixed characteristic $p>0$ and finite residue field $\mathbb{F}$. Let $W(\mathbb{F})$ be the ring of Witt vectors of a finite field $\mathbb{F}$. Since $\mathbb{F}$ is a perfect field, there is a unique ring map $W(\mathbb{F}) \hookrightarrow R$ which induces an identity map on the residue field $\mathbb{F}$. Note that for any local ring $A$ with finite residue field $\mathbb{F}$ and a finitely generated Artinian $A$-module $M$, we have 
\begin{equation}
\label{cardinality}
\big|M\big|=\big|\mathbb{F}\big|^{\ell_A(M)},
\end{equation}
as verified easily by the filtration argument. Let us choose a parameter ideal $(p,z)$ of $R$ and define the set of principal ideals:
\begin{equation}
\label{Fittideal}
\mathscr{L}_{R,W(\mathbb{F})}(z,\{a_n\}_{n \in \mathbb{N}}):=\{(\mathbf{x}_n)\}_{n \in \mathbb{N}},
\end{equation}
where $\mathbf{x}_n=z+a_n p^n$ is attached by Corollary \ref{LocalBertini} $\rm (1)$ and $\{a_n\}_{n \in \mathbb{N}}$ is any sequence of elements in $W(\mathbb{F})^{\times}$. Suppose that $M$ and $N$ are finitely generated torsion $R$-modules. Let $\{\fq_i\}_{1 \le i \le k}$ be the set of all height-one primes outside $(p)$ and that appear as components in the characteristic ideal of either $M$ or $N$. Consider an additional condition:
\begin{equation}
\label{linearideal}
z_t \in \fq_t~\mbox{and}~z_t \notin \bigcup_{i \ne t} \fq_i~\mbox{for each}~1 \le t \le k.~\mbox{If}~k=1,~\mbox{then choose}~z \in \fq_1\setminus \{0\}.
\end{equation}
We notice that the existence of $z_t \in R$ satisfying $(\ref{linearideal})$ is ensured by Prime Avoidance Lemma \cite[Theorem A.1.1 at page 392]{SwHu}.

\begin{theorem}[Specialization Method outside $p$]\label{prop1} 
Let $(R,\fm,\mathbb{F})$ be a two-dimensional Noetherian complete local normal domain with mixed characteristic $p>0$ and finite residue field $\mathbb{F}$. Suppose that $M$ and $N$ are finitely generated torsion $R$-modules and fix a finite set of elements $\{z_i\}_{1 \le i \le k}$ satisfying $(\ref{linearideal})$. Then the following statements are equivalent:
 
\begin{enumerate}
\item[(a)]
Let $\fq$ be any height-one prime of $R$ which does not lie over $p$. Then we have
$$
\Char_R(N)_{\fq} \subseteq \Char_R(M)_{\fq}.
$$ 

\item[(b)]
There exists a constant $c \in \mathbb{N}$, depending on $M$ and $N$, such that
$$ 
c \cdot \frac{|N/\mathbf{x}N |}{|M/\mathbf{x}M |} \in \mathbb{N}
$$
for all but finitely many principal ideals
$$
(\mathbf{x}) \in \bigcup_{1\le i \le k} \mathscr{L}_{R,W(\mathbb{F})}(z_i,\{a_n\}_{n \in \mathbb{N}}).
$$
\end{enumerate}
\end{theorem}

\begin{proof}
Take the fundamental pseudo-isomorphisms
\begin{equation}
\label{equation : fundamental_isomorphism}
f_M:M \to \bigoplus_i R/\fp_i^{e_i} \oplus \bigoplus_r R/\fq_r^{f_r}~\mbox{and}~f_N:N \to \bigoplus_j R/\fp_j^{e'_j}\oplus \bigoplus_s R/\fq_s^{f'_s},
\end{equation}
where $\{\fp_i\}_{1 \le i \le h}$ is the set of all height-one primes of $R$ containing $p$, and $\{\fq_i\}_{1 \le i \le k}$ is the set of height-one primes not containing $p$. We call the target of \eqref{equation : fundamental_isomorphism} a \textit{fundamental module} associated to $M$ (resp. $N$) and denote it by $F(M)$ (resp. $F(N)$). Then kernel and cokernel of $f_M$ (resp. $f_N$) are finite, since $R$ is of Krull dimension two and $f_M$ (resp. $f_N$) is a pseudo-isomorphism. 

By the same argument as \cite[Proposition 3.11]{Oc1}, the proof of $\rm (a) \iff (b)$ for a pair of modules $(M,N)$ is reduced to the proof of $\rm(a) \iff (b)$ for a pair of modules $(F(M),F(N))$. We are thus reduced to the case where $M$ and $N$ are fundamental modules. That is, we may assume that
$$
M=R/J_M  \oplus R/I_M ~\mbox{and}~N=R/J_N  \oplus R/I_N,
$$
where $J_M=\prod_i \fp_i^{e_i}$, $I_M=\prod_r \fq_r^{f_r}$ and $J_N=\prod_j \fp_j^{e'_j}$, $I_N=\prod_s \fq_s^{f'_s}$. Note that $\Ass_R(M)$ and $\Ass_R(N)$ are the subsets of $\{\fp_1,\ldots,\fp_h,\fq_1,\ldots,\fq_k\}$.

First we prove $\rm (a) \Rightarrow (b)$. Then we find that
$$
|M/\mathbf{x}M|=|R/J_M+(\mathbf{x})| \cdot |R/I_M+(\mathbf{x})|
~\mbox{and}~|N/\mathbf{x}N|=|R/J_N+(\mathbf{x})| \cdot |R/I_N+(\mathbf{x})|.
$$
For those elements $\mathbf{x} \in R$ of specified type, it follows that $|R/J_M+(\mathbf{x})|$ is finite and bounded. So the implication $\rm(a) \Rightarrow  (b)$ follows by putting
$$
c=\max_{\mathbf{x}}\{|R/J_M+(\mathbf{x})|\}.
$$

Next we prove $\rm (b) \Rightarrow (a)$. Let us prove the equations:  
\begin{equation}
\label{linear1}
\ell_R(M/\mathbf{x}M)=\sum_{i=1}^h \ell_{R_{\fp_i}}(M_{\fp_i}) \cdot \ell_R(R/\fp_i+(\mathbf{x}))+\sum_{i=1}^k \ell_{R_{\fq_i}}(M_{\fq_i}) \cdot \ell_R(R/\fq_i+(\mathbf{x}))
\end{equation}
and
\begin{equation}
\label{linear2}
\ell_R(N/\mathbf{x}N)=\sum_{i=1}^h \ell_{R_{\fp_i}}(N_{\fp_i}) \cdot \ell_R(R/\fp_i+(\mathbf{x}))+\sum_{i=1}^k \ell_{R_{\fq_i}}(N_{\fq_i}) \cdot \ell_R(R/\fq_i+(\mathbf{x})),
\end{equation}
under the assumption $\mathbf{x} \notin \bigcup_{i=1}^h \fp_i \cup \bigcup_{i=1}^k \fq_i$. 
As the proofs of $(\ref{linear1})$ and $(\ref{linear2})$ are exactly the same, we will prove only $(\ref{linear1})$. As we will show below, the equation $(\ref{linear1})$ is reduced to the following equation. 
\begin{multline}
\label{equation:specialization11}
\ell_R(M/\mathbf{x}M)-\ell_R(M[\mathbf{x}])
\\ 
=\sum_{i=1}^h \ell_{R_{\fp_i}}(M_{\fp_i}) \cdot \ell_R(R/\fp_i+(\mathbf{x}))+\sum_{i=1}^k \ell_{R_{\fq_i}}(M_{\fq_i}) \cdot \ell_R(R/\fq_i+(\mathbf{x})).
\end{multline}
In fact, since $M[\mathbf{x}]=0$ by the assumption that $M$ is a fundamental module, $(\ref{linear1})$ follows immediately from \eqref{equation:specialization11}. To prove \eqref{equation:specialization11}, its left-hand side is additive with respect to short exact sequences of $R$-modules by Lemma \ref{lemma2}. Its right-hand side is also additive. Thus, we reduce the proof to the case where $M=R/P$ for $P \in \{\fp_1,\ldots,\fp_h,\fq_1,\ldots,\fq_k\}$ 
by taking a prime filtration of $M$. In this case, the left-hand side of the equation is $\ell_R(R/P+(\mathbf{x}))$, while the right-hand side is $\ell_{R_P}(R_P/PR_P) \cdot \ell_R(R/P+(\mathbf{x}))$. They are equal to each other. To prove $\rm (b) \Rightarrow (a)$, assume that
\begin{equation}
\label{linear5}
\ell_{R_{\fq_j}}(M_{\fq_j})>\ell_{R_{\fq_j}}(N_{\fq_j})
\end{equation}
for some $1 \le j \le k$. Thus, we may put $j=1$ for simplicity. Fix an element $z:=z_1 \in R$ such that
\begin{equation}
\label{linear55}
z \in \fq_1~\mbox{and}~z \notin \bigcup_{i \ne 1}\fq_i.
\end{equation}
By definition, $(p, z)$ is a parameter ideal of $R$ and $\Lambda_{W(\mathbb{F})}:=W(\mathbb{F})[[z]] \to R$ is module-finite. Since $\Lambda_{W(\mathbb{F})}$ is a UFD, we have $(g_i(z))=\Lambda_{W(\mathbb{F})} \cap \fq_i$, where $g_i(z)$ is an irreducible monic polynomial by Theorem \ref{Weierstrass}. From $(\ref{linear55})$, we find that
\begin{equation}
\label{linear56}
(z)=(g_1(z))~\mbox{and}~(z) \ne (g_i(z))~\mbox{for any}~i \ne 1.
\end{equation}

Applying Lemma \ref{lemma3} to the extension: $\Lambda_{W(\mathbb{F})}/(g_i(z)) \to R/\fq_i$, we obtain
\begin{equation}
\label{linear3}
\ell_R(R/\fq_i+(\mathbf{x}))=\rank_{\Lambda_{W(\mathbb{F})}/(g_i(z))}(R/\fq_i) \cdot
\ell_{\Lambda_{W(\mathbb{F})}}(\Lambda_{W(\mathbb{F})}/(g_i(z),\mathbf{x})),
\end{equation}
under the assumption $\mathbf{x} \notin (g_i(z))$ in $\Lambda_{W(\mathbb{F})}$. Let us put
$$
g_M(z):=\prod_{i=1}^kg_i(z)^{m_i},~\mbox{where}~m_i:={\rank_{\Lambda_{W(\mathbb{F})}/(g_i(z))}(R/\fq_i)  \cdot \ell_{R_{\fq_i}}(M_{\fq_i})}
$$
and
$$
g_N(z):=\prod_{i=1}^kg_i(z)^{n_i},~\mbox{where}~n_i:={\rank_{\Lambda_{W(\mathbb{F})}/
(g_i(z))}(R/\fq_i) \cdot \ell_{R_{\fq_i}}(N_{\fq_i})}).
$$
It follows that $g_N(z)$ is not divisible by $g_M(z)$ in view of $(\ref{linear5})$ and $(\ref{linear56})$. To be more precise, we have the following:
\begin{equation}
\label{linear557}
g_1(z)^{n_1}~\mbox{is not divisible by}~g_1(z)^{m_1}. 
\end{equation}
Set $\mathbf{x}_n:=z+a_np^n \in \Lambda_{W(\mathbb{F})}$ for $n>0$. Then by $(\ref{linear56})$ and $(\ref{linear557})$, a calculation over $\Lambda_{W(\mathbb{F})}$ (see \cite[Proposition 3.11 (i)]{Oc1} for the same calculation) yields that
\begin{equation}
\label{inequality1}
\text{$\frac{\vert \Lambda_{W(\mathbb{F})}/(g_N(z),\mathbf{x}_n)\vert}
{\vert \Lambda_{W(\mathbb{F})}/(g_M(z),\mathbf{x}_n) \vert}$ tends to $0$ as $n$ goes to 
$\infty$.}
\end{equation}
On the other hand, since $p \in \fp_i$, we have the following:
\begin{equation}
\label{linear57}
\limsup_{n \to \infty}  \ell_R(R/\fp_i+(\mathbf{x}_n)) < \infty~\mbox{for}~i=1,\ldots,h.
\end{equation}

By $(\ref{cardinality})$, $(\ref{linear1})$, $(\ref{linear2})$, $(\ref{linear3})$ and using Lemma \ref{zerodivisor} repeatedly, we get
\begin{align}
\label{absolutevalue1}
|M/\mathbf{x}_n M|&=|\mathbb{F}|^{\ell_R (M/\mathbf{x}_n M)}
\\
&= 
|\mathbb{F}|^{\sum_{i=1}^h \ell_{R_{\fp_i}}(M_{\fp_i}) \cdot \ell_R(R/\fp_i+(\mathbf{x}_n))} \cdot |\mathbb{F}|^{\sum_{i=1}^k \ell_{R_{\fq_i}}(M_{\fq_i}) \cdot \ell_R(R/\fq_i+(\mathbf{x}_n))} \nonumber 
\\
&=|\mathbb{F}|^{\sum_{i=1}^h \ell_{R_{\fp_i}}(M_{\fp_i}) \cdot \ell_R(R/\fp_i+(\mathbf{x}_n))} \cdot |\Lambda_{W(\mathbb{F})}/(g_M(z),\mathbf{x}_n)| \nonumber 
\end{align}
and similarly, 
\begin{equation}
\label{absolutevalue2}
|N/\mathbf{x}_n N|=|\mathbb{F}|^{\sum_{i=1}^h \ell_{R_{\fp_i}}(M_{\fp_i}) \cdot \ell_R(R/\fp_i+(\mathbf{x}_n))} \cdot |\Lambda_{W(\mathbb{F})}/(g_N(z),\mathbf{x}_n)|.
\end{equation}
Now putting together $(\ref{inequality1})$, $(\ref{linear57})$, 
$(\ref{absolutevalue1})$ and $(\ref{absolutevalue2})$, we conclude that
$$
\dfrac{|N/\mathbf{x}_n N|}{|M/\mathbf{x}_n M|}~\mbox{tends to}~ 0~\mbox{as}~n~\mbox{goes to}~\infty. 
$$
However, this is a contradiction to the hypothesis. That is, $(\ref{linear5})$ is false and we have thus proved $\rm (b) \Rightarrow (a)$. We complete the proof.\\
\end{proof}

\subsection{Specialization method in dimension two for height-one primes over $p$}
Let $(R,\fm,\mathbb{F})$ be a two-dimensional Noetherian complete local normal domain with mixed characteristic $p>0$ and finite residue field $\mathbb{F}$. Let us choose a parameter ideal $(p,z)$ of $R$ and define the set of principal ideals:
\begin{equation}
\label{Fittideal1}
\mathscr{E}_{R,W(\mathbb{F})}(z,r,\{a_n\}_{n \in \mathbb{N}}):=\{(\mathbf{x}_n)\}_{n \in \mathbb{N}},
\end{equation}
where $\mathbf{x}_n=z^n+r+a_n p$ is attached by Corollary \ref{LocalBertini} $\rm (2)$ and $\{a_n\}_{n \in \mathbb{N}}$ is any sequence of elements in $W(\mathbb{F})^{\times}$. It is necessary to make a special choice of $r \in R$. Let $\{\fp_i\}_{1 \le i \le h}$ be the set of all height-one primes of $R$ over $p$. Consider an additional condition:
\begin{equation}
\label{linearideal2}
r_t \in \fp_t~\mbox{and}~r_t \notin \bigcup_{i \ne t} \fp_i~\mbox{for each}~1 \le t \le h.~\mbox{If}~h=1,~\mbox{then choose}~r=0.
\end{equation}
We notice that the existence of $r_t \in R$ satisfying $(\ref{linearideal2})$ is ensured by Prime Avoidance Lemma. By the choice made as above, $(p,z^n+r_i)$ is a parameter ideal of $R$ for $n \in \mathbb{N}$ and $1 \le i \le h$.

\begin{theorem}[Specialization Method over $p$]
\label{prop2}
Let $(R,\fm,\mathbb{F})$ be a two-dimensional Noetherian complete local normal domain with mixed characteristic $p>0$ and finite residue field $\mathbb{F}$. 
Suppose that $M$ and $N$ are finitely generated torsion $R$-modules and fix a finite set of elements $\{r_i\}_{1 \le i \le h}$ satisfying $(\ref{linearideal2})$. Then the following statements are equivalent:

\begin{enumerate}
\item[(a)]
Let $\fp$ be any height-one prime of $R$ which lies over $p$. Then we have
$$
\Char_R(N)_{\fp} \subseteq \Char_R(M)_{\fp}.
$$

\item[(b)]
There exists a constant $c \in \mathbb{N}$, depending on $M$ and $N$, such that
$$
c \cdot \frac{|N/\mathbf{x}N |}{|M/\mathbf{x}M |} \in \mathbb{N}
$$
for all but finitely many principal ideals
$$
(\mathbf{x}) \in \bigcup_{1\le i \le h} \mathscr{E}_{R,W(\mathbb{F})}(z,r_i,\{a_n\}_{n \in \mathbb{N}}).
$$
\end{enumerate}
\end{theorem}

\begin{proof}
Take the fundamental pseudo-isomorphisms
$$
f_M:M \to \bigoplus_i R/\fp_i^{e_i} \oplus \bigoplus_r R/\fq_r^{f_r}~\mbox{and}~f_N:N \to \bigoplus_j R/\fp_j^{e'_j}\oplus \bigoplus_s R/\fq_s^{f'_s},
$$
where $\{\fp_i\}_{1 \le i \le h}$ is the set of all height-one primes of $R$ containing $p$, and $\{\fq_i\}_{1 \le i \le k}$ is the set of height-one primes not containing $p$. Then as in the proof of Theorem \ref{prop1}, it is sufficient to prove the theorem for modules of the following types:
$$
M=R/J_M \oplus R/I_M~\mbox{and}~N=R/J_N \oplus R/I_N,
$$
where $J_M=\prod_i \fp_i^{e_i}$, $I_M=\prod_r \fq_r^{f_r}$ and $J_N=\prod_j \fp_j^{e'_j}$, $I_N=\prod_s \fq_s^{f'_s}$.

First we prove $\rm (a) \Rightarrow (b)$. Then we find that
$$
|M/\mathbf{x}M|=|R/J_M+(\mathbf{x})| \cdot |R/I_M+(\mathbf{x})|~\mbox{and}~|N/\mathbf{x}N|=|R/J_N+(\mathbf{x})| \cdot |R/I_N+(\mathbf{x})|.
$$
For those elements $\mathbf{x} \in R$ of specified type, it follows that $|R/I_M+(\mathbf{x})|$ is finite and bounded. So the implication $\rm (a) \Rightarrow (b)$ follows by putting
$$
c=\max_{\mathbf{x}}\{|R/I_M+(\mathbf{x})|\}.
$$

Next we prove $\rm (b) \Rightarrow (a)$. To prove it by contradiction, let us assume that
\begin{equation}
\label{linear6}
\ell_{R_{\fp_j}}(M_{\fp_j})>\ell_{R_{\fp_j}}(N_{\fp_j})
\end{equation}
for some $1 \le j \le h$. Thus, we may put $j=1$ for simplicity. Set $\mathbf{x}_n:=z^n+r_1+a_np \in R$ for $n>0$, where  $r_1 \in \fp_1$ and $r_1 \notin \bigcup_{i \ne 1} \fp_i$. Recall the following formula:
\begin{equation}
\label{linear7}
\ell_R(M/\mathbf{x}M)=\sum_{i=1}^h \ell_{R_{\fp_i}}(M_{\fp_i}) \cdot \ell_R(R/\fp_i+(\mathbf{x}))+\sum_{i=1}^k \ell_{R_{\fq_i}}(M_{\fq_i}) \cdot \ell_R(R/\fq_i+(\mathbf{x}))
\end{equation}
and
\begin{equation}
\label{linear8}
\ell_R(N/\mathbf{x}N)=\sum_{i=1}^h \ell_{R_{\fp_i}}(N_{\fp_i}) \cdot \ell_R(R/\fp_i+(\mathbf{x}))+\sum_{i=1}^k \ell_{R_{\fq_i}}(M_{\fq_i}) \cdot \ell_R(R/\fq_i+(\mathbf{x}))
\end{equation}
under the assumption $ \mathbf{x} \notin \bigcup_{i=1}^h \fp_i \cup \bigcup_{i=1}^k \fq_i$. Let us denote by $\widetilde{R}_i$ the integral closure of $R_i:=R/\fp_i$ in $\Frac(R_i)$, where $\fp_i \in \Spec R$ is as above. As $R_i$ is a complete local domain of Krull dimension one, $R_i \to \widetilde{R}_i$ is module-finite and $\widetilde{R}_i$ is a discrete valuation ring. Moreover, $\widetilde{R}_i/R_i$ is a finitely generated torsion $R$-module. Hence we have $\ell_R(\widetilde{R}_i/R_i)<\infty$. Consider the following commutative diagram:
$$
\begin{CD}
0 @>>> R_i @>>> \widetilde{R}_i @>>> \widetilde{R}_i/R_i @>>> 0 \\
@. @V\times \mathbf{x}_nVV @V\times \mathbf{x}_nVV @V\times \mathbf{x}_nVV \\
0 @>>> R_i @>>> \widetilde{R}_i @>>> \widetilde{R}_i/R_i @>>> 0 \\
\end{CD}
$$
The left vertical and the middle vertical maps are injective, as the maps are the multiplication by a nonzero element in a domain $R_i$. Thus by applying the snake lemma to the above commutative diagram, we have the following exact sequence:
$$
0 \to \ker \big(\widetilde{R}_i/R_i \overset{\times\mathbf{x}_n}{\longrightarrow } \widetilde{R}_i/R_i\big) 
\to R_i /(\mathbf{x}_n) \to \widetilde{R}_i/(\mathbf{x}_n) 
\to \coker \big(\widetilde{R}_i/R_i \overset{\times\mathbf{x}_n}{\longrightarrow} \widetilde{R}_i/R_i\big)
\to 0. 
$$ 
Since we know $\ell_R(\widetilde{R}_i/R_i)<\infty$, it follows that the $R$-lengths of the modules in the first and the last terms of the sequence are equal to each other. Hence, the equality $\ell_R(R_i/(\mathbf{x}_n))=\ell_R(\widetilde{R}_i/(\mathbf{x}_n))$ follows. Let $v$ be the valuation of $\widetilde{R}_i$. Then we have $\ell_R(\widetilde{R}_i/(\mathbf{x}_n))=[\widetilde{\mathbb{F}}_i:\mathbb{F}] \cdot v(\mathbf{x}_n)$, where $\widetilde{\mathbb{F}}_i$ is the residue field of $\widetilde{R}_i$. We find that
\begin{equation}
\label{linear9}
\limsup_{n \to \infty}\ell_R(R/\fp_i+(\mathbf{x}_n)) < \infty~\mbox{for any}~i \ne 1,
\end{equation}
and
\begin{equation}
\label{linear10}
\lim_{n \to \infty}\ell_R(R/\fp_1+(\mathbf{x}_n))=\infty.
\end{equation}
On the other hand, since $a_np \notin \fq_i$, we find that
\begin{equation}
\label{linear11}
\limsup_{n \to \infty}\ell_R(R/\fq_i+(\mathbf{x}_n)) < \infty~\mbox{for any}~i.
\end{equation}
Now putting together $(\ref{cardinality})$, $(\ref{linear7})$, $(\ref{linear8})$, $(\ref{linear9})$, $(\ref{linear10})$ and $(\ref{linear11})$, we conclude that 
$$
\dfrac{|N/\mathbf{x}_nN |}{|M/\mathbf{x}_nM |}~\mbox{tends to}~0~\mbox{as}~n~\mbox{goes to}~\infty.
$$ 
However, this is a contradiction to the hypothesis. That is, $(\ref{linear6})$ is false and we have thus proved $\rm (b) \Rightarrow (a)$. This completes the proof.\\
\end{proof}

\subsection{Specialization method in dimension at least three}
Let us briefly recall the specialization method in the case of dimension at least three, together with notation from \cite{OcSh}. The philosophy behind the method is that, if the inclusion $\Char_{R/(\mathbf{x})}(M/\mathbf{x}M) \subseteq \Char_{R/(\mathbf{x})}(N/\mathbf{x}N)$ is fulfilled for sufficiently many $\mathbf{x} \in R$, then we can recover the inclusion $\Char_R(M) \subseteq \Char_R(N)$. 

Let $(R,\fm,\mathbb{F})$ be a Noetherian complete local domain with mixed characteristic and finite residue field $\mathbb{F}$, let $\overline{\mathbb{F}}$ be a fixed algebraic closure of $\mathbb{F}$ and let $W(\overline{\mathbb{F}})$ denote the ring of Witt vectors. Let $\mathbb{P}^n(W(\overline{\mathbb{F}}))$ denote the $n$-dimensional projective space with coordinates in $W(\overline{\mathbb{F}})$ such that every point of $\mathbb{P}^n(W(\overline{\mathbb{F}}))$ is \textit{normalized}; let us choose a point $a=(a_0:\cdots:a_d) \in \mathbb{P}^n(W(\overline{\mathbb{F}}))$. Then we require $a_i \in W(\overline{\mathbb{F}})^{\times}$ for some $0 \le i \le d$. Let us set
$$
R_{W(\overline{\mathbb{F}})}:=R \widehat{\otimes}_{W(\mathbb{F})} W(\overline{\mathbb{F}})~\mbox{and}~R_{W(\mathbb{F}')}:=R \otimes_{W(\mathbb{F})} W(\mathbb{F}'),
$$
where $\mathbb{F} \to \mathbb{F}'$ is a finite field extension. We shall need the following facts.

\begin{enumerate}
\item
$R_{W(\overline{\mathbb{F}})}$ and $R_{W(\mathbb{F}')}$ are complete local rings in view of Lemma \ref{localring}. Moreover, they are normal domains if $R$ is normal. 

\item
The natural extension $R \to R_{W(\mathbb{F}')}$ is module-finite. The extension $R_{W(\mathbb{F}')} \to R_{W(\overline{\mathbb{F}})}$ is not integral, because $W(\mathbb{F}') \to W(\overline{\mathbb{F}})$ is not integral. 

\item
Fix a set of minimal generators $x_0,\ldots,x_n$ of $\fm$. Recall that we constructed a set-theoretic injective map \cite[(4.7)]{OcSh}:
$$
\theta_{W(\overline{\mathbb{F}})}: \mathbb{P}^n(\overline{\mathbb{F}}) \hookrightarrow \mathbb{P}^n(W(\overline{\mathbb{F}})).
$$
Then the composite map $\Sp_{W(\overline{\mathbb{F}})} \circ \theta_{W(\overline{\mathbb{F}})}$ is the identity map, where
$$
\Sp_{W(\overline{\mathbb{F}})}:\mathbb{P}^n(W(\overline{\mathbb{F}})) \twoheadrightarrow \mathbb{P}^n(\overline{\mathbb{F}})
$$
is the \textit{specialization map}. This map plays an essential role in the formulation of the local Bertini theorems in mixed characteristic. 
\end{enumerate}

Let us fix a point $a=(a_0:\cdots:a_n) \in \theta_{W(\overline{\mathbb{F}})}(\mathbb{P}^n(\overline{\mathbb{F}}))$ and let $\widetilde{a}=(\widetilde{a}_0,\ldots,\widetilde{a}_n) \in \mathbb{A}^{n+1}(W(\overline{\mathbb{F}}))$ be its \textit{representative}. This we explain below. The point $\widetilde{a}=(\widetilde{a}_0,\ldots,\widetilde{a}_n) \in \mathbb{A}^{n+1}(W(\overline{\mathbb{F}}))$ is chosen in such a way that $\widetilde{a}_i \in W(\overline{\mathbb{F}})^{\times}$ for some $0 \le i \le n$. Let us form a linear combination
\begin{equation}
\label{linearcomb}
\mathbf{x}_{\widetilde{a}}:=\sum_{i=0}^n \widetilde{a}_ix_i,
\end{equation}
where $\mathbf{x}_{\widetilde{a}}$ itself depends on the choice of $\widetilde{a}$, but the principal ideal generated by it does not. We recall the following Avoidance Lemma (see \cite[Lemma 4.2]{OcSh}).

\begin{lemma}
\label{Case1}
Let us fix a set of minimal generators $x_0,\ldots,x_n$ of the maximal ideal of $R_{W(\overline{\mathbb{F}})}$, together with a non-maximal prime ideal $\fp$ of $R_{W(\overline{\mathbb{F}})}$. Then there exists a non-empty Zariski open subset $U \subseteq \mathbb{P}^n(\overline{\mathbb{F}})$ such that
$$
\mathbf{x}_{\widetilde{a}}=\sum_{i=0}^n \widetilde{a}_ix_i \notin \fp
$$
for every $a=(a_0:\cdots:a_n) \in \Sp_{W(\overline{\mathbb{F}})}^{-1}(U)$. If moreover $\widetilde{a} \in \theta_{W(\overline{\mathbb{F}})}(U)$, then there exists a finite field extension $\mathbb{F} \to \mathbb{F}'$ such that $\mathbf{x}_{\widetilde{a}}$ is an element of $R_{W(\mathbb{F}')}$.
\end{lemma}

Lemma \ref{Case1} ensures that one can find infinitely many specialization elements with certain properties. Let $M$ be a finitely generated torsion $R$-module. Then, in \cite[Definition 8.1]{OcSh}, we defined a certain subset:
$$
\mathcal{L}_{W(\overline{\mathbb{F}})}(M_{W(\overline{\mathbb{F}})}) \subseteq \mathbb{P}^n(W(\overline{\mathbb{F}}))
$$
consisting of specialization elements with good properties for $M$.
\begin{enumerate}
\item[-]
By \cite[Lemma 8.6]{OcSh}, the set $\mathcal{L}_{W(\overline{\mathbb{F}})}(M_{W(\overline{\mathbb{F}})})$ is identified with a Zariski open subset of $\mathbb{P}^n(\overline{\mathbb{F}})$ through the map $\theta_{W(\overline{\mathbb{F}})}: \mathbb{P}^n(\overline{\mathbb{F}}) \to \mathbb{P}^n(W(\overline{\mathbb{F}}))$.
\end{enumerate}
In other words, one can consider $\mathcal{L}_{W(\overline{\mathbb{F}})}(M_{W(\overline{\mathbb{F}})})$ as a classifying space of specialization elements \cite[Proposition 5.1]{OcSh}. Let $I$ be an ideal of a ring $A$. Let us denote by $\Min_A(I)$ the set of all prime ideals of $A$ that are minimal over $I$. Let us formulate \cite[Theorem 8.8]{OcSh} in the form we need. Let $M$ and $N$ be finitely generated torsion $R$-modules. Let $U_{M,N} \subseteq \mathbb{P}^n(\overline{\mathbb{F}})$ be the Zariski open subset such that $U_{M,N}$ corresponds to the intersection:
$$
\mathcal{L}_{W(\overline{\mathbb{F}})}(M_{W(\overline{\mathbb{F}})}) \cap \mathcal{L}_{W(\overline{\mathbb{F}})}(N_{W(\overline{\mathbb{F}})})
$$
via the map $\theta_{W(\overline{\mathbb{F}})}: \mathbb{P}^n(\overline{\mathbb{F}}) \to \mathbb{P}^n(W(\overline{\mathbb{F}}))$ defined in the above.

\begin{theorem}[Local Bertini theorem]
\label{theorem:previous}
Let $(R,\fm,\mathbb{F})$ be a Noetherian complete local normal domain with mixed characteristic $p>0$, finite residue field and $\depth R \ge 3$. Let us choose a non-empty Zariski open subset $U \subseteq U_{M,N}$ and let $(\mathbf{x}_{\widetilde{a}}) \subseteq R_{W(\overline{\mathbb{F}})}$ be the height-one prime ideal as in $(\ref{linearcomb})$ associated to a fixed point $a \in \theta_{W(\overline{\mathbb{F}})}(U)$. Then we can find a finite field extension $\mathbb{F} \to \mathbb{F}'$ such that
$\mathbf{x}_{\widetilde{a}}$ is an element of $R_{W(\mathbb{F}')}$ and $R_{W(\mathbb{F}')}/(\mathbf{x}_{\widetilde{a}})$ is a complete local normal domain with mixed characteristic. Moreover, the following statements are equivalent:
\begin{enumerate}
\item[\rm{(1)}]
$\Char_R(M) \subseteq \Char_R(N)$.

\item[\rm{(2)}]
For all but finitely many height-one primes:
$$
(\mathbf{x}_{\widetilde{a}}) \subseteq R_{W(\overline{\mathbb{F}})}~\mbox{with}~a \in \theta_{W(\overline{\mathbb{F}})}(U), 
$$
we have
$$
\Char_{R_{W(\mathbb{F}')}/(\mathbf{x}_{\widetilde{a}})}(M_{W(\mathbb{F}')}/\mathbf{x}_{\widetilde{a}}M_{W(\mathbb{F}')}) \subseteq \Char_{R_{W(\mathbb{F}')}/(\mathbf{x}_{\widetilde{a}})}(N_{W(\mathbb{F}')}/\mathbf{x}_{\widetilde{a}}N_{W(\mathbb{F}')}),
$$
where $\mathbb{F}'$ is any finite field extension of $\mathbb{F}$ such that $\mathbf{x}_{\widetilde{a}}$ is an element of $R_{W(\mathbb{F}')}$.

\item[\rm{(3)}]
For all but finitely many height-one primes:
$$
(\mathbf{x}_{\widetilde{a}}) \subseteq R_{W(\overline{\mathbb{F}})}~\mbox{with}~a \in \theta_{W(\overline{\mathbb{F}})}(U),
$$
we have
$$
\Char_{R_{W(\overline{\mathbb{F}})}/(\mathbf{x}_{\widetilde{a}})}(M_{W(\overline{\mathbb{F}})}/\mathbf{x}_{\widetilde{a}}M_{W(\overline{\mathbb{F}})}) \subseteq \Char_{R_{W(\overline{\mathbb{F}})}/(\mathbf{x}_{\widetilde{a}})}(N_{W(\overline{\mathbb{F}})}/\mathbf{x}_{\widetilde{a}}N_{W(\overline{\mathbb{F}})}).
$$
\end{enumerate}
\end{theorem}

\begin{proof}
For the proof of the statement that $R_{W(\mathbb{F}')}/(\mathbf{x}_{\widetilde{a}})$ is a complete local normal domain with mixed characteristic, see \cite[Corollary 4.6]{OcSh}. 

For the equivalence of the above statements, we refer the reader to \cite[Theorem 8.8]{OcSh} for complete details. Here, we indicate the requisite modifications. It suffices to prove the implication $(3) \Rightarrow (1)$ and we may assume that
$$
M=\bigoplus_i R_{W(\overline{\mathbb{F}})}/P_i^{e_i}~\mbox{and}~
N=\bigoplus_i R_{W(\overline{\mathbb{F}})}/Q_i^{f_i},
$$
where $\{P_i\}$ and $\{Q_i\}$ are certain finite sets of height-one primes. Fix a prime ideal $Q_k$ associated to $N$. As demonstrated in {\bf Step 1} of the proof of \cite[Theorem 8.8]{OcSh}, it is sufficient to find an infinite set of principal ideals $\{(\mathbf{x}_{\widetilde{a}_i})\}_{i \in \mathbb{N}}$ corresponding to the points of $\theta_{W(\overline{\mathbb{F}})}(U)$ such that
\begin{equation}
\label{MinMax}
\bigcup_{i \in \mathbb{N}} \Min_{R_{W(\overline{\mathbb{F}})}}\big(Q_k+(\mathbf{x}_{\widetilde{a}_i})\big)~\mbox{is an infinite set}.
\end{equation}
Since every prime ideal which belongs to $(\ref{MinMax})$ is of height-two in view of \cite[Definition 8.1]{OcSh} and the Krull dimension of $R$ is greater than or equal to three, these primes are properly contained in the maximal ideal of $R$. Then as in the proof of \cite[Lemma 8.6]{OcSh}, a choice of the sequence $\{\mathbf{x}_{\widetilde{a}_i}\}_{i \in \mathbb{N}}$ satisfying $(\ref{MinMax})$ is made possible by applying repeatedly Lemma \ref{Case1} to each finite set
$\Min_{R_{W(\overline{\mathbb{F}})}}(Q_k+(\mathbf{x}_{\widetilde{a}_i}))$. Now we complete the proof of the theorem.\\
\end{proof}

\section{Galois cohomology over number fields}
\label{GaloisCoho}

\subsection{Galois cohomology}
We fix a prime number $p>2$. Let $K$ be a number field and let $G_K=\Gal(\overline{K}/K)$ be the absolute Galois group. For a finite set $\Sigma$ of primes of $K$ containing all infinite primes, denote by $K_{\Sigma}$ the maximal extension of $K$ which is unramified outside $\Sigma$ and $G_{\Sigma}=\Gal(K_{\Sigma}/K)$. Most part of the discussions in this section goes over a complete Noetherian local ring that is torsion free and finite over $\Lambda$, where $\Lambda$ is either $\mathbb{Z}_p[[x_1,\ldots,x_d]]$ or $\mathbb{F}_p[[x_1,\ldots,x_d]]$. The symbol $(R,\fm,\mathbb{F})$ will denote a Noetherian complete local ring, where $\fm$ is its unique maximal ideal and $\mathbb{F}$ is its finite residue field. Let $M$ be a compact or discrete $R$-module with a continuous $G_K$-action.

We put 
$$
M^{\PD}:=\Hom(M,\mathbb{Q}_p/\mathbb{Z}_p),~M^*:=\Hom_R(M,R)~\mbox{and}~ M^*(1):=\Hom_R(M,R) \otimes_{\mathbb{Z}_p}\mathbb{Z}_p(1),
$$
where $\mathbb{Z}_p(1)$ is the \textit{Tate twist}. The $R$-module structure of $M^{\PD}$ is obvious. The $G_K$-action of $M^{\PD}$ is as follows. For $f:M \to \mathbb{Q}_p/\mathbb{Z}_p \in M^{\PD}$, we define $g \cdot f$ to be $(g \cdot f)(m):=f(g^{-1}m)$ for $g \in G_K$ and $m \in M$. For an $R$-module $M$ and an ideal $I \subseteq R$, we put
$$
M[I]:=\{m \in M~|~rm=0~\mbox{for all}~r \in I\}.
$$
For Galois cohomology groups for finite modules, we refer the reader to \cite{NSW}. In what follows, we will consider Galois cohomology for modules over $(R,\fm,\mathbb{F})$. If $M$ is a compact $R$-module with a continuous $G$-action for a profinite group $G$, its Galois cohomology is computed as
$$
H^i(G,M)=\varprojlim_n H^i(G,M/\fm^nM).
$$ 
In the case that $M$ is a discrete $G$-module, we get
$$
H^i(G,M)=\varinjlim_n H^i(G,M[\fm^n]).
$$
These cohomology groups coincide with the cohomology groups defined by continuous cochains, if $G$ is either $G_{\Sigma}$ or $G_{K_v}$ (of cohomological dimension $\ge 2$ if $p>2$). Here $G_{K_v}$ denotes the local absolute Galois group. It is known that if $M^{\PD}$ is a finitely generated $R$-module, then $H^i(G_{\Sigma},M)^{\PD}$ is finitely generated for all $i \ge 0$ \cite[Proposition 3.2]{Gr2}. We set
\begin{equation}
\label{TateShafarevich}
\textcyr{Sh}^i_{\Sigma}(M):=\ker\big(H^i(G_{\Sigma},M) \to \bigoplus_{v \in \Sigma} H^i(G_{\mathbb{Q}_v},M)\big).
\end{equation}
This is called the \textit{Tate-Shafarevich group}.

\begin{theorem}[Poitou-Tate duality]\label{theorem:P_T}
Let $M$ be either a compact or discrete $G_{\Sigma}$-module over a complete Noetherian local ring $(R,\fm,\mathbb{F})$. Then we have a perfect pairing:
$$
\textcyr{Sh}^1_{\Sigma}(M^{\PD}(1)) \times \textcyr{Sh}^2_{\Sigma}(M) \to \mathbb{Q}_p/\mathbb{Z}_p
$$
and there is the following exact sequence:
$$
0 \to H^0(G_{\Sigma},M) \to \bigoplus_{v \in \Sigma} H^0(G_{\mathbb{Q}_v},M) \to H^2(G_{\Sigma},M^{\PD}(1))^{\PD}
$$
$$
\to H^1(G_{\Sigma},M) \to \bigoplus_{v \in \Sigma} H^1(G_{\mathbb{Q}_v},M) \to H^1(G_{\Sigma},M^{\PD}(1))^{\PD}
$$
$$
\to H^2(G_{\Sigma},M) \to \bigoplus_{v \in \Sigma} H^2(G_{\mathbb{Q}_v},M) \to H^0(G_{\Sigma},M^{\PD}(1))^{\PD} \to 0.
$$
\end{theorem}

\begin{proof}
Assume that $M$ is a compact module. Since $R$ is the inverse limit of finite $p$-groups, say $R=\varprojlim_n R/\fm^n$, $M=\varprojlim_n M/\fm^n M$, and $M^{\PD}=\varinjlim_n M^{\PD}[\fm^n]$, the theorem follows by taking inverse and inductive limits for the usual Poitou-Tate duality theorem \cite[Theorem 8.6.7 and  8.6.10]{NSW} for finite discrete $G_{\Sigma}$-modules. Note that the inverse limit is an exact functor in this case. The case that $M$ is discrete can be treated similarly.
\end{proof}

Let us recall (local and global) Euler-Poincar\'e characteristic formula.

\begin{theorem}
\label{Tate}
Let $K$ be a number field with $r_1$ real and $r_2$ complex places, respectively, and let $M$ be a finite discrete $p$-torsion $G_K$-module. Let $v$ be any finite place of $K$ with $v|p$. Then 
\begin{equation}
\label{EulerPo1}
\frac{|H^0(G_{K_v},M)| \cdot |H^2(G_{K_v},M)|}{|H^1(G_{K_v},M)|}=p^{-[K_v:\mathbb{Q}_p] \cdot v_p(|M|)},
\end{equation}
where $v_p(-)$ is the $p$-adic valuation. Fix a finite set $\Sigma$ of places containing $p$ and all infinite places of $K$. Assume the $G_K$-module $M$ is unramified outside $\Sigma$. Then
\begin{equation}
\label{EulerPo2}
\frac{|H^0(G_{\Sigma},M)| \cdot |H^2(G_{\Sigma},M)|}{|H^1(G_{\Sigma},M)|}=\prod_{v:\ \mathrm{infinite}}|H^0(G_v,M)| \cdot p^{-(r_1+r_2) \cdot v_p(|M|)}.
\end{equation}
If $v$ does not divide $p$, then the quantity in the local Euler-Poincar\'e characteristic formula $(\ref{EulerPo1})$ is equal to one.
\end{theorem}

\begin{proof}
For $(\ref{EulerPo1})$ and $(\ref{EulerPo2})$, we refer the reader to \cite[Theorem 7.3.1 and Theorem 8.7.4]{NSW} and \cite[Theorem 3.2]{TK} for the proof and generalizations.\\
\end{proof}

We often use the following (well-known) lemma.

\begin{lemma}
\label{Pont}
Let $(R,\fm,\mathbb{F})$ be a Noetherian complete local ring with finite residue field, let $\fa \subseteq R$ be an ideal and let $M$ be either a compact or discrete $R$-module. Then
$$
(M/\fa M)^{\PD} \simeq M^{\PD}[\fa].
$$
\end{lemma}

\begin{proof}
The claimed isomorphism follows easily from the short exact sequence $0 \to \fa M \to M \to M/\fa M \to 0$. 
\end{proof}

\subsection{Control of Galois cohomology groups}
Let $(R,\fm,\mathbb{F})$ be a reduced local, finite and torsion free ring over $\Lambda$ with residual characteristic $p>2$, where as usual, $\Lambda$ is either $\mathbb{Z}_p[[x_1,\ldots,x_d]]$ or $\mathbb{F}_p[[x_1,\ldots,x_d]]$. Let $T$ be a continuous $R[G_{\mathbb{Q}}]$-module free of finite rank over $R$ (the base field $\mathbb{Q}$ may be replaced with a number field $K$ in the following discussion). 
In this section, we prepare some technique of the specialization method inspired by some papers of Greenberg (see \cite{Gr1}, \cite{Gr2} for more discussions).

\begin{lemma}
\label{last}
Let $(R,\fm,\mathbb{F})$ be a Noetherian complete local domain of dimension at least two with finite residue field $\mathbb{F}$ of characteristic $p>2$. 
Let $T$ be a continuous $R[G_{\mathbb{Q}}]$-module free of rank at least two over $R$ which is ramified at a finite set of primes $\Sigma$ which 
contains the prime $p$. Let us assume the following conditions:
\begin{enumerate}
\item[\rm{(i)}]
The $G_{\mathbb{Q}}$-representation $T \otimes_R R/\fm$ is irreducible.
\item[\rm{(ii)}]
$H^2 (G_{\mathbb{Q}_\ell},T^\ast (1))=0$ for every $\ell \in \Sigma \setminus \{\infty\}$. 
\end{enumerate}
Then, the natural $R/\fa$-module map:
$$
\textcyr{Sh}^2_{\Sigma}(T^\ast (1))/ \fa \textcyr{Sh}^2_{\Sigma}(T^\ast (1))
\to 
\textcyr{Sh}^2_{\Sigma}((T^*/\fa T^*)(1)), 
$$
which is induced by a surjection of $R[G_{\mathbb{Q}}]$-modules: $T^*(1) \twoheadrightarrow 
(T^*/\fa T^*)(1)$, is an isomorphism for any ideal $\fa \subseteq R$. 
\end{lemma}

Before proving the lemma, we prepare the notation. 
Let $D:=T \otimes_R R^{\PD}$ on which $G_{\Sigma}=\Gal(\mathbb{Q}_{\Sigma}/\mathbb{Q})$ acts by hypothesis and let us note that
\begin{equation}
\label{niceexact}
(T^*)^{\PD} \simeq \Hom_{\mathbb{Z}_p}(T^* \otimes_R R,\mathbb{Q}_p/\mathbb{Z}_p) \simeq \Hom_R(T^*,R^{\PD}) \simeq D. 
\end{equation}

\begin{proof} 
By $(\ref{niceexact})$ together with the Poitou-Tate duality (Theorem \ref{theorem:P_T}) and Lemma \ref{Pont}, 
the natural $R/\fa$-module map:
$$
\textcyr{Sh}^2_{\Sigma}(T^\ast (1))/ \fa \textcyr{Sh}^2_{\Sigma} (T^\ast (1))
\to 
\textcyr{Sh}^2_{\Sigma}((T^*/\fa T^*)(1)) 
$$
is an isomorphism if and only if the natural map of $R/\fa$-modules:
\begin{equation}\label{equation:sha^1_isom}
\textcyr{Sh}^1_{\Sigma}(D[\fa]) \to \textcyr{Sh}^1_{\Sigma}(D)[\fa]
\end{equation}
is an isomorphism. In order to prove this, we consider the following commutative diagram:  
\begin{equation}\label{equation:commutative_diagram_d[a]}
\begin{CD}
0 @>>> \textcyr{Sh}^1_{\Sigma}(D[\fa]) @>>> H^1 ( G_{\Sigma} , D[\fa ] ) @>>> \displaystyle{\prod_{\ell \in \Sigma \setminus \{\infty\}}}
H^1(G_{\mathbb{Q}_\ell},D[\fa ]) \\ 
@. @VVV @VVV @VVV \\ 
0 @>>> \textcyr{Sh}^1_{\Sigma}(D)[\fa] @>>> H^1 ( G_{\Sigma} , D )[\fa ] @>>> \displaystyle{\prod_{\ell \in \Sigma \setminus \{\infty\}}}
H^1(G_{\mathbb{Q}_\ell},D) [\fa ]. \\ 
\end{CD}
\end{equation}
In order to show that \eqref{equation:sha^1_isom} is an isomorphism, it suffices to show that the map 
$\prod_{\ell \in \Sigma} H^1(G_{\mathbb{Q}_\ell},D[\fa]) \to \prod_{\ell \in \Sigma} H^1(G_{\mathbb{Q}_\ell},D)[\fa]$ of 
\eqref{equation:commutative_diagram_d[a]} is injective and the map $H^1(G_{\Sigma},D[\fa ]) \to H^1(G_{\Sigma} ,D)[\fa]$ 
of \eqref{equation:commutative_diagram_d[a]} is an isomorphism by the snake lemma. 
First, we prove that $\prod_{\ell \in \Sigma} H^1(G_{\mathbb{Q}_\ell},D[\fa]) \to \prod_{\ell \in \Sigma} H^1(G_{\mathbb{Q}_\ell},D)[\fa]$ is injective. Let $\fa=(x_1,\ldots,x_n)$. Then there is a short exact sequence:
\begin{equation}\label{equation:shortD}
0 \to D[x_1,\ldots,x_i] \to D[x_1,\ldots,x_{i-1}] \xrightarrow{\times x_i} x_i D[x_1,\ldots,x_{i-1}] \to 0.
\end{equation}
The condition $\rm(ii)$ together with the local Tate duality shows that $H^0(G_{\mathbb{Q}_{\ell}},D)=0$. Then since $x_i D[x_1,\ldots,x_{i-1}]$ is a submodule of $D$, the induced map: 
$$
\alpha_i:H^1(G_{\mathbb{Q}_\ell},D[x_1,\ldots,x_i]) \to H^1(G_{\mathbb{Q}_\ell},D[x_1,\ldots,x_{i-1}])[x_i ]
$$
is injective. Composing the maps $\alpha_i$ for $i=1,\ldots,n$, we see that the map $H^1(G_{\mathbb{Q}_\ell},D[\fa]) \to 
H^1(G_{\mathbb{Q}_\ell},D)[\fa ]$ is injective. We have proved that the map \eqref{equation:sha^1_isom} is injective and thus, the map $\prod_{\ell \in \Sigma} H^1(G_{\mathbb{Q}_\ell},D[\fa]) \to \prod_{\ell \in \Sigma} H^1(G_{\mathbb{Q}_\ell},D)[\fa]$ is also injective.

Next, we prove that $H^1(G_{\Sigma},D[\fa ]) \to H^1(G_{\Sigma} ,D)[\fa]$ is an isomorphism. Since the rank of $T$ is at least two by assumption, the condition $\rm(i)$ together with Nakayama's lemma implies that $D$ admits no nonzero subquotient on which $G_{\Sigma}$ acts trivially. In particular, we have $H^0(G_{\Sigma},D)=0$. Then by the same argument as above, we can show that 
$H^1(G_{\Sigma},D[\fa]) \to H^1(G_{\Sigma},D)[\fa]$ is injective. In order to prove the surjectivity of $H^1(G_{\Sigma},D[\fa]) \to H^1(G_{\Sigma},D)[\fa]$, 
we consider the short exact sequence:
$$
0 \to x_i D[x_1,\ldots,x_{i-1}] \to D[x_1,\ldots,x_{i-1}] \to D[x_1,\ldots,x_{i-1}]/x_i D[x_1,\ldots,x_{i-1}] \to 0.
$$
Since $D[x_1,\ldots,x_{i-1}]/x_i D[x_1,\ldots,x_{i-1}]$ is a subquotient of $D$, we have
$$
H^0(G_{\Sigma},D[x_1,\ldots,x_{i-1}]/x_i D[x_1,\ldots,x_{i-1}])=0
$$
and hence the following injection is induced: 
\begin{equation}\label{eaution:injection}
H^1(G_{\Sigma},x_i D[x_1,\ldots,x_{i-1}]) \hookrightarrow H^1(G_{\Sigma},D[x_1,\ldots,x_{i-1}]). 
\end{equation}
We also note that $D[x_1,\ldots,x_{i-1}] \xrightarrow{\times x_i} D[x_1,\ldots,x_{i-1}]$ factors as 
\begin{equation}\label{equation:composition}
D[x_1,\ldots,x_{i-1}] \to x_iD[x_1,\ldots,x_{i-1}] \to D[x_1,\ldots,x_{i-1}].
\end{equation}
The injection \eqref{eaution:injection} combined with the map of cohomology induced from \eqref{equation:composition} yields
\begin{multline*}
\ker\big[H^1(G_{\Sigma},D[x_1,\ldots,x_{i-1}]) \to H^1(G_{\Sigma},x_iD[x_1,\ldots,x_{i-1}])\big]
\\ 
=\ker\big[ H^1(G_{\Sigma},D[x_1,\ldots,x_{i-1}]) \xrightarrow{\times x_i} H^1(G_{\Sigma},D[x_1,\ldots,x_{i-1}])\big] . 
\end{multline*}
This equality combined with the long exact sequence of cohomology induced from \eqref{equation:shortD}
shows the surjectivity of the natural map $\beta_i:H^1(G_{\Sigma},D[x_1,\ldots,x_i]) \to H^1(G_{\Sigma},D[x_1,\ldots,x_{i-1}])[x_i]$. Composing the maps $\beta_i$ for $i=1,\ldots,n$, we see that $H^1(G_{\Sigma},D[\fa ]) \to H^1(G_{\Sigma} ,D)[\fa]$ is surjective. This completes the proof of the lemma. 
\end{proof}

\begin{lemma}
\label{last2}
Let $(R,\fm,\mathbb{F})$ be a Noetherian complete local normal domain of dimension at least two with finite residue field $\mathbb{F}$ of characteristic $p>0$ and assume that the following conditions hold: 

\begin{enumerate}
\item[\rm{(i)}]
$H^2(G_{\Sigma},D)$ is a finite group (resp. a trivial group).

\item[\rm{(ii)}]
$T^*(1)/(T^*(1))^{G_{\mathbb{Q}_\ell}}$ is $R$-reflexive for all $\ell \in \Sigma$.

\item[\rm{(iii)}]
$(T^*(1))^{G_{\mathbb{Q}_{\ell_0}}}=0$ for at least one $\ell_0 \in \Sigma \setminus \{\infty\}$.
\end{enumerate}
Then there exist finitely many height-one primes $\fp_1,\ldots,\fp_m$ with the property that, if $\fa \subseteq R$ is a principal ideal with $\fa \nsubseteq \bigcup_{i=1}^m \fp_i$, then
$$
H^2(G_{\Sigma},D[\fa])
$$
is a finite group (resp. a trivial group).
\end{lemma}

\begin{proof}
Recall that a Noetherian domain is \textit{reflexive}, if we have $R=\cap_{\fp} R_{\fp}$, where $\fp$ ranges over all height-one primes of $R$. Since $R$ is a normal domain, it is a reflexive domain and a result of Greenberg \cite[Theorem 1]{Gr2} applies to our case in view of $\rm(ii)$ and $\rm(iii)$. The Pontryagin dual of $\textcyr{Sh}^2_{\Sigma}(D)$ is $R$-reflexive. However, since $H^2(G_{\Sigma},D)$ is finite by $\rm(i)$, it follows that $\textcyr{Sh}^2_{\Sigma}(D)$ is finite and hence it is trivial. We can apply \cite[Proposition 6.10]{Gr2} to our situation. The proof of \cite[Proposition 6.10]{Gr2} will work over a module-finite extension of $\Lambda$. Write $\fa=(f)$ with $f \in R$. Then  
we have $H^2(G_{\Sigma},D[\fa])=H^2(G_{\Sigma},D[f])$. Consider the long exact sequence:
$$
H^1(G_{\Sigma},D) \xrightarrow{\times f} H^1(G_{\Sigma},D) \to H^2(G_{\Sigma},D[f]) \to H^2(G_{\Sigma},D) \xrightarrow{\times f} H^2(G_{\Sigma},D) \to 0,
$$
which is induced by $0 \to D[f] \to D \xrightarrow{\times f} D \to 0$. Then the map $H^1(G_{\Sigma},D) \xrightarrow{\times f} H^1(G_{\Sigma},D)$ is surjective, if $\fa \nsubseteq \bigcup_{i=1}^m \fp_i$ is satisfied for a certain set of height-one primes $\fp_1,\ldots,\fp_m$ of $R$ by \cite[Proposition 6.10]{Gr2}. Hence the map $H^2(G_{\Sigma},D[f]) \to H^2(G_{\Sigma},D)$ is injective, which proves that $H^2(G_{\Sigma},D[f])$ is finite by $\rm(i)$. The case where $H^2(G_{\Sigma},D)$ is trivial can be treated similarly.\\
\end{proof}

\section{Euler system theory over a Noetherian complete local ring}
\label{EulerSystem}
\subsection{Definition of Euler system}
An Euler system for a $p$-adic Galois representation is defined as a norm-compatible system of elements in the tower of the first Galois cohomology groups. We fix a prime number $p>0$ and let $\Sigma$ be a finite set of primes of $\mathbb{Q}$ containing $\{p,\infty \}$. Let $n>0$ be a square-free integer that is not divisible by primes in $\Sigma$ and let $\mathbb{Q}(\mu_n)$ be the $n$-th cyclotomic extension by adjoining a primitive $n$-th root of unity $\mu_n$. Let $G_{\mathbb{Q}(\mu_n)}:=\Gal(\overline{\mathbb{Q}}/\mathbb{Q}(\mu_n))$ and let $\mathbb{Q}(\mu_n)_{\Sigma}/\mathbb{Q}(\mu_n)$ be the maximal Galois extension that is unramified outside primes lying above those in $\Sigma$ and put $G_{\Sigma,n}:=\Gal(\mathbb{Q}(\mu_n)_{\Sigma}/\mathbb{Q}(\mu_n))$. We will write $G_{\Sigma}:=G_{\Sigma,1}$ for simplicity.

\begin{definition}[Euler system over $\mathbb{Q}$]
\label{definitionEuler}
Let $T$ be a finite free module over a Noetherian complete local ring $(R,\fm,\mathbb{F})$ with finite residue field, where $T$ is equipped with a continuous $G_{\mathbb{Q}}$-action and $T$ is unramified outside a finite set of primes $\Sigma$ containing $\{p, \infty\}$. Let $\mathfrak{N}$ be the set of square-free integers which are all relatively prime to $\Sigma$. An \textit{Euler system} for $(T,R,\Sigma)$ is a collection of cohomology classes
$$
\big\{\mathbf{z}_n \in H^1(G_{\mathbb{Q}(\mu_n)},T^*(1))\big\}_{n \in \mathfrak{N}}
$$ 
which satisfies the following properties:
\begin{enumerate}
\item[\rm{(i)}]
$\mathbf{z}_n$ is unramified at all primes of $\mathbb{Q}(\mu_n)$ which are not over primes dividing $np$. 

\item[\rm{(ii)}]
For a prime $\ell$ with $(n,\ell)=1$ and for $n\ell \in \mathfrak{N}$, we have 
$$
\Cor_{\mathbb{Q}(\mu_{n\ell})/\mathbb{Q}(\mu_{n})}(\mathbf{z}_{n\ell})=P(\mathrm{Frob}_{\ell};T)\mathbf{z}_n,
$$ 
where $P(t;T)=\det(1-\Frob_{\ell}t:T \to T)$ and $\Frob_{\ell}$ is the geometric Frobenius in the Galois group $\Gal(\mathbb{Q}(\mu_n)/\mathbb{Q})$.
\end{enumerate}
\end{definition}

For Euler system over a general number field, we refer the reader to \cite{Rub}. By Definition \ref{definitionEuler} (i), we see that $\mathbf{z}_n$ descends to an element in a finitely generated $R$-module $H^1(G_{\Sigma,n},T^*(1))$. There are a few examples of Euler systems discovered and its finding is a deep arithmetic problem. We will need the following condition:

\begin{enumerate}
\item[($\bf{RED}$):] $(R,\fm,\mathbb{F})$ is a reduced local ring and finite flat over $\mathbb{Z}_p$. 
\end{enumerate}

A commutative ring satisfying the condition $(\bf{RED})$ is a one-dimensional local ring, but it is not necessarily a \textit{principal ideal ring}. Thus, it is necessary to develop techniques which allow us to calculate Galois cohomology groups.

\begin{lemma}
\label{dis1}
Assume that $(R,\fm,\mathbb{F})$ is a Noetherian complete reduced local ring with finite residue field. Let $T$ be a continuous $R[G_{\Sigma}]$-module free of rank at least two over $R$. Assume that the $G_{\Sigma}$-representation $T \otimes_R R/\fm$ is irreducible. Then $H^1(G_{\Sigma},T^*(1))$ is an $R$-torsion free module. 
\end{lemma}

\begin{proof}
The assumption implies that $T^*(1) \otimes_R R/\fm$ is irreducible, since $T^*(1)$ is a twist of $T^*$ by a cyclotomic character and we have $(T^*(1) \otimes_R R/\fm)^{G_{\Sigma}}=0$ by the assumption that $\rank_R T \ge 2$. It follows that for any nonzero divisor $x \in \fm$, the quotient $T^*(1)/x T^*(1)$ is irreducible as a representation of $G_{\Sigma}$ by Nakayama's lemma. Then this implies that
$$
(T^*(1)/x T^*(1))^{G_{\Sigma}}=0,
$$
and the short exact sequence $0 \to T^*(1) \xrightarrow{\times x} T^*(1) \to T^*(1)/x T^*(1) \to 0$ yields that the map $H^1(G_{\Sigma},T^*(1)) \xrightarrow{\times x} H^1(G_{\Sigma},T^*(1))$ is injective. In particular, $H^1(G_{\Sigma},T^*(1))$ contains no nonzero pseudo-null submodules.
\end{proof}

\begin{definition}
\label{definition:isogeneous} 
Let $T_1$ and $T_2$ be finitely generated torsion-free $R$-modules over a Noetherian complete local ring $(R,\fm,\mathbb{F})$ with finite residue field, equipped with continuous $G_K$-actions. Then $T_1$ and $T_2$ are \textit{isogeneous}, if there exists an injective $R[G_K]$-homomorphism $\phi:T_1 \to T_2$ whose cokernel is annihilated by a nonzero divisor of $R$. In this case, $\phi$ is called an \textit{isogeny}. 
\end{definition}

\begin{lemma}
The relation of being isogeneous defined in Definition \ref{definition:isogeneous} is an equivalence relation. 
\end{lemma}

\begin{proof}
Let us take an injective $R[G_K]$-homomorphism $\phi:T_1 \hookrightarrow T_2$ as in Definition \ref{definition:isogeneous}, whose cokernel is annihilated by a nonzero divisor $x$ of $R$. Identify $T_1$ with its image $\phi(T_1)$. Since $x$ annihilates $T_2/T_1$, the $R[G_K]$-module $xT_2$ is identified with a submodule of $T_1$. So we have that $T_1/xT_2$ is a submodule of $T_2/xT_2$ and $x(T_1/xT_2)=0$. Since $x$ is a nonzero divisor of $R$, the multiplication $R \xrightarrow{\times x} R$ is injective. Since $T_2$ is a torsion-free $R$-module, we have an isomorphism $T_2 \simeq xT_2$. So we obtained a $G_K$-equivariant injection $T_2 \simeq xT_2 \hookrightarrow T_1$ whose cokernel is annihilated by $x$.  

Suppose that we are given another injective $R[G_K]$-homomorphism $\phi':T_2 \to T_3$ as in Definition \ref{definition:isogeneous}, whose cokernel is annihilated by a nonzero divisor $y$ of $R$. Then the cokernel of the injective $R[G_K]$-homomorphism $\phi' \circ \phi :T_1 \to T_3$ is annihilated by $xy$. 
Thus we proved that the relation of being isogeneous defined in Definition \ref{definition:isogeneous} 
is an equivalence relation. 
\end{proof}

We construct an Euler system in an isogeneous representation.

\begin{lemma}
Assume that the ring $(R,\fm,\mathbb{F})$ satisfies the condition $(\bf{RED})$. Let $T$ be a continuous $R[G_{\Sigma}]$-module which is finite free over $R$ and let $\widetilde{T}:=T \otimes_R \widetilde{R}$, where $\widetilde{R}$ is the normalization of $R$ in its total ring of fractions. Let $\big\{\mathbf{z}_n \in H^1(G_{\mathbb{Q}(\mu_n)},T^*(1))\big\}_{n \in \mathfrak{N}}$ be an Euler system for $(T,R,\Sigma)$. Then there is a natural $R$-module map
$$
H^1(G_{\mathbb{Q}(\mu_n)},T^*(1)) \to H^1(G_{\mathbb{Q}(\mu_n)},\widetilde{T}^*(1))
$$
and one can obtain an Euler system for $(\widetilde{T},\widetilde{R},\Sigma)$:
$$
\big\{\widetilde{\mathbf{z}}_n \in H^1(G_{\mathbb{Q}(\mu_n)},\widetilde{T}^*(1))\big\}_{n \in \mathfrak{N}}
$$ 
which is given as the image of $\big\{\mathbf{z}_n \in H^1(G_{\mathbb{Q}(\mu_n)},T^*(1))\big\}_{n \in \mathfrak{N}}$. 

Finally, assume that $|H^1(G_{\Sigma},T^*(1))/R \mathbf{z}_1|<\infty$. Then we have $|H^1(G_{\Sigma},\widetilde{T}^*(1))/\widetilde{R} \widetilde{\mathbf{z}}_1|<\infty$.
\end{lemma}

\begin{proof}
Since $\widetilde{R}/R$ is finite, the cokernel of the injective composed map 
$$
\Hom_R(T,R) \to \Hom_R(T,R)\otimes_R \widetilde{R} \to \Hom_{\widetilde{R}}(T \otimes_R \widetilde{R},\widetilde{R})
$$ 
is finite. Let us put $\widetilde{T}^*(1):=\Hom_{\widetilde{R}}(T \otimes_R \widetilde{R},\widetilde{R})(1)$. This induces a natural map of $R$-modules with finite kernel and cokernel:
\begin{equation}
\label{induced1}
H^1(G_{\mathbb{Q}(\mu_n)},T^*(1)) \to H^1(G_{\mathbb{Q}(\mu_n)},\widetilde{T}^*(1)).
\end{equation}
Then via $(\ref{induced1})$, one can define a family of elements:
$$
\big\{\widetilde{\mathbf{z}}_n \in H^1(G_{\mathbb{Q}(\mu_n)},\widetilde{T}^*(1))\big\}_{n \in \mathfrak{N}}
$$ 
as the image of $\big\{\mathbf{z}_n \in H^1(G_{\mathbb{Q}(\mu_n)},T^*(1))\big\}_{n \in \mathfrak{N}}$. This gives an Euler system for $(\widetilde{T},\widetilde{R},\Sigma)$. Indeed, it is evident from the construction that $\widetilde{\mathbf{z}}_n$ is unramified outside $\Sigma$, which descends to an element in $H^1(G_{\Sigma,n},\widetilde{T}^*(1))$ via the surjection $G_{\mathbb{Q}(\mu_n)} \twoheadrightarrow G_{\Sigma,n}$.

Finally, assume that $|H^1(G_{\Sigma},T^*(1))/R \mathbf{z}_1|< \infty$. Since both kernel and cokernel of $(\ref{induced1})$ are finite, it follows that $|H^1(G_{\Sigma},\widetilde{T}^*(1))/\widetilde{R} \widetilde{\mathbf{z}}_1|< \infty$ by snake lemma.
\end{proof}

\begin{lemma}
\label{dis2}
Let $(R,\fm,\mathbb{F})$ be a Noetherian complete local ring with finite residue field and let $M$ be a finitely generated $R$-module with a continuous $G$-action, where $G$ is a profinite group.

\begin{enumerate}
\item[(i)]
Assume that $R$ is reduced and $R \to \widetilde{R}$ is the normalization map. Then the induced map
$$
H^i(G,M) \otimes_R \widetilde{R} \to H^i(G,M \otimes_R \widetilde{R})
$$
has $R$-torsion kernel and cokernel for all $i \ge 0$.

\item[(ii)]
Let $(R,\fm,\mathbb{F}) \to (S,\fn,\mathbb{F}')$ be a flat map of complete Noetherian local rings with finite residue fields. Then the natual map:
$$
H^i(G,M) \otimes_R S \to H^i(G,M \otimes_R S)
$$
is an isomorphism for all $i \ge 0$.
\end{enumerate}
\end{lemma}

\begin{proof}
We prove the assertion $\rm(i)$. Let $P_{\bullet}$ denote the projective resolution of the trivial $G$-module $\mathbb{Z}$. Then the natural map
$$
\Hom(P_{\bullet},M) \otimes_R \widetilde{R} \to \Hom(P_{\bullet},M \otimes_R \widetilde{R})
$$
is an isomorphism, since every $P_i$ is a finitely generated projective module. Since $R$ is a complete local ring and $\Frac(R)=\Frac(\widetilde{R})$, it follows that $R \to \widetilde{R}$ is module-finite and $\widetilde{R}/R$ is an $R$-torsion module. Thus, both kernel and cokernel of the natural map $H^i(G,M) \otimes_R \widetilde{R} \to H^i(\Hom(P_{\bullet},M) \otimes_R \widetilde{R})$ are $R$-torsion. Composing this map with $H^i(\Hom(P_{\bullet},M) \otimes_R \widetilde{R}) \simeq H^i(G,M \otimes_R \widetilde{R})$, both kernel and cokernel of the cohomology map 
$$
H^i(G,M) \otimes_R \widetilde{R} \to H^i(G,M \otimes_R \widetilde{R})
$$ 
are $R$-torsion modules.

We prove the assertion $\rm(ii)$. Since $R \to S$ is flat, we have $(M \otimes_R S)^H=M^H \otimes_R S$ for any subgroup $H \subseteq G$. The proof may be completed by keeping track of $\rm(i)$ by replacing $R \to \widetilde{R}$ with $R \to S$. 
\end{proof}

\begin{lemma}
\label{cond}
Assume that $(R,\fm,\mathbb{F})$ satisfies the condition $(\bf{RED})$ and $T$ is a finite free $R$-module with a continuous $G_{K_v}$-action, where $K$ is a number field and $v$ is a non Archimedean prime of $K$. Let $\widetilde{T}:=T \otimes_R \widetilde{R}$ and assume that $H^2(G_{K_v},T^*(1))=0$. Then $H^2(G_{K_v},\widetilde{T}^*(1))=0$.
\end{lemma}

\begin{proof}
Since $H^2(G_{K_v},\widetilde{T}^*(1))^{\PD} \simeq H^0(G_{K_v},(\widetilde{T}^*)^{\PD})=((\widetilde{T}^*)_{G_{K_v}})^{\PD}$ by local Tate duality theorem \cite[Theorem 1.4.1]{Rub}, we have $H^2(G_{K_v},\widetilde{T}^*(1))=0$ if and only if $(\widetilde{T}^*)_{G_{K_v}}=0$. Hence, it suffices to prove that $(\widetilde{T}^*)_{G_{K_v}}=0$. 
We remark that $H^2(G_{K_v},{T}^*(1))=0$ if and only if $({T}^*)_{G_{K_v}}=0$ by the same argument using local Tate duality. 

Consider the short exact sequence of $R$-modules $0 \to R \to \widetilde{R} \to \widetilde{R}/R \to 0$. 
Since $R$ satisfies $(\bf{RED})$, the quotient $\widetilde{R}/R$ is an 
Artinian $R$-module. 
By applying the tensor product $(-)\otimes_R T^*$ with free $R$-module $T^*$ 
to the above short exact sequence, 
we have a short exact sequence $0 \to T^* \to \widetilde{T}^* \to T^* \otimes_R \widetilde{R}/R \to 0$. 
Taking the $G_{K_v}$-coinvariant quotient, we have an exact sequence:
$$
(T^*)_{G_{K_v}} \to (\widetilde{T}^*)_{G_{K_v}} \to (T^* \otimes_R \widetilde{R}/R)_{G_{K_v}} \to 0.
$$
By the above remark together with the assumption $H^2(G_{K_v},T^*(1))=0$, 
the first term $(T^*)_{G_{K_v}}$ of this sequence is already known to be zero. 
From this, it suffices to prove that the last term 
$(T^* \otimes_R \widetilde{R}/R)_{G_{K_v}}$ is zero in order to prove that the middle 
term is zero. Since $R$ is local, any simple 
$R$-subquotient of $\widetilde{R}/R$ is isomorphic to $R/\fm R$.   
Hence, the $R$-module $T^* \otimes_R \widetilde{R}/R$ has finite length and is obtained by taking successive extensions of $T^*/\fm T^*$. Recall that we have $(T^*/\fm T^*)_{G_{K_v}}=0$ since 
$(T^*)_{G_{K_v}}=0$. 
If we have an exact sequence $0 \to M' \to M \to M'' \to 0$ of Artinian $R$-modules, 
we have an exact sequence as follows 
$$
(T^* \otimes_R M' )_{G_{K_v}}  \to (T^* \otimes_R M)_{G_{K_v}} \to (T^* \otimes_R M'' )_{G_{K_v}} \to 0 . 
$$
by the argument similar to the above. By the induction argument with respect to the $R$-length of $M$, 
we prove $(T^* \otimes_R M)_{G_{K_v}}=0$ for any Artinian $A$-module $M$ and hence $(T^* \otimes_R \widetilde{R}/R)_{G_{K_v}}=0$.\\
\end{proof}

\subsection{Euler system bound over a reduced local ring of dimension one}

Let $T$ be as in the setting of Euler system over $\mathbb{Q}$ and recall the notation $D:=T \otimes_R R^{\PD}$, the discrete Galois module on which the Galois group $G_{\mathbb{Q}}$ acts through the first factor.

\begin{definition}
Let $(R,\fm,\mathbb{F})$ be a complete reduced local ring satisfying $(\bf{RED})$. Then \textit{Euler system bound} holds for $(T,R,\Sigma)$, if both of the following conditions are satisfied:

\begin{enumerate}
\item
$\textcyr{Sh}^2_{\Sigma}(T^*(1))$ is a finite group.

\item
$c \cdot |H^1(G_{\Sigma},T^*(1))/R \mathbf{z}_1|$ is divisible by $| \textcyr{Sh}^2_{\Sigma}(T^*(1))|$ for some $c \in \mathbb{N}$. Let us call $c$ an \textit{error term}, which is explicitly specified by $(T,R,\Sigma)$.
\end{enumerate}
\end{definition}

It is a deep arithmetic problem to determine under which conditions $H^1(G_{\Sigma},T^*(1))/R \mathbf{z}_1$ (or $\textcyr{Sh}^2_{\Sigma}(T^*(1))$) is a finite group. Our aim in \S~\ref{proof2} is to prove that Euler system bound holds over a module-finite extension $\Lambda=\mathbb{Z}_p[[x_1,\ldots,x_d]] \to R$, where $R$ is a Cohen-Macaulay normal domain. To this aim, we prove the following theorem as a direct generalization of Euler system bound over a discrete valuation ring, which was previously established by Kato, Perrin-Riou and Rubin independently (see \cite{KatoKodai}, \cite{Pe} and \cite{Rub} for their work), to a local ring satisfying $(\bf{RED})$. See also \cite[Theorem 4.7]{Oc1}. Let $\mu_A(M)$ denote the number of minimal generators of a finitely generated module $M$ over a local ring $(A,\fm)$. This is known to be equal to $\dim_{A/\fm}(M/\fm M)$ by Nakayama's lemma.

\begin{theorem}\label{Euler}
Let $(R,\fm,\mathbb{F})$ be a complete Noetherian reduced local ring which is finite and flat over $\mathbb{Z}_p$ for a prime number $p>2$. Suppose that
$$
\big\{\mathbf{z}_n \in H^1(G_{\Sigma,n},T^*(1))\big\}_{n \in \mathfrak{N}}
$$ 
is an Euler system for $(T,R,\Sigma)$, $T$ is a free $R$-module of rank two and suppose that the following conditions hold:

\begin{enumerate}
\item[\rm{(i)}]
$T \otimes_R R/\fm$ is absolutely irreducible as a representation of $G_{\mathbb{Q}}$.

\item[\rm{(ii)}]
The quotient $H^1(G_{\Sigma},T^*(1))/R \mathbf{z}_1$ is a finite group.

\item[\rm{(iii)}]
$H^2(G_{\mathbb{Q}_{\ell}},T^*(1))=0$ for every $\ell \in \Sigma \setminus \{\infty\}$ and $H^2(G_{\Sigma},D)$ is a finite group.

\item[\rm{(iv)}]
The determinant character $\wedge^2 \rho:G_{\mathbb{Q}} \to R^{\times}$ (resp. $\wedge^2 \rho^*(1):G_{\mathbb{Q}} \to R^{\times}$) associated with the $G_{\mathbb{Q}}$-representation $T$ (resp. $T^*(1)$) has an element of infinite order.

\item[\rm{(v)}]
The $R$-module $T$ splits into eigenspaces: $T=T^+ \oplus T^-$ with respect to the complex conjugation in $G_{\mathbb{Q}}$, and $T^+_{\fp}$ (resp. $T^-_{\fp}$) is of $R_{\fp}$-rank one for each minimal prime $\fp$ of $R$.

\item[\rm{(vi)}]
There exist $\sigma_1 \in G_{\mathbb{Q}(\mu_{p^{\infty}})}$ and $\sigma_2 \in G_{\mathbb{Q}}$ such that $\rho(\sigma_1) \simeq \begin{pmatrix} 1 & \epsilon \\ 0 & 1 \end{pmatrix} \in GL_2(R)$ for a nonzero divisor $\epsilon \in R$ and $\sigma_2$ acts on $T$ as multiplication by $-1$. 

\end{enumerate}
Then Euler system bound holds for $(T,R,\Sigma)$:
\begin{enumerate}
\item[\rm{(1)}]
$\textcyr{Sh}^2_{\Sigma}(T^*(1))$ is a finite group.
\item[\rm{(2)}]
We have
$$
c \cdot |H^1(G_{\Sigma},T^*(1))/R \mathbf{z}_1|~\mbox{is divisible by}~| \textcyr{Sh}^2_{\Sigma}(T^*(1))|,
$$
where $c:=|R/(\epsilon^k)|$ is an error term for Euler system bound and $k$ is the number of minimal generators of the $R$-module $\textcyr{Sh}^2_{\Sigma}(T^*(1))$.
\end{enumerate}
\end{theorem}

Before starting the proof, we fix notation. We denote by $\widetilde{R}$ the integral closure of the reduced local ring $R$ in its total ring of fractions. Let $\fp_1,\ldots,\fp_n$ be the set of all minimal primes of $R$. Set $R_i:=R/\fp_i$, $\widetilde{T}:=T \otimes_R \widetilde{R}$ and $\widetilde{T}_i:=T \otimes_R \widetilde{R}_i$.

\begin{proof}
First, we prove the assertion $\rm(1)$, which is not so hard. Then we establish the assertion $\rm(2)$, which requires some computations using Poitou-Tate and Pontryagin dualities.

Let us first show that the hypotheses of the theorem remain true after replacing $(T,R,\Sigma)$ with $(\widetilde{T},\widetilde{R},\Sigma)$. By construction, we have the decomposition $\widetilde{R} =\bigoplus_i \widetilde{R}_i$, where the product is the finite dirct product of rings and each $\widetilde{R}_i$ is a $p$-torsion free complete discrete valuation ring with finite residue field. Moreover, there is the decomposition $\widetilde{T}=\bigoplus_i \widetilde{T}_i$ and each $\widetilde{T}_i$ is a free $\widetilde{R}_i$-module of rank two. The map $R \to \widetilde{R}_i$ factors as $R \twoheadrightarrow R_i \hookrightarrow \widetilde{R}_i$ and the cokernel of $R_i \hookrightarrow \widetilde{R}_i$ is finite. Putting this together with Lemma \ref{dis2} and Lemma \ref{cond}, it is straightforward to see that the hypotheses of the theorem are satisfied for the $G_{\mathbb{Q}}$-representation $\widetilde{T}_i$, and the assumption on the Galois elements $\sigma_1, \sigma_2$ in the theorem remains true for $\widetilde{T}_i$. It follows from the snake lemma that 
\begin{equation}
\label{EuSys00}
c=\big|R/(\epsilon^k)\big|=\big|\widetilde{R}/(\epsilon^k)\big|. 
\end{equation}

Let us prove $\rm(1)$. Consider a natural injection:
$$
T^*:=\Hom_R(T,R) \to \Hom_R(T,R)\otimes_R \widetilde{R} \to \widetilde{T}^*:=\Hom_{\widetilde{R}}(T \otimes_R \widetilde{R},\widetilde{R}).
$$
This induces a map of $R$-modules $\textcyr{Sh}^2_{\Sigma}(T^*(1)) \to \textcyr{Sh}^2_{\Sigma}(\widetilde{T}^*(1))$. Since the injection $R \hookrightarrow \widetilde{R}$ has finite cokernel, the above map has finite kernel and cokernel. By the result of Euler system bound over a discrete valuation ring $\widetilde{R}_i$ (see \cite[Theorem 4.7]{Oc1}), the $\widetilde{R}_i$-module $\textcyr{Sh}^2_{\Sigma}(\widetilde{T}^*_i(1))$ is a finite group. The same is true for $\textcyr{Sh}^2_{\Sigma}(\widetilde{T}^*(1))$. Then we see that $\textcyr{Sh}^2_{\Sigma}(T^*(1))$ is also finite. Thus, we have proved $\rm(1)$.

Let us prove $\rm(2)$. According to Euler system bound over a discrete valuation ring applied to $(\widetilde{T}_i,\widetilde{R}_i,\Sigma)$, we claim the following statement:
\begin{equation}
\label{EuSys1}
c_i \cdot |H^1(G_{\Sigma},\widetilde{T}^*_i(1))/\widetilde{R}_i \mathbf{z}_1|~\mbox{is divisible by}~|\textcyr{Sh}^2_{\Sigma}(\widetilde{T}^*_i(1))|,
\end{equation}
for each $i$, where $c_i:=|\widetilde{R}_i/(\epsilon^k)|$ and $k$ is as in the theorem. Then $(\ref{EuSys1})$ implies the following statement:
\begin{equation}
\label{EuSys2}
c \cdot |H^1(G_{\Sigma},\widetilde{T}^*(1))/\widetilde{R} \mathbf{z}_1|~\mbox{is divisible by}~|\textcyr{Sh}^2_{\Sigma}(\widetilde{T}^*(1))|,
\end{equation}
where $c=|R/(\epsilon^k)|=|\widetilde{R}/(\epsilon^k)|$ by $(\ref{EuSys00})$. So let us prove $(\ref{EuSys1})$. Indeed, in view of \cite[Theorem 4.7]{Oc1}, it suffices to prove that
\begin{equation}
\label{EuSys3}
\mu_{\widetilde{R}_i}\big(\textcyr{Sh}^2_{\Sigma}(\widetilde{T}_i^*(1))\big)=k.
\end{equation}
Since $H^2(G_{\mathbb{Q}_{\ell}},\widetilde{T}_i^*(1))=0$ for $\ell \in \Sigma\setminus\{\infty\}$, we have the following isomorphism by Lemma \ref{last}:
$$
\textcyr{Sh}^2_{\Sigma}(\widetilde{T}_i^*(1))/\varpi_i \textcyr{Sh}^2_{\Sigma}(\widetilde{T}_i^*(1)) \simeq \textcyr{Sh}^2_{\Sigma}((\widetilde{T}_i^*/\varpi_i \widetilde{T}_i^*)(1)),
$$
where $\varpi_i$ is a generator of the maximal ideal of $\widetilde{R}_i$. By Nakayama's lemma, we have
$$
\mu_{\widetilde{R}_i}\big(\textcyr{Sh}^2_{\Sigma}(\widetilde{T}_i^*(1))\big)=\dim_{\widetilde{R}_i/(\varpi_i)}\big(\textcyr{Sh}^2_{\Sigma}((\widetilde{T}^*_i/\varpi_i \widetilde{T}^*_i)(1))\big).
$$
Applying the field extension $R/\fm \to \widetilde{R}_i/(\varpi_i)$ to Lemma \ref{dis2}, it follows from $(\ref{TateShafarevich})$ that
$$
\dim_{\widetilde{R}_i/(\varpi_i)}\big(\textcyr{Sh}^2_{\Sigma}((\widetilde{T}^*_i/\varpi_i \widetilde{T}^*_i)(1))\big)=\dim_{R/\fm}\big(\textcyr{Sh}^2_{\Sigma}((T^*/\fm T^*)(1))\big).
$$
By the hypothesis $\rm(iii)$ and Lemma \ref{last}, we get
$$
\textcyr{Sh}^2_{\Sigma}(T^*(1))/\fm \textcyr{Sh}^2_{\Sigma}(T^*(1)) \simeq \textcyr{Sh}^2_{\Sigma}((T^*/\fm T^*)(1)).
$$
Putting all above equations together, $(\ref{EuSys3})$ (and thus $(\ref{EuSys1})$) follows, as required. 

The rest of the proof will be devoted to proving the equality:
\begin{equation}
\label{EuSys4}
|\textcyr{Sh}^2_{\Sigma}(\widetilde{T}^*(1))|=r \cdot |\textcyr{Sh}^2_{\Sigma}(T^*(1))|,
\end{equation}
where 
\begin{equation}
\label{number1}
r:=|\coker[H^1(G_{\Sigma},T^*(1)) \to H^1(G_{\Sigma},\widetilde{T}^*(1))]|.
\end{equation}
Applying the functor $(-)^{\PD}(1)$ to the exact sequence $0 \to T^*(1) \to \widetilde{T}^*(1) \to \mathcal{C} \to 0$ (the cokernel $\mathcal{C}$ is a finite $p$-group), we get an exact sequence of discrete modules: $0 \to \mathcal{C}^{\PD}(1) \to \widetilde{D} \to D \to 0$. We claim that this induces the commutative diagram with exact rows:
$$
\footnotesize{
\begin{CD}
0 @>>> H^1(G_{\Sigma},\mathcal{C}^{\PD}(1)) @>\phi_1>> H^1(G_{\Sigma},\widetilde{D})  @>>> H^1(G_{\Sigma},D) @>\psi_1>> H^2(G_{\Sigma},\mathcal{C}^{\PD}(1)) \\
@. @V\loc^1_{\mathcal{C}^{\PD}(1)}VV @V\loc^1_{\widetilde{D}} VV @V\loc^1_{D}VV @V\loc^2_{\mathcal{C}^{\PD}(1)}VV \\
0 @>>> \displaystyle{\bigoplus_{\ell \in \Sigma}} H^1(G_{\mathbb{Q}_\ell},\mathcal{C}^{\PD}(1)) @>\phi_2>> \displaystyle{\bigoplus_{\ell \in \Sigma}} H^1(G_{\mathbb{Q}_\ell},\widetilde{D}) @>>> \displaystyle{\bigoplus_{\ell \in \Sigma}} H^1(G_{\mathbb{Q}_\ell},D) @>\psi_2>> \displaystyle{\bigoplus_{\ell \in \Sigma}} H^2(G_{\mathbb{Q}_\ell},\mathcal{C}^{\PD}(1)) \\ 
\end{CD}}
$$

Henceforth, we denote by $\mathcal{C}_1 \xrightarrow{\phi} \mathcal{C}_2$ the map $\coker(\phi_1) \to \coker(\phi_2)$ (resp. by $\mathcal{C}'_1 \xrightarrow{\phi'} \mathcal{C}'_2$ the map $\im(\psi_1) \to \im(\psi_2)$). Let us check the exactness of the above commutative diagram. We can verify that $\phi_1$ is injective as follows. Since $(T^*(1)/\fm T^*(1))_{G_{\Sigma}}=0$ in view of the hypothesis $\rm(i)$ and
$\rank_R T =2$, Nakayama's lemma implies that $(T^*(1))_{G_{\Sigma}}=0$, and thus $\phi_1$ is injective. Also, since $p>2$ and $\mathcal{C}^{\PD}(1)$ is a finite $p$-group, $H^1(G_{\mathbb{R}},\mathcal{C}^{\PD}(1))=0$. Thus, $\phi_2$ is injective by the hypothesis (iii). For later use, we prove the following:
\begin{equation}
\label{surj1}
\loc^2_{\mathcal{C}^{\PD}(1)}~\mbox{is surjective}.
\end{equation}
Indeed, the cokernel of $\loc^2_{\mathcal{C}^{\PD}(1)}$ is equal to $H^0(G_{\Sigma},\mathcal{C})^{\PD}$ by Poitou-Tate duality. To prove that this group vanishes, it suffices to prove that $H^0(G_{\Sigma},\mathcal{C})=0$. By the exact sequence,
$$
\cdots \to H^0(G_{\Sigma},\widetilde{T}^*(1)) \to H^0(G_{\Sigma},\mathcal{C}) \to H^1(G_{\Sigma},T^*(1)) \to \cdots
$$
we know that $H^0(G_{\Sigma},\widetilde{T}^*(1))$ vanishes  by the hypothesis (i) and $H^1(G_{\Sigma},T^*(1))$ is $R$-torsion free by Lemma \ref{dis1}. Since  $H^0(G_{\Sigma},\mathcal{C})$ is finite, it vanishes. Thus, $\loc^2_{\mathcal{C}^{\PD}(1)}$ is surjective. 

Next, we prove the following:
\begin{equation}
\label{surj2}
\textcyr{Sh}^2_{\Sigma}(\widetilde{D})=0.
\end{equation}
Let us recall that $\ker(\loc^2_{\widetilde{D}})=\textcyr{Sh}^2_{\Sigma}(\widetilde{D})$, where $\loc^2_{\widetilde{D}}:H^2(G_{\Sigma},\widetilde{D}) \to \bigoplus_{\ell \in \Sigma} H^2(G_{\mathbb{Q}_{\ell}},\widetilde{D})$. The Pontryagin dual of $\textcyr{Sh}^2_{\Sigma}(\widetilde{D})$, which is $\textcyr{Sh}^1_{\Sigma}(\widetilde{T}^*(1))$, injects into $H^1(G_{\Sigma},\widetilde{T}^*(1))$ by definition. Then $H^1(G_{\Sigma},\widetilde{T}^*(1))$ is torsion free by Lemma \ref{dis1} and the hypothesis (i). Then, $\textcyr{Sh}^1_{\Sigma}(\widetilde{T}^*(1))$ is also torsion free, while the hypothesis (iii) implies that $\textcyr{Sh}^2_{\Sigma}(\widetilde{D})$ (say $\textcyr{Sh}^1_{\Sigma}(\widetilde{T}^*(1))$) is finite. Hence $\textcyr{Sh}^2_{\Sigma}(\widetilde{D})=0$.

Applying the snake lemma to the commutative diagram whose rows are exact:
$$
\begin{CD}
0 @>>> H^1(G_{\Sigma},\mathcal{C}^{\PD}(1)) @>\phi_1>> H^1(G_{\Sigma},\widetilde{D})  @>>> \mathcal{C}_1@>>> 0\\
@. @V\loc^1_{\mathcal{C}^{\PD}(1)}VV @V\loc^1_{\widetilde{D}}VV @V\phi VV \\
0 @>>> \displaystyle{\bigoplus_{\ell \in \Sigma}} H^1(G_{\mathbb{Q}_\ell},\mathcal{C}^{\PD}(1)) @>\phi_2>> \displaystyle{\bigoplus_{\ell \in \Sigma}} H^1(G_{\mathbb{Q}_\ell},\widetilde{D}) @>>> \mathcal{C}_2 @>>> 0\\ 
\end{CD}
$$
we obtain
\begin{equation}
\label{EuSys5}
\frac{|\textcyr{Sh}^1_{\Sigma}(\widetilde{D})|}{r_1}=\frac{|H^1(G_{\Sigma},\mathcal{C}^{\PD}(1))|}{|\bigoplus_{\ell \in \Sigma} H^1(G_{\mathbb{Q}_{\ell}},\mathcal{C}^{\PD}(1))|} \cdot |\ker(\phi)|,
\end{equation}
where 
$$
r_1:=|\ker[\coker(\loc^1_{\widetilde{D}}) \to \coker(\phi)]|.
$$
Next, consider the following commutative diagram whose rows are exact:
$$
\begin{CD}
0 @>>> \mathcal{C}'_1 @>>> H^2(G_{\Sigma},\mathcal{C}^{\PD}(1)) @>>> H^2(G_{\Sigma},\widetilde{D})  \\
@. @V\phi' VV @V\loc^2_{\mathcal{C}^{\PD}(1)}VV @V\loc^2_{\widetilde{D}}VV \\
0 @>>> \mathcal{C}'_2@>>> \displaystyle{\bigoplus_{\ell \in \Sigma}} H^2(G_{\mathbb{Q}_\ell},\mathcal{C}^{\PD}(1))  @>>> \displaystyle{\bigoplus_{\ell \in \Sigma}} H^2(G_{\mathbb{Q}_\ell},\widetilde{D}) \\ 
\end{CD}
$$
Since $\loc^2_{\widetilde{D}}$ is injective by $(\ref{surj2})$, we have
$\ker(\phi')=\ker(\loc^2_{\mathcal{C}^{\PD}(1)})$. On the other hand, it follows from the injectivity of $\loc^2_{\widetilde{D}}$ and the surjectivity of $\loc^2_{\mathcal{C}^{\PD}(1)}$ by $(\ref{surj1})$ that
\begin{equation}
\label{EuSys666}
\phi'~\mbox{is surjective}.
\end{equation}
Moreover, we have
\begin{equation}
\label{EuSys6666}
\frac{|\mathcal{C}'_2|}{|\mathcal{C}'_1|}=\frac{|\bigoplus_{\ell \in \Sigma} H^2(G_{\mathbb{Q}_\ell},\mathcal{C}^{\PD}(1))|}{|H^2(G_{\Sigma},\mathcal{C}^{\PD}(1))|}.
\end{equation}
Then applying the snake lemma together with $(\ref{EuSys666})$ to the commutative diagram with exact rows: 
$$
\begin{CD}
0 @>>> \mathcal{C}_1 @>>> H^1(G_{\Sigma},D) @>\psi_1>> \mathcal{C}'_1 @>>> 0 \\
@. @V\phi VV @V\loc^1_{D}VV @V\phi'VV \\
0 @>>> \mathcal{C}_2@>\phi_2>> \displaystyle{\bigoplus_{\ell \in \Sigma}} H^1(G_{\mathbb{Q}_\ell},D) @>\psi_2>> \mathcal{C}'_2 @>>> 0\\ 
\end{CD}
$$
we obtain
\begin{equation}
\label{EuSys66}
\frac{1}{|\textcyr{Sh}^1_{\Sigma}(D)|}=\frac{|\mathcal{C}'_2|} {|\mathcal{C}'_1|} \cdot \frac{1}{|\ker(\phi)|} \cdot \frac{|\coker(\phi)|}{|\coker(\loc^1_D)|}= \frac{|\mathcal{C}'_2|} {|\mathcal{C}'_1|} \cdot \frac{r_2}{|\ker(\phi)|},
\end{equation}
where
$$
r_2:=|\ker[\coker(\phi) \to \coker(\loc^1_{D})]|.
$$
Now putting together $(\ref{EuSys6666})$ and $(\ref{EuSys66})$, we get
\begin{equation}
\label{EuSys6}
\frac{1}{|\textcyr{Sh}^1_{\Sigma}(D)|}=\frac{|\bigoplus_{\ell \in \Sigma} H^2(G_{\mathbb{Q}_\ell},\mathcal{C}^{\PD}(1))|} {|H^2(G_{\Sigma},\mathcal{C}^{\PD}(1))|} \cdot \frac{r_2}{|\ker(\phi)|},
\end{equation}

Putting together $(\ref{EuSys5})$ and ($\ref{EuSys6})$ and applying the Poitou-Tate duality theorem, we get the following equality:
$$
\frac{|\textcyr{Sh}^2_{\Sigma}(\widetilde{T}^*(1))|}{|\textcyr{Sh}^2_{\Sigma}(T^*(1))|}=\frac{|H^1(G_{\Sigma},\mathcal{C}^{\PD}(1))|}{| H^2(G_{\Sigma},\mathcal{C}^{\PD}(1))|} \cdot \frac{|\bigoplus_{\ell \in \Sigma} H^2(G_{\mathbb{Q}_\ell},\mathcal{C}^{\PD}(1))|}{|\bigoplus_{\ell \in \Sigma} H^1(G_{\mathbb{Q}_\ell},\mathcal{C}^{\PD}(1))|} \cdot r_1 \cdot r_2.
$$
Note that $H^0(G_{\Sigma},\mathcal{C}^{\PD}(1))=0$ by $H^0(G_{\Sigma},\widetilde{D})=0$. Then by the Euler-Poincar\'e characteristic formula $(\ref{EulerPo1})$ and $(\ref{EulerPo2})$ in Theorem \ref{Tate}, treating cases according as $\ell=p$, $\ell \in \Sigma \setminus \{p,\infty\}$, or $\ell=\infty$, respectively, we obtain the following equality:
$$
|\textcyr{Sh}^2_{\Sigma}(\widetilde{T}^*(1))|=|\textcyr{Sh}^2_{\Sigma}(T^*(1))| \cdot r_1 \cdot r_2.
$$
It remains to show that $r_1 \cdot r_2$ is equal to $r$ which is given in $(\ref{number1})$. For its proof, consider the commutative diagram induced by Poitou-Tate duality, whose rows are exact:
$$
\begin{CD}
0 @>>> \coker(\loc^1_{\widetilde{D}}) @>>> H^1(G_{\Sigma},\widetilde{T}^*(1))^{\PD} @>>> \textcyr{Sh}^2_{\Sigma}(\widetilde{D}) @>>> 0\\
@. @VV\tau_1V @VV\tau_2V @VVV \\
0 @>>> \coker(\loc^1_{D}) @>>> H^1(G_{\Sigma},T^*(1))^{\PD} @>>> \textcyr{Sh}^2_{\Sigma}(D) @>>> 0\\
\end{CD}
$$
We claim that $r=|\ker(\tau_2)|$. Since $r_1 \cdot r_2=|\ker(\tau_1)|$, this is equivalent to showing that the induced injective map $\ker(\tau_1) \to \ker(\tau_2)$ is surjective. However, since $\textcyr{Sh}^2_{\Sigma}(\widetilde{D})=0$ by $(\ref{surj2})$, the equality $r=r_1 \cdot r_2$ follows. So we have proved ($\ref{EuSys4}$).

Finally in view of the hypothesis (ii), we claim that
\begin{equation}
\label{EuSys7}
|H^1(G_{\Sigma},\widetilde{T}^*(1))/\widetilde{R} \mathbf{z}_1|=\frac{r}{|\widetilde{R}/R|} \cdot |H^1(G_{\Sigma},T^*(1))/R \mathbf{z}_1|,
\end{equation}
and all groups are finite. To see this, consider the commutative diagram of $R$-modules with exact rows:
$$
\begin{CD}
0 @>>> R @>>> \widetilde{R} @>>> \widetilde{R}/R @>>> 0 \\
@. @V\times \mathbf{z}_1VV @V\times \mathbf{z}_1VV @V\times 0VV \\
0@>>> H^1(G_{\Sigma},T^*(1)) @>>> H^1(G_{\Sigma},\widetilde{T}^*(1)) @>>> \mathcal{M} @>>> 0 \\
\end{CD}
$$
We have $r=|\mathcal{M}|$ by definition and $(\ref{EuSys7})$ follows from the snake lemma. Now, incorporating both $(\ref{EuSys4})$ and $(\ref{EuSys7})$ into $(\ref{EuSys2})$, the required Euler system bound for $(T,R,\Sigma)$ is obtained. This finishes the proof of the theorem.\\
\end{proof}

\subsection{Euler system bound over a Cohen-Macaulay normal domain}

In this subsection, we prove an Euler system bound over Cohen-Macaulay normal complete local domains with Krull dimension at least two using Theorem \ref{Euler}. Later in \S \ref{proof2}, we shall show one of divisibilities of the Iwasawa main conjecture for two-variable Hida deformations by applying this Euler system bound. The main tool in the proof of the Euler system bound is the specialization methods that we have developed so far. We mention several facts on the behavior of Galois cohomology and characteristic ideals under etale extensions. 

Let $(R,\fm,\mathbb{F})$ be a Noetherian complete local normal domain, whose coefficient ring is $W(\mathbb{F})$. 
For a finite field extension $\mathbb{F} \to \mathbb{F}'$, we define $R_{W(\mathbb{F}')}:=R \otimes_{W(\mathbb{F})} W(\mathbb{F}')$.  
Then we have the following. 
\begin{enumerate}
\item
Let $M$ be a finitely generated $R$-module with continuous $G$-action. Then we have 
\begin{equation}
\label{controlGal}
H^i(G,M) \otimes_R R_{W(\mathbb{F}')} \simeq H^i(G,M \otimes_R R_{W(\mathbb{F}')}), 
\end{equation}
\begin{equation}
\label{controlSha}
\textcyr{Sh}^2_{\Sigma}(M) \otimes_R R_{W(\mathbb{F}')} \simeq \textcyr{Sh}^2_{\Sigma}(M \otimes_R R_{W(\mathbb{F}')}), 
\end{equation}
which follow from Lemma \ref{dis2} and the definition of Tate-Shafarevich group. 
\item
The following equality between reflexive ideals holds:
\begin{equation}
\label{controlchar}
\Char_R(M) R_{W(\mathbb{F}')}=\Char_{R_{W(\mathbb{F}')}}(M \otimes_R R_{W(\mathbb{F}')}), 
\end{equation}
which is easy to check from the definition.
\end{enumerate}

\begin{lemma}
\label{topological}
Let $(R,\fm,\mathbb{F})$ be a two-dimensional Noetherian complete local normal domain with mixed characteristic $p>0$ and finite residue field. Let $u \in R^{\times}$ be an element of infinite order. Then the following statements hold.  
\begin{enumerate}
\item[\rm{(1)}] 
Let us choose $z \in \fm$ such that $(p,z)$ is a parameter ideal of $R$. 
Then, for each $n>0$, $R/(z+a_np^n)$ is a reduced ring of mixed characteristic, and 
the image of $u \in R^{\times}$ under the reduction map
$$
R^{\times} \to \big(R/(z+a_np^n)\big)^{\times}
$$
is of infinite order for infinitely many different choices $a_n \in W(\mathbb{F})^{\times}$.

\item[\rm{(2)}]
Let us choose $z, r \in \fm$ such that $(p,z^n+r)$ is a parameter ideal of $R$ for all $n>0$. Then, for each $n>0$, $R/(z^n+r+a_n p)$ is a reduced ring of mixed characteristic, and the image of $u \in R^{\times}$ under the reduction map
$$
R^{\times} \to \big(R/(z^n+r+a_np^n)\big)^{\times}
$$
is of infinite order for infinitely many different choices $a_n \in W(\mathbb{F})^{\times}$.
\end{enumerate}
\end{lemma}

\begin{proof}
Since the proofs of the assertions $\rm(1)$ and $\rm(2)$ proceed almost in the same manner, we only prove $\rm(2)$. Let us fix an integer $n>0$ and consider a module-finite extension $W(\mathbb{F})[[z^n+r]] \to R$. By Corollary \ref{LocalBertini} $\rm(2)$, there is a finite subset $S_n \subseteq W(\mathbb{F})^{\times}$ such that $R/(z+r+a_n p^n)$ is reduced of mixed characteristic for every $a_n \in W(\mathbb{F})^{\times} \setminus S_n$. Let us put
$$
\mathcal{U}_n:=\{P \in \Spec R~|~(u^k-1) \subseteq P~\mbox{for some}~k~\mbox{and}~\Ht(P)=1\}.
$$
Let $\mathcal{V}_n$ denote the image of $\mathcal{U}_n$ under the map $\Spec R \to \Spec W(\mathbb{F}) [[z^n+r]]$. Then we can check that $\mathcal{V}_n$ is a countable set. Let us put
$$
T_n:=\{a_n \in W(\mathbb{F})^{\times}~|~(z^n+r+a_np) \in \mathcal{V}_n\}.
$$
Notice that if we choose $a_n \ne a'_n$, then
$$
(z^n+r+a_n p)~\mbox{and}~(z^n+r+a'_n p)
$$
are distinct prime ideals of $W(\mathbb{F})[[z^n+r]]$. Moreover, $W(\mathbb{F})^{\times}$ is an \textit{uncountable} set (for this, one uses the fact that elements of $W(\mathbb{F})$ are written as $p$-series) and it follows that $W(\mathbb{F})^{\times} \setminus S_n \cup T_n$ is uncountable. In particular, it is infinite. Now we conclude that the ring $R/(z^n+r+a_np)$ is reduced of mixed characteristic and the image of $u$ under $R^{\times} \to \big(R/(z^n+r+a_np)\big)^{\times}$ is of infinite order for every $a_n \in W(\mathbb{F})^{\times} \setminus S_n \cup T_n$. This completes the proof of the lemma.
\end{proof}

We shall prove the following theorem via reduction to Theorem \ref{Euler}.

\begin{theorem}\label{veryfinal}
Let $(R,\fm,\mathbb{F})$ be a Noetherian complete local Cohen-Macaulay normal domain 
of Krull dimension $d\geq 2$ with mixed characteristic $p>2$ and finite residue field $\mathbb{F}$. Suppose that
$$
\big\{\mathbf{z}_n \in H^1(G_{\Sigma,n},T^*(1))\big\}_{n \in \mathfrak{N}}
$$ 
is an Euler system for $(T,R,\Sigma)$, $T$ is a free $R$-module of rank two with 
continuous $G_{\mathbb{Q}}$-action and suppose that the following conditions hold:
\begin{enumerate}
\item[\rm{(i)}]
$T \otimes_R R/\fm$ is absolutely irreducible as a representation of $G_{\mathbb{Q}}$.

\item[\rm{(ii)}]
The quotient $H^1(G_{\Sigma},T^*(1))/R \mathbf{z}_1$ is an $R$-torsion module.

\item[\rm{(iii)}]
$H^2(G_{\mathbb{Q}_{\ell}},T^*(1))=0$ for every $\ell \in \Sigma\setminus\{\infty\}$ and $H^2(G_{\Sigma},D)$ is a finite group.

\item[\rm{(iv)}]
The determinant character $\wedge^2 \rho:G_{\mathbb{Q}} \to R^{\times}$ (resp. $\wedge^2 \rho^*(1):G_{\mathbb{Q}} \to R^{\times}$) associated with the $G_{\mathbb{Q}}$-representation $T$ (resp. $T^*(1)$) has an element of infinite order.

\item[\rm{(v)}]
The $R$-module $T$ splits into eigenspaces: $T=T^+ \oplus T^-$ with respect to the complex conjugation in $G_{\mathbb{Q}}$, and $T^+$ (resp. $T^-$) is of $R$-rank one.

\item[\rm{(vi)}]
There exist $\sigma_1 \in G_{\mathbb{Q}(\mu_{p^{\infty}})}$ and $\sigma_2 \in G_{\mathbb{Q}}$ such that $\rho(\sigma_1) \simeq \begin{pmatrix} 1 & P \\ 0 & 1 \end{pmatrix} \in GL_2(R)$ for a nonzero element $P \in R$ and $\sigma_2$ acts on $T$ as multiplication by $-1$. 

\end{enumerate}
Then Euler system bound holds for $(T,R,\Sigma)$:
\begin{enumerate}
\item[\rm{(1)}]
$\textcyr{Sh}^2_{\Sigma}(T^*(1))$ is a finitely generated torsion $R$-module.

\item[\rm{(2)}]
We have an inclusion of reflexive ideals:
\begin{equation}
\label{inclusionchar}
(P^k) \Char_R\big(H^1 (G_{\Sigma},T^*(1))/R \mathbf{z}_1\big) \subseteq \Char_R\big(\textcyr{Sh}^2_{\Sigma} (T^*(1))\big),
\end{equation}
where $k$ is the number of minimal generators of $\textcyr{Sh}^2_{\Sigma}(T^*(1))$ as an $R$-module.
\end{enumerate}
\end{theorem}

\begin{proof}
Let us put $d:=\dim R$. Let $\mathcal{S}$ be a set of ideals of $R$ (resp. $R_{W(\overline{\mathbb{F}})}$) which consists of certain nonzero principal ideals 
(resp. certain nonzero principal \textit{prime} ideals) when $d=2$ (resp. $d\geq 3$). Later, we make more precise the definition of $\mathcal{S}$. Then we want to define the \textit{subset} $\mathcal{S}_{\rm{(i)}}$ (resp. $\mathcal{S}_{\rm{(ii)}}$, $\mathcal{S}_{\rm{(iii)}}$, 
$\mathcal{S}_{\rm{(iv)}}$, $\mathcal{S}_{\rm{(v)}}$, $\mathcal{S}_{\rm{(vi)}}$) of $\mathcal{S}$ 
to be the set consisting of height-one ideals $\mathfrak{a} \in \mathcal{S}$ for which the condition $\rm(i)$ (resp. $\rm(ii)$, $\rm(iii)$, $\rm(iv)$, $\rm(v)$, $\rm(vi)$) does not hold, after replacing the quantities: $T$, $T^* (1)$, $R$, $P$ and $\mathbf{z}_1$ with their respective quotients 
$T/\mathfrak{a}T$, $(T^*/\mathfrak{a}T^*) (1)$, $R/\mathfrak{a}$, $\overline{P}:=P \pmod{\fa} \in R/\mathfrak{a}$ and $\overline{\mathbf{z}}_1$ which is the image of $\mathbf{z}_1$ in $H^1(G_{\Sigma},(T^*/\mathfrak{a}T^*)(1))$. We put
$$
\mathcal{S}_{\rm{bad}}:=\mathcal{S}_{\rm{(i)}} \cup \mathcal{S}_{\rm{(ii)}} \cup \mathcal{S}_{\rm{(iii)}} \cup \mathcal{S}_{\rm{(iv)}} \cup \mathcal{S}_{\rm{(v)}} \cup \mathcal{S}_{\rm{(vi)}}. 
$$

We prove the assertion $\rm(1)$ and the formula $(\ref{inclusionchar})$ in the assertion $\rm(2)$ simultaneously. Let us point out that $\rm(1)$ is required to prove $\rm(2)$ theoretically. However, the proof of both assertions requires that $\mathcal{S} \setminus \mathcal{S}_{\rm{bad}}$ is large enough. In fact, to prove the assertion $\rm{(1)}$, we need to check that $\mathcal{S} \setminus \mathcal{S}_{\rm{bad}}$ is not empty. To prove the assertion $\rm{(2)}$, we need to check that $\mathcal{S} \setminus \mathcal{S}_{\rm{bad}}$ is infinite and large enough to apply Theorem \ref{prop1} and Theorem \ref{prop2} in the case $d=2$ and Theorem \ref{theorem:previous} in the case $d \ge 3$, respectively. Let us proceed by induction with respect to the Krull dimension $d \geq 2$ of $(R,\fm,\mathbb{F})$. As the initial step for induction of the proofs of assertions $\rm(1)$ and $\rm(2)$, we start from the proof of the case $d=2$.
\\

First, we prove the assertion $\rm(1)$ when $d=2$, assuming that $\mathcal{S} \setminus \mathcal{S}_{\rm{bad}}$ is not empty. Let $\fa \in \mathcal{S} \setminus \mathcal{S}_{\rm{bad}}$. Then $\fa$ is a nonzero principal ideal and
the $(R/\fa)[G_{\mathbb{Q}}]$-representation $T/\fa T$ satisfies the hypotheses \rm{(i)} through \rm{(vi)} of Theorem \ref{Euler}, which in particular implies that $\textcyr{Sh}^2_{\Sigma}((T^*/\fa T^*)(1))$ is finite. By Lemma \ref{last} applied to the map $\textcyr{Sh}^2_{\Sigma}(T^*(1))/\fa \textcyr{Sh}^2_{\Sigma}(T^*(1)) \to \textcyr{Sh}^2_{\Sigma}((T^*/\fa T^*)(1))$, and by Nakayama's lemma, it follows by the finiteness of $\textcyr{Sh}^2_{\Sigma}((T^*/\fa T^*)(1))$ due to
Theorem \ref{Euler}, that $\textcyr{Sh}^2_{\Sigma}(T^*(1))$ is a finitely generated torsion $R$-module. 

Next, we prove the assertion $(\ref{inclusionchar})$ when $d=2$. We define
$\mathcal{S}$ to be the set of height-one ideals $\{(\mathbf{x}_n) \subseteq R\}_{n \in \mathbb{N}}$, where we take $\mathbf{x}_n=z+a_np^n$ as in Lemma \ref{topological} $\rm(1)$ and $(\ref{linearideal})$ for the choice $z \in R$, and where we take $\mathbf{x}_n=z^n+r+a_n p$ as in Lemma \ref{topological} $\rm(2)$ and $(\ref{linearideal2})$ for the choice $z,r \in R$.

We set
\begin{align*}
& M:=\textcyr{Sh}^2_{\Sigma} (T^*(1)), \\
& N:=R/(P^k) \oplus  H^1 (G_{\Sigma},T^*(1))/R \mathbf{z}_1.
\end{align*}
The inclusion \eqref{inclusionchar} holds if and only if 
\begin{equation}\label{inclusionchar_p}
\Char_R (N)_\fp \subseteq \Char_R (M)_\fp~\mbox{holds for every height one prime}~\fp~\mbox{of}~R. 
\end{equation}
The proof of \eqref{inclusionchar_p} is divided into the case when $\fp$ is not lying over $p$ and the case when $\fp$ is lying over $p$. 
Theorem \ref{prop1} plays a role in the former case and Theorem \ref{prop2} plays a role in the latter case. Since the proof goes in the same way in both cases, we shall concentrate on the former case.

Since we are assuming $d=2$, the maximal pseudo-null submodule of $H^2(G_{\Sigma},T^*(1))$ is finite. Let the \textit{error term} $c$ equal the cardinality of the maximal pseudo-null submodule of $H^2(G_{\Sigma},T^*(1))$. By applying the Euler system bound of Theorem \ref{Euler} for the $(R/\mathfrak{a})[G_{\mathbb{Q}}]$-module $T/\mathfrak{a}T$, we find that
\begin{equation}
\label{reduce1}
c \cdot \frac{|N/\mathfrak{a}N|}{|M/\mathfrak{a}M|} \in \mathbb{N}
\end{equation}
for all $\mathfrak{a} \in \mathcal{S} \setminus \mathcal{S}_{\rm{bad}}$. Once we prove that the set $\mathcal{S}_{\rm{bad}}$ is finite, Theorem \ref{prop1} applies to deduce \eqref{inclusionchar_p} for every height-one prime $\fp$ of $R$ which is not lying over $p$. Similarly, the finiteness of $\mathcal{S}_{\rm{bad}}$ and Theorem \ref{prop2} allow us to deduce \eqref{inclusionchar_p} for every height-one prime $\fp$ of $R$ which is lying over $p$. 
Hence the proof of \eqref{inclusionchar} will be completed when $d=2$. Thus, in order to prove the assertion $\rm(2)$ for the case $d=2$, it remains to show that $\mathcal{S}_{\rm{bad}}$ is a finite set.

First, we note that if the condition $\rm(i)$ (resp. $\rm(v)$) holds for $R[G_{\mathbb{Q}}]$-module $T$, 
the condition $\rm(i)$ (resp. $\rm(v)$) holds for the $(R/\mathfrak{a})[G_{\mathbb{Q}}]$-module $T/\mathfrak{a}T$ and for all $\mathfrak{a} \in \mathcal{S}$. So we get $\mathcal{S}_{\rm{(i)}}=\mathcal{S}_{\rm{(v)}}=\emptyset$. It remains to check that $\mathcal{S}_{\rm{(ii)}}$, $\mathcal{S}_{\rm{(iii)}}$, $\mathcal{S}_{\rm{(iv)}}$, $\mathcal{S}_{\rm{(vi)}}$ are finite.

$\mathbf{Finiteness~of}$ $\mathcal{S}_{\rm{(ii)}}$: We need to find an ideal $\mathfrak{a} \in \mathcal{S}$ such that
\begin{equation}
\label{finitegroup}
\big|H^1(G_{\Sigma},(T^*/\fa T^*)(1))/(R/\mathfrak{a}) \mathbf{z}_1\big|< \infty.
\end{equation}
Let us put $(f)=\mathfrak{a}$ and write $R \xrightarrow{\times f} R$ for the multiplication map. Then the short exact sequence $0 \to R \xrightarrow{\times f} R \to R/\mathfrak{a} \to 0$ yields the long exact sequence:
$$
\cdots \to H^1(G_{\Sigma},T^*(1)) \xrightarrow{\times f} H^1(G_{\Sigma},T^*(1)) \to H^1(G_{\Sigma},(T^*/\fa T^*)(1)) \to H^2(G_{\Sigma},T^*(1)) \xrightarrow{\times f} \cdots
$$
For simplicity, we put
$$
N_i:=H^i(G_{\Sigma},T^*(1))~\mbox{and}~J_i:=\Ann_R(N_i/R \mathbf{z}_1),
$$
which is the annihilator of the $R$-module $N_i/R \mathbf{z}_1$. 

Let us consider $N_1$. By hypothesis $\rm(ii)$, $\Supp_R(N_1/R \mathbf{z}_1)=\Spec (R/J_1)$ is a proper closed subset of $\Spec R$. First, we suppose that $\Spec(R/J_1) \subseteq \{\fm\}$. Then we choose any nonzero element $f \in \fm$. Second, we suppose that $\Spec(R/J_1)=\{\fp_1,\ldots,\fp_n,\fm\}$. Then we choose $f \in \fm$ such that $f \notin \bigcup_{i=1}^n \fp_i$.

Let us consider $N_2$, in which the image of $\mathbf{z}_1$ is zero and we have $\Supp_R(N_2)=\Spec(R/J_2)$. It suffices to choose $f \in \fm$ such that $\ker\big[ N_2 \xrightarrow{\times f} N_2\big] $ is finite. First, we suppose that $\Spec(R/J_2) \subseteq \{\fm\}$. Then we choose any nonzero element $f \in \fm$. Second, we suppose that $\Spec(R/J_2)=\{\fq_1,\ldots,\fq_{n'},\fm\}$. Then we choose $f \in \fm$ such that $f \notin \bigcup_{i=1}^{n'} \fq_i$. Finally, suppose that $\Spec(R/J_2)=\Spec R$. Since $R$ is assumed to be an integral domain, this implies that $J_2=(0)$. That is, $N_2$ is a torsion-free $R$-module. Then we choose any nonzero element $f \in \fm$. In all cases considered above, $\ker\big[ N_2 \xrightarrow{\times f} N_2\big] $ is finite. Let $\fa \in \mathcal{S}$. Then
$$
\fa \in \mathcal{S}_{\rm{(ii)}} \iff \fa \subseteq \bigcup_{i=1}^n \fp_i \cup \bigcup_{i=1}^{n'} \fq_i,
$$ 
which shows that $\mathcal{S}_{\rm{(ii)}}$ is finite (see Remark \ref{heightoneprime}). So if $\fa \in \mathcal{S} \setminus \mathcal{S}_{\rm{(ii)}}$,
then $(\ref{finitegroup})$ holds, as required.

$\mathbf{Finiteness~of}$ $\mathcal{S}_{\rm{(iii)}}$: We prove that the finiteness of $H^2(G_{\Sigma},D[\mathfrak{a}])$ follows from Lemma \ref{last2}. Indeed, by Poitou-Tate duality and the hypothesis $|H^2(G_{\Sigma},D)| < \infty$, the group $H^2(G_{\mathbb{Q}_{\ell}},D)$ is finite for every $\ell \in \Sigma$. The local Tate duality shows that $T^*(1)^{G_{\mathbb{Q}_{\ell}}}$ is finite. Since $T^*(1)^{G_{\mathbb{Q}_{\ell}}}$ is $R$-torsion free, it is finite only when $T^*(1)^{G_{\mathbb{Q}_{\ell}}}=0$. Thus, the assumptions of Lemma \ref{last2} are satisfied and we find that $H^2(G_{\Sigma},D[\mathfrak{a}])$ is finite. Next, we prove that $H^2(G_{\mathbb{Q}_{\ell}},(T^*/\fa T^*)(1))=0$ for $\ell \in \Sigma\setminus\{\infty\}$. Notice that the second local Galois cohomology group is exactly  controlled. That is, we have  $H^2(G_{\mathbb{Q}_{\ell}},T^*(1))/\mathfrak{a} H^2(G_{\mathbb{Q}_{\ell}},T^*(1)) \simeq H^2(G_{\mathbb{Q}_{\ell}},(T^*/\fa T^*)(1))$. So the required vanishing follows from the assumption $H^2(G_{\mathbb{Q}_{\ell}},T^*(1))=0$. Hence we have $\mathcal{S}_{\rm{(iii)}}=\emptyset$.

$\mathbf{Finiteness~of}$ $\mathcal{S}_{\rm{(iv)}}$: Since $\wedge^2 \rho:G_{\mathbb{Q}} \to R^{\times}$ (resp. $\wedge^2 \rho^*(1):G_{\mathbb{Q}} \to R^{\times}$) is of infinite order, the image contains an element $u \in R^\times$ (resp. $u' \in R^\times$) of infinite order. Then by Lemma \ref{topological} $\rm(1)$, the image of $u$ (resp. $u'$) under $R^{\times} \to \big(R/\mathfrak{a} \big)^{\times}$ is of infinite order for every $n \in \mathbb{N}$ and $\mathfrak{a}=(z+a_n p^n) \in \mathcal{S}$. Hence we have $\mathcal{S}_{\rm{(iv)}}=\emptyset$.

$\mathbf{Finiteness~of}$ $\mathcal{S}_{\rm{(vi)}}$: It suffices to choose $(f)=\mathfrak{a} \subseteq R$ such that $\mathfrak{a}$ is not contained in $\fp$, where $\fp \in \Min_R(R/(P))$ and  $P \in R$ is given as in $\rm(vi)$. Let
$\fa \in \mathcal{S}$. Then
$$
\mathfrak{a} \in \mathcal{S}_{\rm{(vi)}} \iff \mathfrak{a} \subseteq \fp~\mbox{for some}~\fp \in \Min_R(R/(P)),
$$
which shows that $\mathcal{S}_{\rm{(vi)}}$ is finite (see Remark \ref{heightoneprime}).

Now $\mathcal{S}_{\rm{bad}}=\mathcal{S}_{\rm{(i)}} \cup \mathcal{S}_{\rm{(ii)}} \cup \mathcal{S}_{\rm{(iii)}} \cup \mathcal{S}_{\rm{(iv)}} \cup \mathcal{S}_{\rm{(v)}} \cup \mathcal{S}_{\rm{(vi)}}$ forms a finite set of height-one ideals of $R$. 
As explained at the beginning of the proof, we complete the proofs of assertions $\rm(1)$ and $\rm(2)$ in the case $d=2$. 
\\

Suppose that $d \ge 3$ and that the theorem has been proved in the case that $\dim R=d-1$. Then we establish the theorem in the case $\dim R=d$ by using an induction hypothesis that $(\ref{inclusionchar})$ holds in the case that $\dim R=d-1$, in which the local Bertini theorem (Theorem \ref{theorem:previous}) plays a crucial role. First of all, we set
\begin{align*}
& M:=\textcyr{Sh}^2_{\Sigma} (T^*(1)), \\
& N:=R/(P^k) \oplus  H^1 (G_{\Sigma},T^*(1))/R \mathbf{z}_1.
\end{align*}
Then we want to prove that $\Char_R(N) \subseteq \Char_R (M)$. To this aim, we consider scalar extensions $M_{W(\overline{\mathbb{F}})}$ and $M_{W(\overline{\mathbb{F}})}$ as
$R_{W(\overline{\mathbb{F}})}:=R \widehat{\otimes}_{W(\mathbb{F})} W(\overline{\mathbb{F}})$-modules. Notice that $R_{W(\overline{\mathbb{F}})}$ is a complete local Cohen-Macaulay normal domain with residue field $\overline{\mathbb{F}}$ in view of Lemma \ref{localring}. With the notation as in Theorem \ref{theorem:previous}, we set
$$
\widetilde{\mathcal{S}}:=\{(\mathbf{x}_{\widetilde{a}}) \subseteq R_{W(\overline{\mathbb{F}})}~|~a \in \theta_{W(\overline{\mathbb{F}})}(U_{M,N})\}.
$$
Let us take an element $\fp:=(\mathbf{x}_{\widetilde{a}}) \in \widetilde{\mathcal{S}}$. Then there exists a finite extension of fields $\mathbb{F} \to \mathbb{F}'$ such that $\fp \subseteq R_{W(\mathbb{F}')}$. By Theorem \ref{theorem:previous}, we want to prove that
\begin{equation}
\label{charcheck}
\Char_{R_{W(\mathbb{F}')}/\fp}(N_{W(\mathbb{F}')}/\fp N_{W(\mathbb{F}')}) \subseteq \Char_{R_{W(\mathbb{F}')}/\fp}(M_{W(\mathbb{F}')}/\fp M_{W(\mathbb{F}')}).
\end{equation}

It is easy to check that the conditions $\rm(i),\rm(ii),\rm(iii),\rm(iv),\rm(v)$ and $\rm(vi)$ are not effected by base change $T \to T_{W(\mathbb{F}')}$. After checking that the conditions $\rm(i),\rm(ii),\rm(iii),\rm(iv),\rm(v)$ and $\rm(vi)$ hold for the $R_{W(\mathbb{F}')}/\fp[G_{\mathbb{Q}}]$-module $T_{W(\mathbb{F}')}/\fp T_{W(\mathbb{F}')}$, we obtain the following by induction hypothesis:
\begin{equation}
\label{charcheck2}
(\overline{P}^k) \Char_{R_{W(\mathbb{F}')}/\fp}\big(H^1 (G_{\Sigma},(T^*_{W(\mathbb{F}')}/\fp T^*_{W(\mathbb{F}')})(1))/(R_{W(\mathbb{F}')}/\fp) \overline{\mathbf{z}}_1\big) 
\end{equation}
$$
\subseteq \Char_{R_{W(\mathbb{F}')}/\fp}\big(\textcyr{Sh}^2_{\Sigma}((T^*_{W(\mathbb{F}')}/\fp T^*_{W(\mathbb{F}')})(1))\big).
$$
By Lemma \ref{last} together with $(\ref{controlSha})$, we find that
\begin{equation}
\label{charcheck3}
\Char_{R_{W(\mathbb{F}')}/\fp}(M_{W(\mathbb{F}')}/\fp M_{W(\mathbb{F}')})=\Char_{R_{W(\mathbb{F}')}/\fp}\big(\textcyr{Sh}^2_{\Sigma}((T^*_{W(\mathbb{F}')}/\fp T^*_{W(\mathbb{F}')})(1))\big).
\end{equation}
Next, let us consider the long exact sequence:
\begin{multline*}
\cdots \to H^1 (G_{\Sigma},T^*_{W(\mathbb{F}')}(1)) \xrightarrow{\times \mathbf{x}_{\widetilde{a}}} H^1 (G_{\Sigma},T^*_{W(\mathbb{F}')}(1)) \to H^1 (G_{\Sigma},(T^*_{W(\mathbb{F}')}/\fp T^*_{W(\mathbb{F}')})(1))
\\ 
\to H^2(G_{\Sigma},T^*_{W(\mathbb{F}')}(1)) \xrightarrow{\times \mathbf{x}_{\widetilde{a}}} \cdots
\end{multline*}
which is induced by the exact sequence $0 \to T^*_{W(\mathbb{F}')} \xrightarrow{\times \mathbf{x}_{\widetilde{a}}} T^*_{W(\mathbb{F}')} \to T^*_{W(\mathbb{F}')}/\fp T^*_{W(\mathbb{F}')} \to 0$. The above long exact sequence yields the following short exact sequence:
\begin{multline}\label{equation:shortsequence1}
0 \to H^1 (G_{\Sigma},T^*_{W(\mathbb{F}')}(1))/\fp H^1 (G_{\Sigma},T^*_{W(\mathbb{F}')}(1)) \to H^1 (G_{\Sigma},(T^*_{W(\mathbb{F}')}/\fp T^*_{W(\mathbb{F}')})(1))
\\ 
\to H^2(G_{\Sigma},T^*_{W(\mathbb{F}')}(1))[\mathbf{x}_{\widetilde{a}}] \to 0.
\end{multline}
Let us choose $\mathbf{x}_{\widetilde{a}}$ such that
$$
\mathbf{x}_{\widetilde{a}} \notin \bigcup \fq,
$$
where $\fq$ ranges over all associated prime ideals of the $R_{W(\mathbb{F}')}$-module $H^2(G_{\Sigma},T^*_{W(\mathbb{F}')}(1))$ which are of height 
at most two. Notice that the set of such $\mathbf{x}_{\widetilde{a}}$ is non-empty by the assumption $d \ge 3$. 

For such $\mathbf{x}_{\widetilde{a}}$, we have the following equality by combining \eqref{controlGal} and \eqref{equation:shortsequence1}: 
\begin{multline}
\label{charcheck4}
\Char_{R_{W(\mathbb{F}')}/\fp}(N_{W(\mathbb{F}')}/\fp N_{W(\mathbb{F}')})
\\ 
=(\overline{P}^k) \Char_{R_{W(\mathbb{F}')}/\fp}\big(H^1 (G_{\Sigma},(T^*_{W(\mathbb{F}')}/\fp T^*_{W(\mathbb{F}')})(1))/(R_{W(\mathbb{F}')}/\fp) \overline{\mathbf{z}}_1\big). 
\end{multline}

Thus, in view of $(\ref{charcheck2})$, $(\ref{charcheck3})$ and $(\ref{charcheck4})$, it suffices to define the subset $\mathcal{S}$ of $\widetilde{\mathcal{S}}$ as follows:
$$
\mathcal{S}:=\big\{(\mathbf{x}_{\widetilde{a}})~\big|~\mathbf{x}_{\widetilde{a}} \notin \fq~\mbox{for all}~\fq \in \Ass_{R_{W(\overline{\mathbb{F}})}}\big(H^2(G_{\Sigma},T^*_{W(\overline{\mathbb{F}})}(1))\big) \setminus\{\fm_{R_{W(\overline{\mathbb{F}})}}\}\big\} \cap \widetilde{\mathcal{S}}
$$
to prove $(\ref{charcheck})$. Let us define $U_{\fq} \subseteq \mathbb{P}^n(\overline{\mathbb{F}})$ to be the Zariski-dense open subset attached to each $\fq \in \Ass_{R_{W(\overline{\mathbb{F}})}}\big(H^2(G_{\Sigma},T^*_{W(\overline{\mathbb{F}})}(1))\big) \setminus\{\fm_{R_{W(\overline{\mathbb{F}})}}\}$ via Lemma \ref{Case1}. Now we apply Theorem \ref{theorem:previous} to the Zariski open subset:
$$
U:=\big(\bigcap U_{\fq}\big) \cap U_{M,N} \subseteq \mathbb{P}^n(\overline{\mathbb{F}})
$$
for $\fq \in \Ass_{R_{W(\overline{\mathbb{F}})}}\big(H^2(G_{\Sigma},T^*_{W(\overline{\mathbb{F}})}(1))\big) \setminus\{\fm_{R_{W(\overline{\mathbb{F}})}}\}$.

As the proof of the assertion $\rm(1)$ is similar to the case $d=2$, we skip the details. So let us prove the assertion $\rm(2)$. Write $R$ for $R_{W(\mathbb{F}')}$ to simplify the notation for a finite field extension $\mathbb{F} \to \mathbb{F}'$. As we did in the case $d=2$, we need to determine the subsets $\mathcal{S}_{\rm{(i)}}, \mathcal{S}_{\rm{(ii)}}, \mathcal{S}_{\rm{(iii)}}, \mathcal{S}_{\rm{(iv)}}, \mathcal{S}_{\rm{(v)}}$ and $\mathcal{S}_{\rm{(vi)}}$ of $\mathcal{S}$, respectively. Notice that the conditions $\rm(i)$ and $\rm(v)$ in the theorem hold for the $R/\fp[G_{\mathbb{Q}}]$-module $T/\fp T$ for any $\fp \in \mathcal{S}$ and we have $\mathcal{S}_{\rm{(i)}}=\mathcal{S}_{\rm{(v)}}=\emptyset$. Moreover, the conditions $\rm(iii)$ and $\rm(vi)$ are treated in a similar manner to the case $d=2$ and we find that $\mathcal{S}_{\rm{(iii)}}=\emptyset$ 
and $\mathcal{S}_{\rm{(vi)}}$ is finite. So it remains to discuss the conditions $\rm(ii)$ and $\rm(iv)$ for the $R/\fp[G_{\mathbb{Q}}]$-module $T/\fp T$.

$\mathbf{Finiteness~of}$ $\mathcal{S}_{\rm{(ii)}}$: Write $R \xrightarrow{\times \mathbf{x}_{\widetilde{a}}} R$ for the multiplication map with $(\mathbf{x}_{\widetilde{a}})=\fp$. Then the short exact sequence $0 \to R \xrightarrow{\times \mathbf{x}_{\widetilde{a}}} R \to R/\fp \to 0$ yields the long exact sequence:
$$
\cdots \to H^1(G_{\Sigma},T^*(1)) \xrightarrow{\times \mathbf{x}_{\widetilde{a}}} H^1(G_{\Sigma},T^*(1)) \to H^1(G_{\Sigma},(T^*/\fp T^*)(1)) \to H^2(G_{\Sigma},T^*(1)) \xrightarrow{\times \mathbf{x}_{\widetilde{a}}} \cdots .
$$
Now we have to find $\fp=(\mathbf{x}_{\widetilde{a}}) \subseteq R$ such that
\begin{equation}
\label{torsionmodule}
H^1(G_{\Sigma},(T^*/\fp T^*)(1))/(R/\fp) \mathbf{z}_1~\mbox{is a torsion}~R/\fp \mbox{-module}.
\end{equation}
As the proof of $(\ref{torsionmodule})$ is done similarly to the case $d=2$, we skip the details.

$\mathbf{Finiteness~of}$ $\mathcal{S}_{\rm{(iv)}}$: By hypothesis, there exists an element $u \in R^{\times}$ of infinite order, which is in the image of the determinant character $\wedge^2 \rho: G_{\mathbb{Q}} \to R^{\times}$. Since the ring of Witt vectors $W(\mathbb{F})$ is a coefficient ring of $R$, we have $R=W(\mathbb{F})+\fm$ and $u=a+t$ for some $a \in W(\mathbb{F})^{\times}$ and $t \in \fm$. We shall make a choice of a principal prime ideal $\fp$ such that the following condition holds:
\begin{equation}
\label{character}
\fp~\mbox{has the property that}~u^n-1 \notin \fp~\mbox{for all}~n \ge 1.
\end{equation}
If $a \ne 1$, then there is nothing to prove due to $u^n-1\in R^{\times}$. So we assume that $a=1$ in what follows. Then
$$
u^n-1=(1+t)^n-1=\sum_{k=1}^n{n \choose k}t^k=t \big(n+\frac{n(n-1)}{2}t+\cdots+t^{n-1}\big).
$$
Writing $n=mp^e$ with $e \ge 0$ and $p \nmid m$, we see that $\dfrac{u^m-1}{t} \in R^{\times}$ since $t \in \fm$. 
Thus, there are only finitely many height-one primes in $R$ containing $u^m-1$ for varying $m$. Consider an ideal of $R$:
$$
I_{m,e}:=(p,u^{mp^e}-1).
$$
Since $u^{mp^e}-1 \equiv (u^m-1)^{p^e} \pmod{pR}$ and $\dfrac{u^m-1}{t} \in R^{\times}$, it follows that
\begin{equation}
\label{character2}
\bigcup_{m, e \ge 1}\Min_R(I_{m,e})~\mbox{is a finite set}.
\end{equation}
Every prime ideal $\fq$ belonging to $\bigcup_{m, e \ge 1}\Min_R(I_{m,e})$ is of height at most 2. Moreover, since $d=\dim R \ge 3$, such a prime ideal has the property that $\fq \ne \fm$. Then Lemma \ref{Case1} allows us to find sufficiently many prime ideals $\fp \in \Spec R$ so as to satisfy $(\ref{character})$. In other words, the reduced character $\wedge^2 \overline{\rho}:G_{\mathbb{Q}} \to (R/\fp)^{\times}$ is of infinite order.

Now we get the desired $(\ref{charcheck})$ for all height-one prime ideals $\fp \in \mathcal{S} \setminus \mathcal{S}_{\rm{bad}}$. Hence, $\Char_R(N) \subseteq \Char_R(M)$, that is, $(\ref{inclusionchar})$ has been established in the case $d=\dim R$. This completes the proof of the theorem.
\end{proof}

\section{Modular forms and Hida theory}
\label{Hidafamily}

In this section, we give a brief review on modular form, Hecke algebra and their extension in the context of Hida theory.

\subsection{$p$-optimal complex period via modular symbols}

In this section, we construct a $p$-adic $L$-function associated to a normalized $p$-ordinary Hecke eigencusp form which interpolates the special values of the complex $L$-function. In the next section, we will give the construction of two-variable $p$-adic $L$-function attached to a branch of the nearly ordinary Hecke algebra. This was first constructed by Mazur and Kitagawa \cite{Kit}. Another constructions were given by Greenberg and Stevens \cite{GS93} and Emerton, Pollack and Weston \cite{EPW}.

Let us introduce \textit{p-optimal complex period} and \textit{p-adic period} in order to have the algebraicity of complex $L$-functions. As the complex period of Deligne type \cite{De79} is determined only up to multiple of $\overline{\mathbb{Q}}^{\times}$, it is necessary to introduce the $p$-optimal complex period to formulate the interpolation formula for the Hida deformation of a $p$-adic $L$-function without ambiguities. Fix a prime number $p>0$ and embeddings 
$i_{\infty}:\overline{\mathbb{Q}} \hookrightarrow \mathbb{C}$ and $i_p:\overline{\mathbb{Q}} \hookrightarrow \overline{\mathbb{Q}}_p$. 

\begin{definition}
Let $k \ge 1$ be an integer and let $p>0$ be a prime number. Fix an embedding $i_p:\overline{\mathbb{Q}} \hookrightarrow \overline{\mathbb{Q}}_p$.

\begin{enumerate}
\item
Let $M>0$ be an integer such that $p \mid M$. Then we say that a cusp form $f \in S_k(\Gamma_1(M);\overline{\mathbb{Q}})$ is \textit{p-stablized}, if it is a normalized Hecke eigencusp form.

\item
Let $N>0$ is an integer such that $p \nmid N$. Let $f \in S_k(\Gamma_1(N);\overline{\mathbb{Q}})$ be a normalized Hecke eigencusp form and let $\alpha,\beta \in \overline{\mathbb{Q}}_p$ be the roots of the equation $X^2-a_p(f)X+p^{k-1}\psi_f(p)=0$, where $\psi_f$ is the neben character of $f$. Let
$$
f_{\alpha}:=f(q)-\beta f(q^p),~f_{\beta}:=f(q)-\alpha f(q^p).
$$
which are shown to be cusp forms in the space $S_k(\Gamma_1(Np),\overline{\mathbb{Q}})$. We call $f_{\alpha}$ and $f_{\beta}$ the \textit{p-stabilization} of $f$.

\item
We say that $ f \in S_k(\Gamma_1(M),\overline{\mathbb{Q}})$ is \textit{p-ordinary}, if $i_p(a_p(f)) \in \overline{\mathbb{Q}}_p$ is a $p$-adic unit via $i_p:\overline{\mathbb{Q}} \hookrightarrow \overline{\mathbb{Q}}_p$. A normalized $p$-ordinary Hecke eigencusp form $f \in S_k(\Gamma_1(M),\overline{\mathbb{Q}})$ is called a \textit{p-stabilized newform}, if it comes from the $p$-stabilization of a certain primitive form. 
\end{enumerate}
\end{definition}

\begin{remark}
\begin{enumerate}
\item
Let $f \in S_k(\Gamma_1(M),\overline{\mathbb{Q}})$ be a normalized $p$-stabilized Hecke eigencusp form Then it is proved that $f \in S_k(\Gamma_1(Mp),\overline{\mathbb{Q}})$ is $p$-stabilized.

\item
It is necessary to raise the power of $p$ in the ``level'' for constructing Hida deformations and it is essential to consider $p$-ordinary $p$-stabilized newforms.
\end{enumerate}
\end{remark}

Let $f \in S_k(\Gamma_1(M),\overline{\mathbb{Q}})$ be a $p$-ordinary $p$-stabilized Hecke eigencusp form of weight $k \ge 2$; in particular, we have $p|M$. Denote by $\mathbb{Q}_f:=\mathbb{Q}(\{a_n(f)\})$ the \textit{Hecke field} associated to $f$. Set $\mathbb{Z}_{f,(p)}:=i^{-1}_p(\overline{\mathbb{Z}}_p)$. This is a discrete valuation ring with the field of fractions $\mathbb{Q}_f$. For an integer $k \ge 2$ and a ring $A$, we define an $A$-module
$$
L_k(A):=\bigoplus_{0 \le i \le k-2} A \cdot X^i Y^{k-2-i},
$$
where $A \cdot X^iY^{k-2-i}$ denotes the $A$-submodule of the polynomial ring $A[X,Y]$ generated by $X^iY^{k-2-i}$. We let the matrix
$$
g=
\begin{pmatrix}
a & b \\ c & d 
\end{pmatrix}
\in M_2(\mathbb{Z}) \cap GL_2(\mathbb{Q})
$$
act on $L_k(A)$ by the formula:
$$
g \cdot p(X,Y)=p\Big((X,Y)\begin{pmatrix} a & c \\ b & d \end{pmatrix}\Big).
$$
for $p(X,Y) \in L_k(A)$.

Let $Y_1(M)_{\mathbb{C}}:=\Gamma_1(M) \setminus \mathfrak{H}$ be the complex affine modular curve. We let $\Gamma_1(M)$ act on the product $\mathfrak{H} \times L_k(A)$ diagonally. We define $\mathcal{L}_k(M)$ as a sheaf of continuous sections of the topological covering map:
$$
\Gamma_1(M) \setminus \mathfrak{H} \times L_k(A) \twoheadrightarrow \Gamma_1(M) \setminus \mathfrak{H}=Y_1(M)_{\mathbb{C}}.
$$
Let $H^1_{\mathrm{Betti}}(Y_1(M)_{\mathbb{C}},\mathcal{L}_k(A))$ be the sheaf cohomology using complex topology on $Y_1(M)_{\mathbb{C}}$. 

Let $\mathcal{O}$ be the ring of $p$-adic integers with $\mathcal{K}$ its field of fractions. Then one can define the etale sheaf $\mathcal{L}_k(A)_{\et}$ on $Y_1(M)_{\overline{\mathbb{Q}}}$ when $A=\mathcal{O}$ or $\mathcal{K}$ and the etale cohomology $H^1_{\et}(Y_1(M)_{\overline{\mathbb{Q}}},\mathcal{L}_k(A)_{\et})$. We remark that both Betti and etale cohomologies groups are equipped with Hecke actions via Hecke correspondences.  We take $A=\mathcal{O}$ or $\mathcal{K}$. By the comparison theorem between etale and Betti cohomologies, one obtains an isomorphism:
\begin{equation}
\label{Betti}
H^1_{\Betti}(Y_1(M)_{\mathbb{C}},\mathcal{L}_k(A)) \simeq H^1_{\et}(Y_1(M)_{\overline{\mathbb{Q}}},\mathcal{L}_k(A)_{\et})~\mbox{via}~i_{\infty}:\overline{\mathbb{Q}} \hookrightarrow \mathbb{C}.
\end{equation}
Moreover, the isomorphism $(\ref{Betti})$ preserves Hecke actions on both sides. Deligne constructed the Galois representation attached to a cusp form by using etale cohomology, which will be discussed later.

\begin{definition}
We define \textit{parabolic cohomology} by
$$
H^1_{\Bettip}(Y_1(M)_{\mathbb{C}},\mathcal{L}_k(A)):=\im \Big(H^1_{\Betti,c}(Y_1(M)_{\mathbb{C}},\mathcal{L}_k(A)) \to H^1_{\Betti}(Y_1(M)_{\mathbb{C}},\mathcal{L}_k(A))\Big), 
$$
where $H^1_{\Betti,c}(Y_1(M)_{\mathbb{C}},\mathcal{L}_k(A))$ denotes the compactly supported cohomology.
\end{definition}

The importance of parabolic cohomology is expressed by the following theorem (see \cite[Theorem 1 in \S~6.2]{Hid2} or \cite[Theorem 8.4]{Sh71} for the proof).

\begin{theorem}[Eichler-Shimura isomorphism]
\label{ES}
Fix integers $k \ge 2$ and $M\ge 1$. Then there is a natural isomorphism of $\mathbb{C}$-vector spaces:
$$
\ES:S_k(\Gamma_1(M)) \oplus \overline{S_k(\Gamma_1(M))} \to H^1_{\Bettip}(Y_1(M)_{\mathbb{C}},\mathcal{L}_k(\mathbb{C})).
$$
Moreover, the map $\ES$ is compatible with the action of Hecke operators. Here, $\overline{S_k(\Gamma_1(M))}$ is the $\mathbb{C}$-vector space spanned by the complex conjugates $\overline{f(z)}$ for all $f(z) \in S_k(\Gamma_1(M))$.
\end{theorem}

Let $\mathcal{D}$ be the free abelian group generated by the set of cusps $\mathbb{P}^1(\mathbb{Q})$. Let $\mathcal{D}_0$ denote the subgroup of $D$ generated by the elements of the form $\{\alpha\}-\{\beta\}$ for $\{\alpha\}, \{\beta\} \in \mathbb{P}^1(\mathbb{Q})$. The group $\Gamma_1(M)$ acts on $\mathfrak{H} \cup \mathbb{P}^1(\mathbb{Q})$ via linear fractional transformations. Via this action, we may regard $\mathcal{D}$ and $\mathcal{D}_0$ as being equipped with $\Gamma_1(M)$-action.

\begin{proposition}
\label{modsy}
Let $A$ be any commutative ring. Then there is an isomorphism of $A$-modules:
$$
\Hom_{\Gamma_1(M)}(\mathcal{D}_0;L_k(A)) \simeq H^1_{\Betti,c}(Y_1(M)_{\mathbb{C}},\mathcal{L}(A)).
$$
This isomorphism is compatible with the Hecke action on both sides.
\end{proposition}

Let $A$ be a ring such that $\frac{1}{2} \in A$. As the $2 \times 2$-matrix $\begin{pmatrix} 1 & 0 \\ 0 & -1 \end{pmatrix} \in \Gamma_1(M)$ induces an involution on any $A[\Gamma_1(M)]$-module, we have the decomposition:
$$
\Hom_{\Gamma_1(M)}(\mathcal{D}_0;L_k(A))=\Hom_{\Gamma_1(M)}(\mathcal{D}_0;L_k(A))^+ \oplus \Hom_{\Gamma_1(M)}(\mathcal{D}_0;L_k(A))^-
$$

\begin{definition}[Modular symbols]
Let $A$ be a commutative ring. Then the $A$-module
$$
\Hom_{\Gamma_1(M)}(\mathcal{D}_0;L_k(A))
$$
is called the \textit{space of modular symbols of weight k} with values in $A$. Let $f=\sum_{i=1}^{\infty} a_n(f) q^n \in S_k(\Gamma_1(M),\overline{\mathbb{Q}})$ be a normalized Hecke eigencusp form of weight $k \ge 2$ and assume that $A$ is a $\mathbb{Z}_{f,(p)}[\frac{1}{2}]$-algebra. We define
$$
\Hom_{\Gamma_1(M)}(\mathcal{D}_0;L_k(A))[f]^{\pm}:=\Big\{h \in \Hom_{\Gamma_1(M)}(\mathcal{D}_0,L_k(A))^{\pm}~\Big|~h|T(n)=a_n(f)h~\mbox{for all}~n \ge 0\Big\},
$$
where $T(n)$ is the Hecke operator. Then
$$
MS_f(A):=\Hom_{\Gamma_1(M)}(\mathcal{D}_0,L_k(A))[f]
$$
is called the \textit{space of modular symbols attached to f}. One can take the $\pm$-decomposition $MS_f^{\pm}(A).$
\end{definition}

\begin{example}
\label{specialcycle}
We give a special cocycle in the space of modular symbols which is associated to a cusp form as follows. Let $f \in S_k(\Gamma_1(M),\overline{\mathbb{Q}})$ be a cusp form of weight $k \ge 2$. We define a map $\eta_f:\mathcal{D}_0 \to L_k(\mathbb{C})$ by
$$
\eta_f(\{\alpha\}-\{\beta\})=\int_{\beta}^{\alpha} f(z)(zX+Y)^{k-2}dz. 
$$
This is called the \textit{modular symbol associated to f}. It is easy to verify that $\eta_f \in MS_f(\mathbb{C})$. Indeed, Eichler-Shimura map is related by the formula $\ES(f,0)=\eta_f$.
\end{example}

Let $f \in S_k(\Gamma_1(M),\overline{\mathbb{Q}})$ be a $p$-ordinary $p$-stabilized newform of weight $k \ge 2$. Let $\mathbb{Z}^{\wedge}_{f,(p)}$ be the completion of $\mathbb{Z}_{f,(p)}$ with respect to the $p$-adic valuation induced by $i_p:\overline{\mathbb{Q}} \hookrightarrow \overline{\mathbb{Q}}_p$. Denote by $\mathbb{Q}^{\wedge}_{f,(p)}$ the field of fractions of $\mathbb{Z}^{\wedge}_{f,(p)}$. By Deligne's theorem, we can attach the $G_{\mathbb{Q}}$-representation $V_f:=H^1_{\et,c}(Y_1(M)_{\overline{\mathbb{Q}}},\mathcal{L}_k(\mathbb{Q}^{\wedge}_{f,(p)})_{\et})[f]$. As is well-known, $V_f$ is a $\mathbb{Q}^{\wedge}_{f,(p)}$-vector space of dimension two. Then one can check that $MS_f^+(\mathbb{Z}_{f,(p)})$ and 
$MS_f^-(\mathbb{Z}_{f,(p)})$ are free $\mathbb{Z}_{f,(p)}$-modules of rank one, as each spans the sub $\mathbb{Q}_{f,(p)}^{\wedge}$-vector space of dimension one inside $V_f$ under the composition:
$$
MS_f(\mathbb{Z}_{f,(p)}) \to
H^1_{\Betti,c}(Y_1(M)_{\mathbb{C}},\mathcal{L}_k(\mathbb{Q}_{f,(p)}^{\wedge}))[f] \simeq V_f=H^1_{\et,c}(Y_1(M)_{\overline{\mathbb{Q}}},\mathcal{L}_k(\mathbb{Q}_{f,(p)}^{\wedge}))[f],
$$
where the first map is due to Proposition \ref{modsy} and the second isomorphism is due to the compactly supported version of $(\ref{Betti})$.

\begin{definition}[$p$-optimal complex period]
Fix a $\mathbb{Z}_{f,(p)}$-basis $b^+_{T_f}$ of the module $MS_f^+(\mathbb{Z}_{f,(p)})$ (resp. $\mathbb{Z}_{f,(p)}$-basis $b^-_{T_f}$ of the module $MS_f^-(\mathbb{Z}_{f,(p)})$). Via an isomorphism $MS_f^{\pm}(\mathbb{Z}_{f,(p)}) \otimes_{\mathbb{Z}_{f,(p)}} \mathbb{C} \simeq MS_f^{\pm}(\mathbb{C})$, one may regard $\{b^{\pm}_{T_f}\}$ as a $\mathbb{C}$-basis of $MS_f^{\pm}(\mathbb{C})$. Let $\eta_f=\eta^+_f+\eta^-_f$ be the sum in the decomposition 
$MS_f(\mathbb{C})=MS_f^+(\mathbb{C}) \oplus MS_f^-(\mathbb{C})$. We define a \textit{p-optimal complex period} $\Omega_{\infty}(f,b^{\pm}_{T_f}) \in \mathbb{C}^{\times}$ to be a constant satisfying the equality 
$$
\eta^{\pm}_f=\Omega_{\infty}(f,b^{\pm}_{T_f}) \cdot b^{\pm}_{T_f},
$$
where $\eta^{\pm}_f$ is the modular symbol associated to $f$ as in Example \ref{specialcycle}.
\\
\end{definition}

\subsection{Nearly ordinary Hida deformation and Galois representation}

We fix notation used in Hida theory (see \cite{Hid0} and \cite{Hid1} for details). For a fixed prime number $p>2$ (we will assume this for the moment), let $\mathbb{Q}_{\infty}/\mathbb{Q}$ be the unique cyclotomic $\mathbb{Z}_p$-extension and let $C_{\infty}$ be the Galois group of $\mathbb{Q}_{\infty}/\mathbb{Q}$. Then there is a canonical character $\chi_{\cyc}:C_{\infty} \simeq 1+p\mathbb{Z}_p \subseteq \mathbb{Z}_p^{\times}$ (the $p$-adic cyclotomic character). Let $N>0$ be an integer such that $p \nmid N$. Let $Y_1(Np^t)_{\mathbb{C}}:=\Gamma_1(Np^t) \setminus \mathfrak{H}$. Then, one can consider an algebraic curve $Y_1(Np^t)$ as defined over $\mathbb{Q}$ associated to $Y_1(Np^t)_{\mathbb{C}}$. Then for 
$d \in (\mathbb{Z}/Np^t\mathbb{Z})^{\times}$ with $d \equiv 1 \pmod p$, the diamond operator $\langle d \rangle$ maps $[E,P]$ to $[E,dP]$. The group of diamond operators forms an inverse system compatible with the inverse system $\{Y_1(Np^t)\}_{t \ge 1}$. We denote the inverse limit by $D_{\infty}$. Then there is a canonical character $\eta:D_{\infty} \simeq 1+p\mathbb{Z}_p \subseteq \mathbb{Z}_p^{\times}$.

Fix an integer $k \ge 2$ and let $\mathcal{O}$ be the ring of integers of a finite extension of $\mathbb{Q}_p$. Let $h_k(\Gamma_{1}(Np^r);\mathcal{O})$ be the Hecke algebra acting on $S_k(\Gamma_1(Np^r);\mathcal{O})$. We define \textit{ordinary Hecke algebra} $h^{\ord}_k(\Gamma_{1}(Np^r);\mathcal{O})$ as the maximal algebra direct summand of $h_k(\Gamma_{1}(Np^r);\mathcal{O})$ such that $h^{\ord}_k(\Gamma_{1}(Np^r);\mathcal{O})$ acts on the space of ordinary forms: $S_k^{\ord}(\Gamma_1(Np^r);\mathcal{O}):=e\big(S^{\ord}_k(\Gamma_1(Np^r);\mathcal{O})\big) \subseteq S_k(\Gamma_1(Np^r);\mathcal{O})$,
where $e$ is Hida's idempotent operator. 

We define the \textit{universal ordinary (Hida's) Hecke algebra} with tame level $N$:
$$
\mathbf{H}^{\ord}_{Np^{\infty}}:=\varprojlim_r h^{\ord}_k(\Gamma_{1}(Np^r);\mathcal{O})
$$
which acts on the $\Lambda$-module:
$$
S_k^{\ord}(\Gamma_1(Np^{\infty});\mathcal{O}):=\bigcup_{r \ge 1} S_k^{\ord}(\Gamma_1(Np^r);\mathcal{O}).
$$
Then we can check that $\mathbf{H}^{\ord}_{Np^{\infty}}$ is regarded as an algebra over $\Lambda:=\mathbb{Z}_p[[D_{\infty}]]$ via the action of $D_{\infty}$. More explicitly, $\mathbf{H}^{\ord}_{Np^{\infty}}$ is a sub-algebra of $\End_{\Lambda}\big(S_k^{\ord}(\Gamma_1(Np^{\infty});\mathcal{O})\big)$ generated by Hecke operators $T_{\ell}$; $\ell \nmid Np$ and $U_{\ell}$; $\ell | Np$. It is a fundamental theorem of Hida theory that the algebra $\mathbf{H}^{\ord}_{Np^{\infty}}$ does not depend on the choice of $k \ge 2$ (see \cite[Theorem 1.1]{Hid1}). Moreover, $\mathbf{H}^{\ord}_{Np^{\infty}}$ is commutative and a semi-local ring which is finite flat over $\Lambda$ and is characterized by the algebra isomorphism:
$$
\mathbf{H}^{\ord}_{Np^{\infty}} \otimes_{\Lambda} \Lambda/(P^r_k) \simeq h^{\ord}_k(\Gamma_{1}(Np^r);\mathcal{O}_{r,k})
$$
for $\mathcal{O}_{r,k}:=\Lambda/(P^r_k)$ and $P^r_k:=u^{p^{r-1}}-\eta(u)^{(k-2)p^{r-1}}$, where $u \in D_{\infty}$ is the fixed topological generator. The above isomorphism is often called the \textit{control theorem for Hida deformations} and it is proved in \cite[Theorem 1.2]{Hid1}. 

We say that $\kappa \in \Hom_{\mathbb{Z}_p}(\mathbf{H}^{\ord}_{Np^{\infty}},\overline{\mathbb{Q}}_p)$ is an \textit{arithmetic character of weight $w(\kappa)$}, if there exists an open subgroup $\Gamma \subseteq D_{\infty}$ for which the restriction of $\kappa$ to the subring $\mathbb{Z}_p[[\Gamma]] \subseteq \mathbb{Z}_p[D_{\infty}]$ coincides with the ring map induced by the character $\eta^{w(\kappa)-2}$ for an integer $w(\kappa) \ge 2$. One may also define the notion of arithmetic character for any $\mathbb{Z}_p[[D_{\infty}]]$-algebra. The \textit{nearly ordinary Hecke algebra} is defined as the completed group algebra:
$$
\mathbf{H}^{\nord}_{Np^{\infty}}:=\mathbf{H}^{\ord}_{Np^{\infty}} \widehat{\otimes}_{\mathbb{Z}_p} \mathbb{Z}_p[[C_{\infty}]],
$$ 
which is a finite flat extension of $\mathbb{Z}_p[[C_{\infty}\times D_{\infty}]]$. The completed group algebra $\mathbb{Z}_p[[C_{\infty}\times D_{\infty}]]$ is non-canonically isomorphic to the power series ring $\mathbb{Z}_p[[X,Y]]$. In fact, an isomorphism is given by a fixed pair of topological generators of $C_{\infty}$ and $D_{\infty}$, respectively. We say that the localization $(\mathbf{H}^{\ord}_{Np^{\infty}})_{\fm}$ for a maximal ideal $\fm \subseteq \mathbf{H}^{\ord}_{Np^{\infty}}$ is a \textit{local component} of $\mathbf{H}^{\ord}_{Np^{\infty}}$.

We let $\mathbf{I}$ denote the integral closure of $\mathbb{Z}_p[[D_{\infty}]]$ in a finite field extension of $\Frac(\mathbb{Z}_p[[D_{\infty}]])$. Let $\mathbf{f}=\sum_{n=1}^{\infty}a_n(\mathbf{f})q^n \in \mathbf{I}[[q]]$ be a $p$-ordinary $p$-stabilized $\mathbf{I}$-adic newform with level $Np^{\infty}$. This means that the specialization $\mathbf{f}_{\kappa}$ is a $p$-ordinary $p$-stabilized cusp newform for all but finitely many arithmetic characters $\kappa \in \Hom_{\mathbb{Z}_p}(\mathbf{I},\overline{\mathbb{Q}}_p)$. We can define the ring map $\mathbf{H}^{\ord}_{Np^{\infty}} \to \mathbf{I}$ by sending $T_{\ell}$ to $a_{\ell}(\mathbf{f})$. Since $\mathbf{I}$ is a domain, this map factors through the quotient $\mathbf{I}^{\ord}_{\mathbf{f}}:=\mathbf{H}^{\ord}_{Np^{\infty}}/\fa_{\mathbf{f}}$ for a minimal prime $\fa_{\mathbf{f}} \subseteq \mathbf{H}^{\ord}_{Np^{\infty}}$. We call $\mathbf{I}^{\ord}_{\mathbf{f}}$ a \textit{branch} of the Hecke algebra $\mathbf{H}^{\ord}_{Np^{\infty}}$. Then $\mathbf{I}^{\ord}_{\mathbf{f}}$ is a local domain which is a finite extension of $\mathbb{Z}_p[[D_{\infty}]]$ and there is a continuous Galois representation attached by Hida:
$$
\rho_{\mathbf{f}}^{\ord}:G_{\mathbb{Q}} \to \Aut_{\mathbf{I}^{\ord}_{\mathbf{f}}}(\mathcal{T}^{\ord}),
$$ 
unramified outside a finite set of primes $\Sigma$, which consists of $\{\infty\}$ and all prime factors of $Np$. The $\mathbf{I}^{\ord}_{\mathbf{f}}$-module $\mathcal{T}^{\ord}$ is defined as a reflexive module
$$
\mathcal{T}^{\ord}:=\big(\varprojlim_n H^1_{\et,c}(Y_1(Np^n)_{\overline{\mathbb{Q}}};\mathbb{Z}^{\wedge}_{\mathbf{f}_{\kappa},(p)})^{\ord}[\mathbf{f}_{\kappa}]\big)^{\rc},
$$
where $\mathbf{f}_{\kappa}$ is any weight two specialization of $\mathbf{f}$ and $\varprojlim_n H^1_{\et,c}(Y_1(Np^n)_{\overline{\mathbb{Q}}};\mathbb{Z}^{\wedge}_{\mathbf{f}_{\kappa},(p)})^{\ord}$ is the inverse limit of the ordinary part of compactly supported etale cohomology groups of affine modular curves:
$$
\cdots \to Y_1(Np^{n+1})_{\overline{\mathbb{Q}}} \to Y_1(Np^n)_{\overline{\mathbb{Q}}} \to \cdots \to Y_1(Np)_{\overline{\mathbb{Q}}},
$$
where the transition map of the cohomology is given by the trace map. The representation $\rho_{\mathbf{f}}^{\ord}$ is called \textit{Hida deformation}.

Under the above setting, there is a well-defined notion of the \textit{residual representation} of $\rho_{\mathbf{f}}^{\ord}$. The original account of residual representations is \cite{MazWi}. A more detailed account is found in \cite{BookIwasawa} and Carayol's article in \cite{Kit}. Let us recall the following facts about $\mathcal{T}^{\ord}$, whose proof was originally given in \cite[Theorem 2.1]{Hid1}. Refer also to \cite[Theorem 2.1.1, Theorem 2.1.3 and Theorem 2.2.1]{EPW} for the review of these facts:
\begin{enumerate}
\item[\rm{(i)}]
$\rho_{\mathbf{f}}^{\ord}$ is unramified outside a finite set $\Sigma$ as above.

\item[\rm{(ii)}]
For a geometric Frobenius $\Frob_{\ell} \in G_{\mathbb{Q}}$ with $\ell \notin \Sigma$, we have $\Tr\rho_{\mathbf{f}}^{\ord}(\Frob_{\ell}))=T_{\ell}$, where $T_{\ell}$ is the Hecke operator in $\mathbf{I}^{\ord}_{\mathbf{f}}$.

\item[\rm{(iii)}]
Let us define $\mathcal{V}^{\ord} :=\mathcal{T}^{\ord} \otimes_{\mathbf{I}^{\ord}_{\mathbf{f}}}\Frac(\mathbf{I}^{\ord}_{\mathbf{f}})$ with $\Frac(\mathbf{I}^{\ord}_{\mathbf{f}})$ being the field of fractions of 
$\mathbf{I}^{\ord}_{\mathbf{f}}$. 
Then, the restriction $\rho_{\mathbf{f}}^{\ord}|_{G_{\mathbb{Q}_p}}$ admits the following $G_{\mathbb{Q}_p}$-invariant filtration:
$$
0 \to F^+\mathcal{V}^{\ord} \to \mathcal{V}^{\ord} \to \mathcal{V}^{\ord}/F^+\mathcal{V}^{\ord} \to 0
$$
in which both $F^+\mathcal{V}^{\ord} $ and $\mathcal{V}^{\ord} /F^+\mathcal{V}^{\ord}$ are $\Frac(\mathbf{I}^{\ord}_{\mathbf{f}})$-modules of rank one and $G_{\mathbb{Q}_p}$ acts on $\mathcal{V}^{\ord}/F^+\mathcal{V}^{\ord}$ via $\alpha$, where $\alpha:G_{\mathbb{Q}_p} \to (\mathbf{I}^{\mathrm{ord}}_{\mathbf{f}})^{\times}$ is an unramified character such that $\alpha(\Frob_p)=U_p$.
\end{enumerate}

Let $\widetilde{\chi}_{\cyc}$ be the universal cyclotomic character $G_{\mathbb{Q}} \twoheadrightarrow C_\infty \hookrightarrow \mathbb{Z}_p [[C_\infty ]]$ induced by $\chi_{\cyc}$. We denote by $\rho_{\mathbf{f}}^{\nord}$ the continuous $G_{\mathbb{Q}}$-representation $\rho_{\mathbf{f}}^{\ord} \widehat{\otimes}\widetilde{\chi}_{\cyc}$ and denote the representation space of $\rho_{\mathbf{f}}^{\ord} \widehat{\otimes}\widetilde{\chi}_{\cyc}$ by $\mathcal{T}^{\nord}$. This is a finitely generated module over $\mathbf{I}^{\nord}_{\mathbf{f}}=\mathbf{I}^{\ord}_{\mathbf{f}} \widehat{\otimes}_{\mathbb{Z}_p} \mathbb{Z}_p[[C_{\infty}]]$. The representation
$$
\rho_{\mathbf{f}}^{\nord}:G_{\mathbb{Q}} \to \Aut_{\mathbf{I}^{\nord}_{\mathbf{f}}}(\mathcal{T}^{\nord})
$$
is called \textit{nearly ordinary deformation} attached to the Hida family $\mathbf{f}$. We also call it \textit{(nearly ordinary) Hida deformation}.

We will consider the following conditions, which are not a major obstacle in practice:

\begin{enumerate}
\item[($\bf{NOR}$):]
$\mathbf{I}^{\ord}_{\mathbf{f}}$ is a normal domain.

\item[($\bf{IRR}$):]
The residual representation associated with $\rho_{\mathbf{f}}^{\ord}$ is absolutely irreducible.

\item[($\bf{FIL}$):] 
The restriction $\rho_{\mathbf{f}}^{\ord}|_{G_{\mathbb{Q}_p}}$ admits a $G_{\mathbb{Q}_p}$-invariant filtration:
$$
0 \to F^+\mathcal{T}^{\ord}  \to \mathcal{T}^{\ord} \to \mathcal{T}^{\ord}/F^+\mathcal{T}^{\ord}   \to 0,
$$
where $F^+\mathcal{T}^{\ord} $ and $\mathcal{T}^{\ord} /F^+\mathcal{T}^{\ord}$ are direct summands as $\mathbf{I}^{\ord}_{\mathbf{f}}$-modules, and the above sequence gives rise to the $G_{\mathbb{Q}_p}$-invariant filtration:
$$
0 \to F^+\mathcal{V}^{\ord}  \to \mathcal{V}^{\ord}  \to \mathcal{V}^{\ord} /F^+\mathcal{V}^{\ord}  \to 0,
$$
after taking the base extension to $\mathcal{V}^{\ord} $. 
\end{enumerate}

We note the following fact.

\begin{lemma}
Assume that the condition $(\bf{NOR})$ holds. Then $\mathbf{I}^{\nord}_{\mathbf{f}}$ is a three-dimensional local Cohen-Macaulay normal domain.
\end{lemma}

\begin{proof}
Since $\mathbf{I}^{\ord}_{\mathbf{f}}$ is a two-dimensional normal local domain, it is Cohen-Macaulay by Serre's normality criterion. Now we observe that $\mathbf{I}^{\nord}_{\mathbf{f}}=\mathbf{I}^{\ord}_{\mathbf{f}}[[C_{\infty}]]$ is a three-dimensional local Cohen-Macaulay normal domain. 
\end{proof}

Let us assume the condition $(\bf{IRR})$. Then the $\mathbf{I}^{\ord}_{\mathbf{f}}$-lattice $\mathcal{T}^{\ord}$ is uniquely determined and it is a free module of rank two. Thus, it defines a representation $\rho_{\mathbf{f}}^{\ord}:G_{\mathbb{Q}} \to GL_2(\mathbf{I}^{\ord}_{\mathbf{f}})$. Moreover, since $\mathbf{I}^{\ord}_{\mathbf{f}}$ is a local domain, there is a unique maximal ideal $\fm_{\mathbf{f}}$ of $\mathbf{H}^{\ord}_{Np^{\infty}}$ such that $\fa_{\mathbf{f}} \subseteq \fm_{\mathbf{f}}$. Then the condition $(\bf{IRR})$ implies that the local component $(\mathbf{H}^{\ord}_{Np^{\infty}})_{\fm_{\mathbf{f}}}$ is Gorenstein. These facts together with the proofs are found in \cite[Proposition 2 at page 254]{MazWi}. There is a sufficient condition that assures the local filtration to consist of rank one free modules when $\mathcal{T}^{\ord}$ is free of rank two.

\begin{enumerate}
\item[($\bf{DIST}$):]
The semi-simplification of the $G_{\mathbb{Q}_p}$-module $\mathcal{T}^{\ord}/\fm_{\mathbf{I}^{\ord}_{\mathbf{f}}}\mathcal{T}^{\ord}$ is the sum of two distinct characters of $G_{\mathbb{Q}_p}$ 
with values in $(\mathbf{I}^{\ord}_{\mathbf{f}}/\fm_{\mathbf{I}^{\ord}_{\mathbf{f}}})^\times$, where $\fm_{\mathbf{I}^{\ord}_{\mathbf{f}}}$ is the maximal ideal of 
$\mathbf{I}^{\ord}_{\mathbf{f}}$.
\end{enumerate}

It is known that, if the condition $(\bf{DIST})$ holds and $\mathcal{T}^{\ord}$ is free of rank two, then both $F^+\mathcal{T}^{\ord}$ and $F^-\mathcal{T}^{\ord}$ are free $\mathbf{I}^{\ord}_{\mathbf{f}}$-modules of rank one. Especially, we have the implication: $(\bf{IRR})+(\bf{DIST}) \Rightarrow (\bf{FIL})$. In the rest of the present article, let us denote $\mathcal{T}^{\nord}$ by $\mathcal{T}$ for simplicity.\\

\subsection{Two-variable $p$-adic $L$-function over Hida deformations}

Recall that Kitagawa constructed $\Lambda$-adic modular symbols in \cite{Kit} in order to construct a two-variable $p$-adic $L$-function.

\begin{theorem}[Kitagawa]
\label{MazurKitagawa}
Assume that the condition $(\bf{IRR})$ holds. Then there exist free $\mathbf{I}^{\ord}_{\mathbf{f}}$-modules $\mathbf{MS}^+_{\mathbf{f}}$ and $\mathbf{MS}^-_{\mathbf{f}}$ of rank one such that for any arithmetic character $\kappa$ of $\mathbf{I}^{\ord}_{\mathbf{f}}$ with $w(\kappa) \ge 2$, the specialization $\mathbf{MS}^{\pm}_{\mathbf{f}} \otimes_{\mathbf{I}^{\ord}_{\mathbf{f}}} \kappa(\mathbf{I}^{\ord}_{\mathbf{f}})$ is isomorphic as $\mathbb{Z}^{\wedge}_{\mathbf{f}_{\kappa},(p)}$-module, to the intersection:
$$
i \big(MS_{\mathbf{f}_{\kappa}}(\mathbb{Z}^{\wedge}_{\mathbf{f}_{\kappa},(p)})\big) \cap H^1_{\Betti,c}(Y_1(M_{\kappa})_{\mathbb{C}},\mathcal{L}_{w(\kappa)}(\mathbb{Q}^{\wedge}_{\mathbf{f}_{\kappa},(p)}))[\mathbf{f}_{\kappa}]^{\pm},
$$
where $M_{\kappa}$ denotes the level of the cusp form $\mathbf{f}_{\kappa}$ and the composition map $i:MS_{\mathbf{f}_{\kappa}}(\mathbb{Z}^{\wedge}_{\mathbf{f}_{\kappa},(p)}) \simeq H^1_{\Betti,c}(Y_1(M_{\kappa})_{\mathbb{C}},\mathcal{L}_{w(\kappa)}(\mathbb{Z}^{\wedge}_{\mathbf{f}_{\kappa},(p)})) \to H^1_{\Betti,c}(Y_1(M_{\kappa})_{\mathbb{C}},\mathcal{L}_{w(\kappa)}(\mathbb{Q}^{\wedge}_{\mathbf{f}_{\kappa},(p)}))$ is the natural map induced by $\mathbb{Z}^{\wedge}_{\mathbf{f}_{\kappa},(p)} \to \mathbb{Q}^{\wedge}_{\mathbf{f}_{\kappa},(p)}$.
\end{theorem}

The modules $\mathbf{MS}^{\pm}_{\mathbf{f}}$ are called the \textit{spaces of $\mathbf{I}^{\ord}_{\mathbf{f}}$-adic modular symbols}. As we noted earlier, $(\bf{IRR})$ ensures that the local ring $(\mathbf{H}_{Np^{\infty}}^{\ord})_{\fm_{\mathbf{f}}}$ is Gorenstein, where $\fm_{\mathbf{f}}$ is the unique maximal ideal of $\mathbf{H}_{Np^{\infty}}^{\ord}$ such that $\fa_{\mathbf{f}} \subseteq \fm_{\mathbf{f}}$. The proof of Theorem \ref{MazurKitagawa} is found in \cite[\S~5.5]{Kit}.

\begin{definition}[$p$-adic period]
Assume that the condition $(\bf{IRR})$ holds. Fix a basis  $\mathbf{B}^{+}_{\mathbf{f}}$ of the $\mathbf{I}^{\ord}_{\mathbf{f}}$-module $\mathbf{MS}^{+}_{\mathbf{f}}$ (resp.
$\mathbf{B}^{-}_{\mathbf{f}}$ of the $\mathbf{I}^{\ord}_{\mathbf{f}}$-module $\mathbf{MS}^-_{\mathbf{f}}$). Fix a basis $b_{\mathcal{T}^{\ord}_{\kappa}}^+$ of the $\mathbb{Z}_{\mathbf{f}_{\kappa},(p)}$-module $MS^+_{\mathbf{f}_{\kappa}}(\mathbb{Z}_{\mathbf{f}_{\kappa},(p)})$ (resp. $b_{\mathcal{T}^{\ord}_{\kappa}}^-$ of the $\mathbb{Z}_{\mathbf{f}_{\kappa},(p)}$-module $MS^-_{\mathbf{f}_{\kappa}}(\mathbb{Z}_{\mathbf{f}_{\kappa},(p)})$). We define the \textit{p-adic period} $\Omega_p(\mathbf{B}_{\mathbf{f}}^{\pm},
b^{\pm}_{\mathcal{T}^{\ord}_{\kappa}}) \in \overline{\mathbb{Z}}_p$ to be a constant satisfying the equality: 
$$
\kappa(\mathbf{B}^{\pm}_{\mathbf{f}})=
\Omega_p(\mathbf{B}_{\mathbf{f}}^{\pm},
b^{\pm}_{\mathcal{T}^{\ord}_{\kappa}}) \cdot b^{\pm}_{\mathcal{T}^{\ord}_{\kappa}}.
$$
\end{definition}

We have the following theorem (see \cite[Theorem 1.1]{Kit}).

\begin{theorem}[Kitagawa]
Assume that the conditions $(\bf{NOR})$ and $(\bf{IRR})$ hold and let $\xi=(\chi^j_{\cyc}\phi,\kappa)$ be an arithmetic character such that $\phi:C_{\infty} \to \overline{\mathbb{Q}}^{\times}$ is a finite order Dirichlet character of $p$-power conductor and $1 \le j \le w(\kappa)+1$ (criticality condition). Then there exists an element $L_p(\{\mathbf{B}^{\pm}_{\mathbf{f}}\}) \in \mathbf{I}^{\nord}_{\mathbf{f}}$ which satisfies the following interpolation property over an $\mathbf{I}^{\ord}_{\mathbf{f}}$-adic family of cusp forms:
\begin{equation}
\label{interpolation}
\frac{\xi(L_p(\{\mathbf{B}^{\pm}_{\mathbf{f}}\}))}{\Omega_p(\mathbf{B}_{\mathbf{f}}^{\pm},
b_{\mathcal{T}^{\ord}_{\kappa}}^{(-1)^{j-1}\phi(-1)})}=(-1)^j (j-1)! \times 
\Eul(\mathbf{f}_{\kappa,}j,\phi) \times
G(\phi) \times \frac{L(\mathbf{f}_{\kappa},\phi^{-1},j)}{(2\pi\sqrt{-1})^j \Omega_{\infty}(\mathbf{f}_{\kappa},b_{\mathcal{T}^{\ord}_{\kappa}}^{(-1)^{j-1}\phi(-1)})},
\end{equation}
where
$$
\Eul(\mathbf{f}_{\kappa},j,\phi)=\begin{cases}
\Big(1-\frac{p^{j-1}}{a_p(\mathbf{f}_{\kappa})}\Big), &  \phi=\mathbf{I} \\
\Big(\frac{p^{j-1}}{a_p(\mathbf{f}_{\kappa})}\Big)^{\ord_p(\phi)}, & \phi \ne \mathbf{I} 
\end{cases}
$$
and $\Omega_p(\mathbf{B}_{\mathbf{f}}^{\pm},
b_{\mathcal{T}^{\ord}_{\kappa}}^{(-1)^{j-1}\phi(-1)}) \in \overline{\mathbb{Z}}_p$ is a $p$-adic period and $\Omega_{\infty}(\mathbf{f}_{\kappa},b_{\mathcal{T}^{\ord}_{\kappa}}^{(-1)^{j-1}\phi(-1)}) \in \mathbb{C}$ is a complex period attached to $\mathbf{f}_{\kappa}$ and $b_{\mathcal{T}^{\ord}_{\kappa}}^{\pm}$, determined by the Eichler-Shimura map. $G(\phi)$ is the Gauss sum of the character $\phi$. Finally, the symbol $\xi(L_p(\{\mathbf{B}^{\pm}_{\mathbf{f}}\})) \in \overline{\mathbb{Q}}_p$ denotes the specialization of $L_p(\{\mathbf{B}^{\pm}_{\mathbf{f}}\})$ at the arithmetic point $\xi=(\chi_{\cyc}^j\phi,\kappa)$. 
\end{theorem}

\begin{remark}
%
By Drinfeld-Manin theorem, 
it is known that we have
$$
\frac{L(\mathbf{f}_{\kappa},\phi^{-1},j)}{(2\pi\sqrt{-1})^j \Omega_{\infty}(\mathbf{f}_{\kappa},b_{\mathcal{T}^{\ord}_{\kappa}}^{(-1)^{j-1}\phi(-1)})}
\in \overline{\mathbb{Q}}.
$$
Hence, the right-hand side of $(\ref{interpolation})$ is regarded as a $p$-adic number through the fixed embedding $i_p:\overline{\mathbb{Q}} \hookrightarrow \overline{\mathbb{Q}}_p$ (see also \cite{De79}).\\
\end{remark}

\section{Algebraic invariants of Hida deformations}
\label{GaloisDeformation}

In the present section, we first review Greenberg's Selmer group attached to the nearly ordinary (Hida) deformation based on the article \cite{Gr1} and prove some necessary results. After that, we recall the definition of Beilinson-Kato Euler system based on \cite{Kato} and introduce a $\Lambda$-adic family of these elements.\\

\subsection{Greenberg Selmer group}
Let $\mathcal{T}$ be the nearly ordinary deformation space attached to an $\mathbf{I}^{\ord}_{\mathbf{f}}$-adic newform $\mathbf{f}$ as introduced in \S~\ref{Hidafamily}. Take the discrete module $\mathcal{D}^{\nord}:=\mathcal{T} \otimes_{\mathbf{I}^{\nord}_{\mathbf{f}}}(\mathbf{I}^{\nord}_{\mathbf{f}})^{\PD}$ viewed as a Galois module via $g \in G_{\mathbb{Q}} \mapsto g(m \otimes \sigma)=g(m) \otimes \sigma$. Then the filtration ($\bf{FIL}$) in \S~\ref{Hidafamily} induces a filtration: $0 \to F^+\mathcal{D}^{\nord}  \to \mathcal{D}^{\nord} \to F^-\mathcal{D}^{\nord}  \to 0$. 
The \textit{Greenberg Selmer group} is defined as
$$
\Sel_{\mathbb{Q}}(\mathcal{T})=\ker\big(H^1(G_{\Sigma},\mathcal{D}^{\nord}) \to 
H^1(I_{\mathbb{Q}_p},\mathcal{D}^{\nord}/ F^+ \mathcal{D}^{\nord}) \times 
\prod_{\ell \nmid p } H^1(I_{\mathbb{Q}_{\ell}},\mathcal{D}^{\nord})\big),
$$
where $H^1(G_{\Sigma},\mathcal{D}^{\nord}) \to H^1(G_{\mathbb{Q}_{\ell}},\mathcal{D}^{\nord})$ is the restriction map. 

We define the \textit{localization map for the nearly ordinary deformations} as the natural composite maps:
$$
\mathbf{Loc}_{/f}: H^1(G_{\Sigma},\mathcal{T}^*(1)) \to H^1(G_{\mathbb{Q}_p},\mathcal{T}^*(1)) \to H^1_{/f}(G_{\mathbb{Q}_p},\mathcal{T}^*(1)),
$$
where $H^1_{/f}(G_{\mathbb{Q}_p},\mathcal{T}^*(1))$ is the \textit{singular part} of the local Galois cohomology as defined in \cite[Lemma 3.10]{Oc}. Specializing to $\mathbf{f}_{\kappa}$ for an arithmetic character $\kappa \in \Hom_{\mathbb{Z}_p}(\mathbf{I}^{\ord}_{\mathbf{f}},\overline{\mathbb{Q}}_p)$ with $w(\kappa) \ge 2$, let $T_{\mathbf{f}_{\kappa}}$ be a $\mathbb{Z}_{\mathbf{f}_{\kappa},(p)}^{\wedge}$-lattice of the $G_{\mathbb{Q}}$-representation attached to $\mathbf{f}_{\kappa}$. Let $\phi:C_{\infty} \to \overline{\mathbb{Q}}^{\times}$ be a finite order Dirichlet character of $p$-power conductor and let $1 \le j \le w(\kappa)+1$ (criticality condition). Then we define the \textit{localization map} as the
composite maps:
$$
\loc_{/f}:H^1(G_{\Sigma},T^*_{\mathbf{f}_{\kappa}}(1-j) \otimes \phi^{-1}) \to H^1(G_{\mathbb{Q}_p},T^*_{\mathbf{f}_{\kappa}}(1-j) \otimes \phi^{-1}) \to H^1_{/f}(G_{\mathbb{Q}_p},T^*_{\mathbf{f}_{\kappa}}(1-j) \otimes \phi^{-1}).
$$
Finally, the \textit{Bloch-Kato's dual exponential map} is defined by
$$
\exp^*:H^1_{/f}(G_{\mathbb{Q}_p},V^*_{\mathbf{f}_{\kappa}}(1-j) \otimes \phi^{-1}) \to \Fil^0 D_{\dR}(V^*_{\mathbf{f}_{\kappa}}(1-j) \otimes \phi^{-1}),
$$
where $V_{\mathbf{f}_{\kappa}}:=T_{\mathbf{f}_{\kappa}} \otimes_{\mathbb{Z}_p} \mathbb{Q}_p$.

\subsection{$p$-adic $L$-function and Beilinson-Kato Euler system}

In \cite{Oc} and \cite{Oc2}, we gave another construction of a two-variable $p$-adic $L$-function attached to $\mathcal{T}$ based on Mazur-Kitagawa $L$-function \cite{Kit}. Recall that Beilinson-Kato Euler system, which is constructed and studied extensively in \cite{Kato}, is adapted into the Hida deformations and it thus defines a $\Lambda$-adic version of Euler system.

We are ready to formulate the main content of Beilinson-Kato elements for nearly ordinary Hida deformations.

\begin{theorem}
\label{Coleman}
Assume that the conditions $(\bf{NOR})$ and $(\bf{IRR})$ hold for the nearly ordinary deformation space $\mathcal{T}$ attached to a Hida family $\mathbf{f}$. Fix an $\mathbf{I}^{\ord}_{\mathbf{f}}$-basis $\{\mathbf{B}^{\pm}_{\mathbf{f}}\}$ of the $\mathbf{I}^{\ord}_{\mathbf{f}}$-modules $\mathbf{MS}^{\pm}_{\mathbf{f}}$. Let us put
$$
\mathfrak{R}:=\{r \in \mathbb{Z}_{\ge 1}~|~(r,2p)=1~\mbox{and}~r~\mbox{is a square-free integer}\}.
$$
Let $G_{\Sigma,r}$ be the Galois group of $\mathbb{Q}(\mu_r)_{\Sigma}/\mathbb{Q}(\mu_r)$ with $\Sigma=\{p,\infty\}$. Then there exists a collection of elements:
$$
\big\{\mathbf{z}^{\BK}_r(\{\mathbf{B}^{\pm}_{\mathbf{f}}\}) \in H^1(G_{\Sigma,r},\mathcal{T}^*(1))\big\}_{r \in \mathfrak{R}}
$$
which defines Euler system for $(\mathcal{T},\mathbf{I}^{\nord}_{\mathbf{f}},p)$ such that the following statements hold:

\begin{enumerate}
\item[\rm(i)]
For $ r \ell\in  \mathfrak{R}$ and the inertial group $I_{\ell} \subseteq G_{\mathbb{Q}}$, we have the following equality
$$
\Cor_{\mathbb{Q}(\mu_{r\ell})/\mathbb{Q}(\mu_{r})}(\mathbf{z}^{\BK}_{r \ell}(\{\mathbf{B}^{\pm}_{\mathbf{f}}\}))=P(\Frob_{\ell};\mathcal{T})\mathbf{z}^{\BK}_{r}(\{\mathbf{B}^{\pm}_{\mathbf{f}}\}),
$$
where $P(\Frob_{\ell};\mathcal{T})=\det(1-\Frob_{\ell}X;\mathcal{T}^*(1)^{I_{\ell}})$ is the characteristic polynomial for $\Frob_{\ell}$.

\item[\rm (ii)]
Let us put
$$
\mathbf{z}^{\BK}(\{\mathbf{B}^{\pm}_{\mathbf{f}}\}):=
\mathbf{z}^{\BK}_1(\{\mathbf{B}^{\pm}_{\mathbf{f}}\}) \in H^1(G_{\Sigma,1},\mathcal{T}^*(1)).
$$
Then for any arithmetic character $\xi:=(\chi_{\cyc}^j \phi, \kappa) \in \Hom_{\mathbb{Z}_p}(\mathbf{I}^{\nord}_{\mathbf{f}},\overline{\mathbb{Q}}_p)$ such that $1 \le j \le w(\kappa)+1$ and $\phi$ is a finite order Dirichlet character of $p$-power conductor, the element
$$
\loc_{/f} \circ \xi(\mathbf{z}^{\BK}_{r \ell}(\{\mathbf{B}^{\pm}_{\mathbf{f}}\})) \in H^1_{/f}(G_{\mathbb{Q}_p}, V^*_{\mathbf{f}_{\kappa}}(1-j) \otimes \phi^{-1})
$$
satisfies the following interpolation formula
$$
\exp^*\big(\loc_{/f} \circ \xi(\mathbf{z}^{\BK}(\{\mathbf{B}^{\pm}_{\mathbf{f}}\}))\big)=
\Omega_p(\mathbf{B}_{\mathbf{f}}^{\sgn(j,\phi)},
b_{\mathcal{T}^{\ord}_{\kappa}}^{\sgn(j,\phi)}) \times
\frac{L(\mathbf{f}_{\kappa},\phi^{-1},j)}{(2 \pi \sqrt{-1})^j \Omega_{\infty}(\mathbf{f}_{\kappa},b_{\mathcal{T}^{\ord}_{\kappa}}^{\sgn(j,\kappa)})} \cdot \overline{\mathbf{f}_{\kappa}}
$$
which holds in $\Fil^0 D_{\dR}(V^*_{\mathbf{f}_{\kappa}}(1-j) \otimes \phi^{-1})$.

\item[\rm(iii)]
There is an injective $\mathbf{I}^{\nord}_{\mathbf{f}}$-module map (Coleman map):
$$
\Xi:H^1_{/f}(G_{\mathbb{Q}_p},\mathcal{T}^*(1)) \to \mathbf{I}^{\nord}_{\mathbf{f}}
$$
with pseudo-null cokernel and such that $L_p(\{\mathbf{B}^{\pm}_{\mathbf{f}}\}):=\Xi(\mathbf{Loc}_{/f}(\mathbf{z}^{\BK}(\{\mathbf{B}^{\pm}_{\mathbf{f}}\})))$ satisfies the following properties:

\begin{enumerate}
\item[\rm{(a)}]
$\Char_{\mathbf{I}^{\nord}_{\mathbf{f}}}\big(H^1_{/f}(G_{\mathbb{Q}_p},\mathcal{T}^*(1))/\mathbf{I}^{\nord}_{\mathbf{f}} \cdot \mathbf{Loc}_{/f}(\mathbf{z}^{\BK}(\{\mathbf{B}^{\pm}_{\mathbf{f}}\})\big)
=\big(L_p(\{\mathbf{B}^{\pm}_{\mathbf{f}}\})\big)$.
 
\item[\rm{(b)}]
$L_p(\{\mathbf{B}^{\pm}_{\mathbf{f}}\}) \in \mathbf{I}^{\nord}_{\mathbf{f}}$ is a two-variable $p$-adic $L$-function which satisfies the interpolation property as described in $(\ref{interpolation})$.
\end{enumerate}

\end{enumerate}
\end{theorem}

We refer the reader to \cite{Oc}, \cite{Oc1} and \cite{Oc2} for Beilinson-Kato elements and the construction of Coleman map for Hida deformations. Theorem \ref{Coleman} is the main content of \cite{Oc}, where the $p$-adic $L$-function is constructed via the dual exponential map. Various $p$-adic $L$-functions have been constructed by several people through different methods. We refer the reader to \cite{Oc2} for comparison results on these functions with their connection to the main conjecture.

\begin{lemma}
\label{deRham}
Under the notation as in Theorem \ref{Coleman}, we have the following assertions.

\begin{enumerate}
\item[\rm{(1)}]
There is an exact sequence of finitely generated $\mathbf{I}^{\nord}_{\mathbf{f}}$-modules $:$ 
$$
0 \to H^1(G_{\Sigma},\mathcal{T}^*(1))/\mathbf{I}^{\nord}_{\mathbf{f}} \cdot \mathbf{z}^{\BK}(\{\mathbf{B}^{\pm}_{\mathbf{f}}\}) \to H^1_{/f}(G_{\mathbb{Q}_p},\mathcal{T}^*(1))/\mathbf{I}^{\nord}_{\mathbf{f}} \cdot \mathbf{Loc}_{/f}(\mathbf{z}^{\BK}(\{\mathbf{B}^{\pm}_{\mathbf{f}}\})
$$
$$
\to \Sel_{\mathbb{Q}}(\mathcal{T})^{\PD} \to \textcyr{Sh}^2_{\Sigma}(\mathcal{T}^*(1)) \to 0.
$$

\item[\rm{(2)}]
$\Sel_{\mathbb{Q}}(\mathcal{T})^{\PD}$ is finitely generated and torsion over $\mathbf{I}^{\nord}_{\mathbf{f}}$.
\end{enumerate}
\end{lemma}

\begin{proof}
For the assertion $\rm(1)$, we have $(\mathcal{D}^{\nord})^{\PD} \simeq \mathcal{T}^*$ as $G_{\mathbb{Q}}$-modules. Since we know $\textcyr{Sh}^1_{\Sigma}(\mathcal{D}^{\nord}) \subseteq \Sel_{\mathbb{Q}}(\mathcal{T})$ and $\mathcal{D}^{\nord} \simeq M^{\PD}(1)$ for $M:=\mathcal{T}^*(1)$, the surjection on the most right side of the above sequence follows. Hence, the existence of the above exact sequence from the Poitou-Tate duality.
 
The assertion $\rm(2)$ is found in \cite[Proposition 4.9]{Oc2}.
\end{proof}

In order to prove one of the expected divisibilities in the Iwasawa main conjecture for the Hida deformations, we will need a deep result of Kato concerning the finiteness of Selmer groups \cite[Section 14]{Kato}. Using this result, it follows that
$$
H^2(G_{\Sigma},\mathcal{D}^{\nord})=0~(\mathrm{weak~Leopoldt~conjecture}).
$$
See \cite[Lemma 8.3]{Oc2} for the proof of this fact. For our purpose, it will be sufficient to have a weaker condition that $H^2(G_{\Sigma},\mathcal{D}^{\nord})$ is a finite group.\\

\subsection{Finiteness of local Galois cohomology groups}

As we have seen in Theorem \ref{Euler} and Theorem \ref{veryfinal}, the vanishing of $H^2(G_{\mathbb{Q}_{\ell}},T^*(1))$ plays a role in Euler system bound. Let us discuss its finiteness in terms of the Fourier expansion of a modular form. Let $f$ be a normalized eigen newform over $\overline{\mathbb{Q}}_p$, of weight $\ge 2$ and tame level $N$. Take $\mathcal{O}$ to be the ring of $p$-integers such that
$a_n(f) \in \mathcal{O}$ for all $n \ge 1$ with $f=\sum_{n=1}^{\infty}a_n(f)q^n$. Denote by $T$ an $\mathcal{O}$-lattice of the $p$-adic Galois representation attached to $f$. Now let $V:=T \otimes_{\mathcal{O}} K$ for the quotient field $K$ of $\mathcal{O}$. Denote by $T^*$ the $\mathcal{O}$-dual representation of $T$. We have the following result.

\begin{lemma}
\label{finite1}
Assume $\ell$ is a prime dividing $N$ and the notation is as above. Then $a_{\ell}(f) \ne 1$ if and only if $H^2(G_{\mathbb{Q}_{\ell}},T^*(1))$ is a finite group.
\end{lemma}

\begin{proof}
$H^2(G_{\mathbb{Q}_{\ell}},T^*(1))$ is a finitely generated $\mathcal{O}$-module and $H^2(G_{\mathbb{Q}_{\ell}},V^*(1)) \simeq H^2(G_{\mathbb{Q}_{\ell}},T^*(1)) \otimes_{\mathbb{Z}_p} \mathbb{Q}_p$. Therefore, the finiteness of $H^2(G_{\mathbb{Q}_{\ell}},T^*(1))$ is equivalent to the vanishing: $H^2(G_{\mathbb{Q}_{\ell}},V^*(1))=0$. Let $I_\ell$ be the inertia group at $\ell$. Then by local and Pontryagin dualities, it suffices to show that $(V^*)_{G_{\mathbb{Q}_{\ell}}}=0$. Note that $(V^*)_{G_{\mathbb{Q}_{\ell}}}$ is the quotient of $(V^*)_{I_{\ell}}$. As $f$ is a newform of tame level $N$, $V$ is ramified at all primes $\ell$ dividing $N$. Hence we have $\dim_K (V^*)_{I_{\ell}} \le 1$. 
If $(V^*)_{I_{\ell}}=0$, there is nothing to prove. So let us assume that
$(V^*)_{I_{\ell}}$ is one-dimensional. Then it defines a character $\varphi:G_{\mathbb{Q}_{\ell}} /I_{\ell} \to K^{\times}$ and thus, the eigenvalue of $\Frob_{\ell}$, acting on this space invertibly, is $a_{\ell}(f)$. Thus,
$$
(1-a_{\ell}(f))(V^*)_{I_{\ell}}=(V^*)_{I_{\ell}} \iff a_{\ell}(f) \ne 1,
$$
which yields that $(V^*)_{G_{\mathbb{Q}_{\ell}}}=0$. The lemma now follows.
\end{proof}

\section{Applications to the Iwasawa Main Conjecture}
\label{proof2}

The aim of this section is to prove the following theorem as an application of Theorem \ref{veryfinal}

\begin{theorem}
\label{MainTh}
Assume that $(\bf{NOR})$, $(\bf{IRR})$ and $(\bf{FIL})$ hold for the nearly ordinary Hida deformation $\mathcal{T}$ attached to a Hida family $\mathbf{f}$. Fix an $\mathbf{I}^{\nord}_{\mathbf{f}}$-basis $\mathbf{B}^{\pm}_{\mathbf{f}}$ of the modules of $\mathbf{I}^{\nord}_{\mathbf{f}}$-adic modular symbols $\mathbf{MS}^{\pm}_{\mathbf{f}}$. Assume further that

\begin{enumerate}
\item[\rm{(i)}]
There exists an element $\sigma_1 \in G_{\mathbb{Q}(\mu_{p^{\infty}})}$ such that $\rho_{\mathbf{f}}^{\nord}(\sigma_1) \simeq \begin{pmatrix} 1 & \epsilon \\ 0 & 1 \end{pmatrix} \in GL_2(\mathbf{I}^{\nord}_{\mathbf{f}})$ for a nonzero element $\epsilon \in \mathbf{I}^{\nord}_{\mathbf{f}}$.

\item[\rm{(ii)}]
There exists an element $\sigma_2 \in G_{\mathbb{Q}}$ such that $\sigma_2$ acts on $\mathcal{T}$ as multiplication by $-1$.

\item[\rm{(iii)}]
If $\ell$ is any prime dividing $Np$, the maximal Galois invariant quotient vanishes; $(\mathcal{T}^*)_{G_{\mathbb{Q}_{\ell}}}=0$.
\end{enumerate}
Let $k$ be the number of minimal $\mathbf{I}^{\nord}_{\mathbf{f}}$-generators of $\textcyr{Sh}^2_{\Sigma}(\mathcal{T}^*(1))$. Then we have
$$
\big(\epsilon^k L_p(\{\mathbf{B}^{\pm}_{\mathbf{f}}\})\big) \subseteq \Char_{\mathbf{I}^{\nord}_{\mathbf{f}}}\big(\Sel_{\mathbb{Q}}(\mathcal{T})^{\PD}\big).
$$
\end{theorem}

\begin{proof}
Let $\mathcal{D}^{\nord}$ denote the discrete module associated to $\mathcal{T}$. Then $\mathcal{D}^{\nord} \simeq (\mathcal{T}^*)^{\PD}$. By assumption, $\mathbf{I}^{\nord}_{\mathbf{f}}$ is a three-dimensional Cohen-Macaulay normal domain. Recall that the characteristic ideal is a reflexive ideal and is additive with respect to short exact sequences. Putting together Theorem \ref{Coleman} and Lemma \ref{deRham}, it follows that
$$
\big(\epsilon^k L_p(\{\mathbf{B}^{\pm}_{\mathbf{f}}\})\big) \subseteq \Char_{\mathbf{I}^{\nord}_{\mathbf{f}}}\big(\Sel_{\mathbb{Q}}^{\Sigma}(\mathcal{T})^{\PD}\big)
$$ 
if and only if
$$
\Char_{\mathbf{I}^{\nord}_{\mathbf{f}}}\big(\mathbf{I}^{\nord}_{\mathbf{f}}/(\epsilon^k) \oplus H^1(G_{\Sigma},\mathcal{T}^*(1))/\mathbf{I}^{\nord}_{\mathbf{f}} \mathbf{z}^{\BK}_1(\{\mathbf{B}^{\pm}_{\mathbf{f}}\})\big) \subseteq \Char_{\mathbf{I}^{\nord}_{\mathbf{f}}}\big(\textcyr{Sh}^2_{\Sigma}(\mathcal{T}^*(1))\big).
$$
Or equivalently,
$$
(\epsilon^k) \Char_{\mathbf{I}^{\nord}_{\mathbf{f}}}\big(H^1(G_{\Sigma},\mathcal{T}^*(1))/\mathbf{I}^{\nord}_{\mathbf{f}} \mathbf{z}^{\BK}_1(\{\mathbf{B}^{\pm}_{\mathbf{f}}\})\big) \subseteq \Char_{\mathbf{I}^{\nord}_{\mathbf{f}}}\big(\textcyr{Sh}^2_{\Sigma}(\mathcal{T}^*(1))\big).
$$
Note that all modules appearing inside the characteristic ideals are finitely generated torsion $\mathbf{I}^{\nord}_{\mathbf{f}}$-modules. Then the conclusion of the theorem follows from Theorem \ref{veryfinal}, as all the hypotheses are satisfied.
\end{proof}

We obtain the following corollary of the main theorem. In \cite{MazWi}, some examples of modular $p$-adic Galois representations containing $SL_2(\mathbb{Z}_p)$ are constructed.

\begin{corollary}
\label{MainCo}
Assume that $(\bf{NOR})$ and $(\bf{FIL})$ hold for the nearly ordinary Hida deformation $\mathcal{T}$ attached to a Hida family $\mathbf{f}$ and assume that the maximal Galois invariant quotient vanishes; $(\mathcal{T}^*)_{G_{\mathbb{Q}_{\ell}}}=0$ for every prime $\ell$ dividing $Np$. Fix an $\mathbf{I}^{\nord}_{\mathbf{f}}$-basis $\mathbf{B}^{\pm}_{\mathbf{f}}$ of the modules of $\mathbf{I}^{\nord}_{\mathbf{f}}$-adic modular symbols $\mathbf{MS}^{\pm}_{\mathbf{f}}$. If the image of the restriction map
$$
\rho_{\mathbf{f}}^{\ord}:G_{\mathbb{Q}(\mu_{p^{\infty}})} \to GL_2(\mathbf{I}^{\ord}_{\mathbf{f}})
$$
contains a conjugate of $SL_2(\mathbb{Z}_p[[D_{\infty}]])$, then we have
$$
\big(L_p(\{\mathbf{B}^{\pm}_{\mathbf{f}}\})\big) \subseteq \Char_{\mathbf{I}^{\nord}_{\mathbf{f}}}\big(\Sel_{\mathbb{Q}}(\mathcal{T})^{\PD}\big).
$$
\end{corollary}

\begin{proof}
It suffices to check that the condition $(\bf{IRR})$ holds. Since the image of $\rho_{\mathbf{f}}^{\ord}$ contains a conjugate of $SL_2(\mathbb{Z}_p[[D_{\infty}]])$ by assumption, the residual representation associated to $\rho_{\mathbf{f}}^{\ord}$ is absolutely irreducible.
\end{proof}

\begin{remark}
\begin{enumerate}
\item
If the isomorphism holds: $\mathbf{I}^{\nord}_{\mathbf{f}} \simeq \mathcal{O}[[X,Y]]$, then it is not necessary to assume that $(\mathcal{T}^*)_{G_{\mathbb{Q}_{\ell}}}=0$ for every prime $\ell$ dividing $Np$ in Corollary \ref{MainCo}, as it was already considered in \cite{Oc1}.

\item
It is an important problem to find an example of $\mathbf{I}^{\nord}_{\mathbf{f}}$ which is normal, but not regular. The authors do not know if there is such an example. It will be interesting to ask the same problem in the case of the nearly ordinary Hecke algebra attached to a $\Lambda$-adic family of Hilbert modular cusp forms as constructed in \cite{Hid11}. A formulation of the $\Lambda$-adic version of the main conjecture in the totally real case requires an elaborate analysis as given in \cite{Oc3}.  
\end{enumerate}
\end{remark}

\end{document}